\documentclass[10pt, reqno,amsmath,amsthm,amssymb,amscd]{amsart}
\usepackage{pifont}
\usepackage{amssymb}
\usepackage{turnstile}

\usepackage{extarrows}
\usepackage{mathrsfs,amssymb}
\usepackage{multicol}\multicolsep=0pt
\usepackage{pstricks,pst-node}
\usepackage[enableskew,vcentermath]{youngtab}

\usepackage[sort]{cite}
\usepackage{xcolor,graphicx}

\def\crulefill{\leavevmode\leaders\hrule height 1pt\hfill\kern 0pt}
\long\def\QUERY#1{%
\leavevmode\newline%
\noindent$\star\star\star$\thinspace\textsf{Comment/Query}\crulefill\newline%
   \space #1\newline\hbox to 120mm{\crulefill}$\star\star\star$\newline}
\newtheorem{Theorem}{Theorem}[section]%[chapter] theorem number will %continue
\newtheorem{Lemma}[Theorem]{Lemma}
\newtheorem{Cor}[Theorem]{Corollary}
\newtheorem{Prop}[Theorem]{Proposition}

 \theoremstyle{definition}
\newtheorem{example}[Theorem]{Example}

\newtheorem{Defn}[Theorem]{Definition}

\newtheorem{remark}[Theorem]{Remark}
\numberwithin{equation}{section}
\theoremstyle{definition}
\makeatletter
\def\enumerate{\begingroup\ifnum\@enumdepth>3\@toodeep\else
      \advance\@enumdepth\@ne
      \edef\@enumctr{enum\romannumeral\the\@enumdepth}%
      \topsep\z@\parskip\z@
      \list{\csname label\@enumctr\endcsname}
        {\@nmbrlisttrue\let\@listctr\@enumctr
         \parsep\z@\itemsep\z@\topsep\z@
         \setcounter{\@enumctr}{0}
         \def\makelabel##1{\hss\llap{\rm ##1}}
       }\fi}

\makeatother

\let\bar=\overline
\let\epsilon=\varepsilon
\def\({\big(}
\def\){\big)}

\def\C{\mathbb C}

\def\Z{\mathbb Z}

\def\0{\underline{0}}

\DeclareMathOperator{\End}{End}

 \DeclareMathOperator*{\Res}{Res}

\def\m{\mathfrak m}

\def\Std{\mathscr{T}^{std}}
\def\upd{\mathscr{T}^{ud}}
\def\m{\mathfrak m}
\def\n{\mathfrak n}
\def\s{\mathfrak s}

\def\t{\mathfrak t}
\def\u{\mathfrak u}
\def\v{\mathfrak v}
\def\w{\mathfrak w}

\def\Hom{\text{Hom}}
\def\Ind{\text{Ind}}
\def\Res{\text{Res}}
\def\U{\mathbf U}
\def\T{\textbf{t}}
\def\S{\textbf{s}}

{\catcode`\|=\active
  \gdef\set#1{\mathinner{\lbrace\,{\mathcode`\|"8000%
                                   \let|\midvert #1}\,\rbrace}}
  \gdef\seT#1{\mathinner{\Big\lbrace\,{\mathcode`\|"8000%
                                   \let|\midverT #1}\,\Big\rbrace}}
}
\def\midvert{\egroup\mid\bgroup}
\def\midverT{\egroup\,\Big|\,\bgroup}

\def\Set[#1]#2|#3|{\Big\{\ #2\ \Big| \
           \vcenter{\hsize #1mm\centering #3}\Big\}}

\begin{document}
\baselineskip15pt
\title{Discriminants of   quantized walled Brauer algebras}
\author{ Mei Si, Linliang Song }
\address{M.S.  School of Mathematical Science, Shanghai Jiao Tong University,  Shanghai, 200240, China}\email{simei@sjtu.edu.cn}
\address{L.S.  School of Mathematical Science, Tongji University,  Shanghai, 200092, China}\email{llsong@tongji.edu.cn}

\thanks{The authors are supported by NSFC  (grant No. 12071346) and the first author is supported by NSFC  (grant No.11971304).}

\begin{abstract} In this paper,  we  compute
Gram determinants associated to all cell modules  of quantized walled Brauer algebras $\mathscr B_{r, t}(\rho, q)$ over an arbitrary field $\kappa$. Suppose $e$ is the quantum characteristic of $q^2$. We  classify the  blocks of   $\mathscr B_{r, t}(\rho, q)$ when $e>\max\{r,t\}$ and $\rho^2=q^{2n}$, $n\in\mathbb Z$.  As an application, we give a criterion for a cell module of $\mathscr B_{r, t}(\rho, q)$ being equal to its simple head over $\kappa$.
\end{abstract}
%\keywords{Quantized walled Brauer algebras, decomposition numbers}
\sloppy \maketitle
\section{Introduction}
Quantized walled Brauer algebras  $\mathscr B_{r, t}(\rho, q)$ were introduced in \cite{KM} in order to study mixed tensor product of natural modules and its linear dual of    the quantum general linear  group $\U_q(\mathfrak{gl}_n)$  over the complex field $\mathbb C$.  Later on,  Leduc\cite{Le} introduced these associative  algebras with arbitrary  parameters $\rho$ and $q$ over a commutative ring $R$ containing the identity $1$.
 It was proved in \cite{Enyang2} that $\mathscr B_{r, t}(\rho,q)$ is cellular over $R$ in the sense of\cite{GL}. Using representation theory of cellular algebras in \cite{GL},  Rui and the last author~\cite{Rsong}  gave a criterion on the semisimplicity of $\mathscr B_{r, t}(\rho,q)$ over an arbitrary field $\kappa$.
  In the non-semisimple case, if $\rho^2\notin \{q^{2a}\mid a\in\mathbb Z\}$, they  proved that $\mathscr B_{r, t}(\rho,q)$ is Morita equivalent to direct product of the  Hecke algebras associated to certain symmetric groups. In the general case,  they compute  the  decomposition numbers   of  $\mathscr B_{r, t}(\rho,q)$ over $\mathbb C$ no matter whether $q$ is a root of unity or not in \cite{Rsong2}.

 In this paper, we first compute the  Gram determinants   associated to all cell modules of  $\mathscr B_{r, t}(\rho, q)$ over $\kappa$. Then we classify the blocks of $\mathscr B_{r, t}(\rho, q)$ over $\kappa$ when $e>\max\{r,t\}$ and $\rho^2= q^{2n}$ with some $ n\in\mathbb Z$, where $e$ is  the quantum characteristic of $q^2$.   This result is a $q$-analog of that in \cite{CVDM}. Finally, we reach the main goal of  this paper to find a necessary and sufficient  condition for  a cell module being equal to its simple head over an arbitrary field  $\kappa$.

Rui and the first author found a remarkable method to compute  the Gram determinants of Brauer algebras and  Birman-Murakami-Wenzl algebras respectively in \cite{RS,RS3}.
Their method depends  heavily  on some  recursive steps  (i.e., Propositions 4.8, 4.9 in \cite{RS} and Propositions 4.5, 4.9 in \cite{RS3}). However, the counterpart of these recursive steps  for $ \mathscr B_{r, t}(\rho, q)$ cannot be obtained by the similar methods as that in  \cite{RS,RS3}.
 In fact, we found it's necessary to adopt a completely different approach for $ \mathscr B_{r, t}(\rho, q)$ in detail.  Fortunately, we find a suitable way to get the required  recursive formulae on  the Gram determinants of $\mathscr B_{r, t}(\rho, q)$, especially in Proposition~\ref{uu}, which is very different from above mentioned propositions in \cite{RS,RS3}. Meanwhile, it is pointed out that the scalars $E_{\t,\t}(r)$ play an important role in computing the Gram determinants.  These scalars are  the structural    coefficients of $E_{r,1}$ acting on the orthogonal basis elements. Unlike  Brauer algebras, we do not know how to get an explicit formula of $E_{\t,\t}(r)$ by using some  formal power series (cf. Lemma~\ref{forofwu}). To overcome  this difficulty, we use the Schur-Weyl duality between $\U_q(\mathfrak{gl}_n)$ and $\mathscr B_{r, t}(\rho,q)$ together with the quantum dimensions of the irreducible modules of $\U_q(\mathfrak{gl}_n)$. This enables us to   give an  explicit formula of  $E_{\t,\t}(r)$ (cf. Proposition~\ref{f and E}).
For two obstructions  mentioned above, we are motivated to research the essential difference between \cite{RS,RS3} and $ \mathscr B_{r, t}(\rho, q)$, which is the main motivation for this paper.

Motivated by   \cite{CVDM}, to classify the blocks of $\mathscr B_{r, t}(\rho,q)$, we have to describe the sufficient condition so that there is a nonzero homomorphism between   two  cell modules (see Theorem~\ref{key}). Our result can be considered as   the $q$-analog
of \cite[Thoerem~6.2]{CVDM}. However, the structure in the current case is much more complicated.  By technical analysis, we achieve this by  reducing the proof in Theorem~\ref{key} to the case that the Young diagram  is a rectangle.

We organize this paper as follows. In section~2, we first recall some basic notations and facts for $\mathscr B_{r, t}(\rho,q)$. We next construct  the Jucys-Murphy basis of $\mathscr B_{r, t}(\rho,q)$. By standard construction on the orthogonal form of cellular algebras  in \cite{Ma2}, we get the orthogonal basis of $\mathscr B_{r, t}(\rho,q)$ in section~3.
The recursive formulae on Gram determinants associated to the cell modules of $\mathscr B_{r, t}(\rho,q)$ are obtained  in section~4.
 Section~5 is dedicated to classify  the blocks of $\mathscr B_{r, t}(\rho, q)$ over $\kappa$ under the assumption that $e>\max\{r,t\}$ and $\rho^2=q^{2n}$ with some $n\in\mathbb Z$. As an application, we give a criterion  for each  cell module of $\mathscr B_{r, t}(\rho,q)$ being equal to its simple head over an arbitrary field $\kappa$.
The last section consists of the proofs of Lemma~\ref{down} and Lemma~\ref{lemeaforeee}, which include tedious computations.

%$[\lambda]:$  The Young diagram of  $\lambda\in \Lambda^+(n)$. If $p$ is a node of $[\lambda]$ in the $i$th row and $j$th column, we also denote $p$ by $(i,j)$.  We also identify any partition with its Young diagram. We write $\lambda\subseteq \mu$ if $[\lambda]$ is contained in $[\mu]$ and denote the corresponding skew partition by $\mu/\lambda$, where $\lambda$ and $\mu$ are partitions. Similarly we have the notation $\lambda\subseteq \mu$ and $\mu/\lambda$ for two bi-partitions $\lambda$ and $\mu$.

\section{The Jucys--Murphy basis of the quantized walled Brauer algebras}
Let $R$ be the localization of $\mathbb Z[q^{\pm 1},
\rho^{\pm 1}]$ at $q-q^{-1}$, where $\rho$ and $q$ are two indeterminates. Denote  $\delta:=\frac{\rho-\rho^{-1}}{q-q^{-1}}$.
%Let $B_\kappa:=B\otimes_R\kappa$  for any free $R$-algebra $B$ (with finite rank).  We will leave out $\kappa$ in the subscripts.
%Let $[M:N]$ be  the multiplicity of the irreducible $B_\kappa$-module $N$ in any composition series of  $M$.

  \begin{Defn}\label{qwb}\cite{Le}
Suppose $r, t\in \mathbb Z^{>0}$. The \textit{quantized walled Brauer algebra}
${\mathscr{B}}_{r,t}(\rho, q)$ is the associative $R$-algebra with generators
$E, g_1, \ldots, g_{r-1}, g_1^*, \ldots, g_{t-1}^*$, subject to
 the following relations:
\begin{multicols}{2}
\begin{enumerate}
\item $(g_i-q) (g_i+q^{-1})=0$, $1\le i\le r-1$,
\item  $g_ig_j=g_jg_i$, if  $|i-j|>1$,
\item $g_ig_{i+1}g_i=g_{i+1}g_ig_{i+1}$, $1\le i<r-1$,
\item $(g^*_i-q) (g^*_i+q^{-1})=0$, $1\le i\le t-1$,
\item  $g^*_ig^*_j=g^*_jg^*_i$, if  $|i-j|>1$,
\item $g^*_ig^*_{i+1}g^*_i=g^*_{i+1}g^*_ig^*_{i+1}$, $1\le i<t-1$,
\item  $g_i g^*_j=g^*_j g_i$, if $1\le i\le r-1$, $1\le j\le t-1$,
\item $g_iE=Eg_i$, $g^*_jE=Eg^*_j$, if $i\neq r-1, j\neq 1$,
\item $Eg_{r-1}E =\rho E=Eg^*_1E$,
\item $E^2=\delta E$,
\item $E{g_{r-1}}^{-
1}g^*_{1} Eg_{r-1} = E{g_{r-1}}^{- 1}g^*_{1 } Eg^*_{1}$,\item
$g_{1}E{g_{ r-1}}^{ - 1}g^*_{1} E = g^*_{1} E{g_{r-1}}^{ -
1}g^*_{1}E$.
\end{enumerate}
\end{multicols}
\end{Defn}
Let $\mathfrak S_r$   be the symmetric group on $r$   symbols with $r\in\mathbb N$. Denote the simple transposition $(i,i+1)$ by $s_i$, $1\leq i\leq r-1$.
We will need another symmetric group  on $t$ symbols  simultaneously and denote it by $\mathfrak S^*_t$ with simple transposition $s^*_i$.
We allow $t=0$ and identify  $\mathscr B_{r, 0}(\rho,q)$ with the usual Hecke algebra $\mathscr H_r$ associated to $\mathfrak S_r$, which is the $R$-algebra generated by $g_1,\ldots,g_{r-1}$ with relations (a)--(c) in Definition~\ref{qwb}.   Similarly, $\mathscr B_{0, t}(\rho,q)$ is the usual Hecke algebra $\mathscr H_t$ associated to $\mathfrak S_t^*$. Let $\mathscr H_{r,t}:=\mathscr H_r\otimes \mathscr H_t$. We write
$g_w=g_{i_1}\ldots g_{i_k}$ for $w\in \mathfrak S_r$ if $s_{i_1}\cdots s_{i_k}$ is a reduced expression of $w$. Similarly, we have the notation $g_w^*$ for all $w\in \mathfrak S_t^*$. Moreover, for any $w=w_1w_2\in \mathfrak S_r\times \mathfrak S^*_t$ with $w_1\in\mathfrak S_r $ and $w_2\in\mathfrak S^*_t $, let $g_w:=g_{w_1}g^*_{w_2}\in \mathscr B_{r,t}(\rho,q)$.
For $1\le i,j\le r$, set $s_{i, j}=s_{i-1} s_{i-1, j}$ if $i>j$ and $s_{i, i}=1$. When $i<j$, we set $s_{i, j}=s_i s_{i+1, j}$.
Similarly, we have $s_{i, j}^*$'s, etc. To simplify the notation, let  $g_{i,j}:=g_{s_{i,j}}$ and $g^*_{i,j}:=g^*_{s_{i,j}^*}$.
For all $1\le i\le r$ and $1\le j\le t$, define
 \begin{equation}\label{eij} E_{ i, j}= g_{r, i}^{-1} g_{j, 1}^*  E_{r, 1}g_{r, i} g_{j, 1}^{*-1}, \text{ and \  $\bar E_{i,j}= g_{r, i}^{-1}  g_{j, 1}^*  E_{r, 1}g_{i, r}^{-1} g_{1, j}^*$,}\end{equation}
 where $E_{r, 1}=E$.
The following result can be checked directly from  Definition~\ref{qwb}.
\begin{Lemma}\label{three isom}
\begin{enumerate}
\item There is an $R$-algebra automorphism $\phi$ of $ \mathscr B_{r,t}(\rho,q) $
which fixes $g_j^*$'s, $1\leq j\leq t-1$,  and sends $E$ and   $g_i$ to $E_{1,1}$and $g_{r-i}$ for  $1\leq i\leq r-1$, respectively.
\item There is an $R$-algebra isomorphism $\varphi: \mathscr B_{r,t}(\rho,q)\to\mathscr B_{t,r}(\rho,q)$
which   sends $E_{r,t}$, $g_i$ and $g_j^*$ to $E_{t,r}$, $g^*_{i}$ and $g_{j}$  for $1\leq i\leq r-1$, $1\leq j\leq t-1$, respectively.
\item There is a $\mathbb Z$-algebra involution $\pi$ of $\mathscr B_{r,t}(\rho,q)$
which fixes $E$, $\rho$, $g_i$'s and $g_j^*$'s and sends $q$ to $-q^{-1}$.
\item There is an $R$-algebra anti-involution $\sigma$ of $\mathscr B_{r, t}(\rho, q)$,    fixing all generators in Definition~\ref{qwb}.
\item
There is an $R$-algebra automorphism  $\Phi$ of $\mathscr{B}_{r,t}(\rho, q)$ which fixes $g_i$'s, $1\leq i\leq r-1$, and sends $E$ and  $g_j^*$ to $E_{r,t}$ and $g_{t-j}^*$ for $1\leq j\leq t-1$, respectively.
\end{enumerate}
\end{Lemma}

\begin{remark} The definition of $\mathscr B_{r,t}(\rho, q)$ in  Definition~\ref{qwb} is different from that in \cite{Le} and  \cite{Rsong}. By Lemma~\ref{three isom}(a),  they are isomorphic and the required isomorphism  $\gamma$ sends the generators $E$, $g_i$ and $g_j^*$ to the generators $e_1$, $g_{r-i}$ and $g_j^*$ in \cite{Rsong}, respectively.  We will freely  use the results in \cite{Rsong} by the isomorphism  $\gamma$ in this paper. The advantage of the current one is that the $R$-linear anti-involution $\sigma$ in Lemma~\ref{three isom}(d) fixes  $E_{r,1}$, where $E_{r, 1}=E$. This fact will be used in   Lemma~\ref{equation of E} when we compute Gram determinants  of $\mathscr B_{r,t}(\rho, q)$.

\end{remark}

Following \cite{Rsong}, we define
\begin{equation}\label{cenele} c_{r,t}=\small{\sum_{i=1}^r\sum_{j=1}^t\bar{E}_{i,j}-\rho^{-1}\sum_{i=2}^r\sum_{j=1}^{i-1}g_{j,i}^{-1}g_{i,j+1}^{-1}-\rho\sum_{i=2}^t
\sum_{j=1}^{i-1}g_{i,j}^{*}g_{j+1,i}^{*}},\end{equation}
where $\bar E_{i,j}$ is given in \eqref{eij}. Then  $c_{r,t}$ is  a central element in $\mathscr B_{r, t}(\rho, q)$ (see \cite[Proposition~2.4]{Rsong}).
%By Definition~\ref{qwb}, $ \mathscr B_{r,t}(\rho, q)\cong \mathscr B_{t,r}(\rho, q)$ and the  required isomorphism  sends $E$, $g_i$ and $g^*_j$ to $E$, $g_{r-i+1}^*$ and  $g_{t-j+1}$, for all $1\leq i\leq r-1, 1\leq j\leq t-1$, respectively. So,  we can  assume  $r\ge t$ without loss of any generality. This will simplify our arguments later on.

%Suppose  $1\leq k<r$.  let $\bar{\mathscr B}_{r-k,t-1}(\rho, q)$ be the subalgebra of $\mathscr B_{r,t}(\rho, q)$ generated by $E_{k,1}, g_i, g_j^*$, $1\leq i\leq k-1$, $2\leq j\leq t-1$. By %Definition~\ref{qwb},   there is an algebra isomorphism
%\begin{equation} \label{subiso} \psi: \mathscr B_{r-k,t-1}(\rho, q)\cong\bar{\mathscr B}_{r-k,t-1}(\rho, q),\end{equation} sending $E$ (resp., $g_i$, $g_j^*$) to $E_{k,1}$ (resp., $g_i$, $g_j^*$).

%It is known that the subalgebra of $\mathscr B_{r,t}(\rho, q)$ generated by $g_1,\ldots, g_{r-1}$ is isomorphic to the Hecke algebra $\mathscr H_r$ associated to $\mathfrak S_r$,  the symmetric group on $r$ letters $\{1,2,\ldots,r\}$. Similarly,  the subalgebra of $\mathscr B_{r,t}(\rho, q)$ generated by $g_1^*,\ldots, g_{t-1}^*$ is isomorphic to the Hecke algebra $\mathscr H_r^*$ associated to $\mathfrak S_r^*$,  the symmetric group on $t$ letters $\{\bar 1, \bar 2,\ldots, \bar t\}$.  Obviously,
%\begin{equation}\label{iso-hec}  \mathscr B_{r,t}(\rho, q)/\langle E\rangle\cong\mathscr H_r\otimes\mathscr H_t^*,\end{equation}
%where $\langle E\rangle$ is  the two sided ideal of $\mathscr B_{r, t}(\rho,q)$ generated by $E$.

\begin{Defn}\label{jme} For all $1\leq i\leq r$ and  $1\leq j\leq t$, define  $ x_i=c_{i,0}-c_{i-1,0}$, and $x_{r+j}=c_{r,j}-c_{r,j-1}$, where $c_{i, j}$'s are given in \eqref{cenele} and $c_{0, 0}=0$.\end{Defn}

\begin{Lemma}\label{y} The following holds for the $R$-algebra   $\mathscr B_{r, t}(\rho, q)$.
\begin{enumerate}
\item $x_1+x_2+\ldots +x_{r+t}=c_{r, t} $.
 \item The elements $x_1, x_2, \ldots, x_{r+t}$  are a family of commutative elements in  $\mathscr B_{r, t}(\rho, q)$.
\item For any $1\le i\le r-1$, we have  \begin{itemize}\item[(1)] $g_i(x_i+x_{i+1})=(x_i+x_{i+1})g_i$,
\item [(2)] $x_i=g_ix_{i+1}g_i+\rho^{-1} g_i $, and  $g_jx_i=x_ig_j$   if $j\neq i, i+1$.
\end{itemize}
\item For any  $1\le i\le t-1$, we have \begin{itemize} \item [(1)] $g_i^*(x_{r+i}+x_{r+i+1})=(x_{r+i}+x_{r+i+1})g_i^*$, \item [(2)] $g_i^*x_{r+i}g_i^* =x_{r+i+1}+\rho g_i^* $,
and  $g_i^*x_j=x_jg_i^*$, if $j\neq r+i, r+i+1$. \end{itemize}
\item $E(x_r+x_{r+1})=(x_r+x_{r+1})E=\delta E$ and $Ex_i=x_iE$ if $i\neq r, r+1$.
  \end{enumerate}
\end{Lemma}
\begin{proof} The statement  (a) follows from Definition~\ref{jme}.  (b) follows from (a) and the fact that $c_{r,t}$ is central in $\mathscr B_{r,t}(\rho,q)$.
 One can verify   (c)-(d) via   Definition~\ref{qwb}(a)-(h) and \eqref{cenele} for  $c_{i,j}$.
  By Definition~\ref{qwb}(h)-(j), $E^2=\delta E$ and  $E\bar E_{j,1}=\rho^{-1} E g^{-1}_{j,r}g^{-1}_{r,j+1}$ for all $1\leq j\leq r-1$. So,
   the first part of  (e) is proved.
From this result together with  Definition~ \ref{jme} and the fact that  $c_{i,j}$ is central in $\mathscr B_{i,j}(\rho,q)$, we can get the second part of  (e).
\end{proof}
\begin{Lemma}\label{e6} For all $a\in\Z^{>0}$, there are some $\omega_r^{(a)}\in R[x_1,x_2,\ldots, x_{r-1}]$ such that  $Ex_r^aE=\omega_r^{(a)} E$.
\end{Lemma}

\begin{proof}
 By \cite[Lemma~4.1(a)]{Rsong}, $E\mathscr B_{r, 1}(\rho, q)E=E \mathscr H_{r-1}=\mathscr H_{r-1}E$, which has a basis $\{E g_{w}\mid w\in\mathfrak S_{r-1}\}$.
Then there is   a unique element $\omega_r^{(a)}\in \mathscr H_{r-1}$ such that
$Ex_r^aE=\omega_r^{(a)}E$.  Since any element $b\in \mathscr H_{r-1}$  commutes with both $E$ and $x_r$, we have $Eb\omega_r^{(a)}=E\omega_r^{(a)} b$, forcing  $b\omega_r^{(a)}=\omega_r^{(a)} b$.
So, $\omega_r^{(a)}$ is a central element  in $\mathscr H_{r-1}$.
Now the result follows from the fact that  the center of $\mathscr H_{r-1}$   consists of all  symmetric polynomials of $x_1,x_2,\ldots,x_{r-1}$.  \end{proof}

% From Lemma~\Ref{E6}, We Have
%\Begin{Equation}\Frac{W_R(U)}{U}E=E\Frac{1}{U-X_R}E.
%\End{Equation}

\begin{Lemma}\label{forofwu} Let
$W_r(u)=\sum_{a\geq0} \frac{\omega_r^{(a)}}{u^{a}}$, where $u$ is an indeterminate. Suppose   $\omega=q-q^{-1}$ and $r\in\mathbb Z_{\geq0}$. Then
\begin{equation}\label{formiofw}
\small{W_r(u)-\frac{\rho u}{\rho^{-1} +\omega u}=\frac{\delta-u}{1+\rho \omega u} \prod _{i=1}^{r-1} \frac{ (u-x_i)^2}{(u-x_i)^2- (\rho^{-1} +\omega u) (\omega x_i +\rho^{-1} )}}.
\end{equation}
\end{Lemma}

\begin{proof} We prove it by induction on $r$. Since $x_1=0$, we have   $ W_1(u) = \delta$.
In this case, the right hand side of \eqref{formiofw} is $\frac{\delta-u}{1+\rho \omega u}$, which is equal to $\delta-\frac{\rho u}{\rho^{-1} +\omega u}$. So,   \eqref{formiofw} is true when $r=1$.
Suppose we have the result for $r=i$. By Lemma~\ref{y}(c),
\begin{equation}\label{wku}
\small{g^{-1}_i \frac{1}{u-x_{i+1}}= \frac{1}{u-x_i} g_i - \frac{\rho^{-1}+\omega u}{(u-x_i)(u-x_{i+1})}}.
\end{equation}
Noting that   $g_i=g_i^{-1}+\omega$, $E_{i+1,1}g_iE_{i+1,1}=\rho E_{i+1,1}$ and $E_{i+1,1} x_i=x_i E_{i+1,1}$ (see  Lemma~\ref{y}(e)), we have
$$
\small{\begin{aligned}
 \frac{W_{i+1}(u)}{u} &  E_{i+1, 1}=  E_{i+1,1}\frac{1}{u-x_{i+1}} E_{i+1,1}= E_{i+1,1}g_i (g_i^{-1}\frac{1}{u-x_{i+1}}) E_{i+1,1}\\
\overset{\eqref{wku}} =  &  E_{i+1,1}g_i \frac{1}{u-x_{i}}g_i E_{i+1,1}- E_{i+1,1}g_i \frac{\rho^{-1}+\omega u}{(u-x_i)(u-x_{i+1})} E_{i+1, 1} \\
=  &  g_i E_{i,1} \frac{1}{u-x_{i}} (E_{i,1}g_{i}^{-1}+\omega E_{i+1,1})- E_{i+1,1} (g_i^{-1}+\omega)\frac{\rho^{-1}+\omega u}{(u-x_i)(u-x_{i+1})}E_{i+1, 1} \\
\overset{\eqref{wku}} =& (\frac{W_i(u)}{u} + \frac{\rho \omega}{u-x_i}) E_{i+1, 1}+(\frac{-\rho( \rho^{-1}+\omega u)}{(u-x_i)^2}+ \frac{(\rho^{-1}+\omega u)^2}{(u-x_i)^2}\frac{W_{i+1}(u)}{u}\\ & - \frac{(\rho^{-1}+\omega u)\omega}{u-x_i} \frac{W_{i+1}(u)}{u}) E_{i+1,1}.
\end{aligned}}$$
Doing so yields the following formula,
\begin{equation}\label{recursion}
\small{W_{i+1}(u)-\frac{\rho u} {\rho^{-1}+\omega u}=\frac{ (u-x_i)^2}{(u-x_i)^2- (\rho^{-1}+\omega u)(\omega x_i +\rho^{-1})} (W_i(u)-\frac{\rho u} {\rho^{-1}+\omega u})}.
\end{equation}
By induction assumption on $r=i$, \eqref{formiofw} is proved for $r=i+1$.
\end{proof}

Before  we construct the Jucys-Murphy basis of  $\mathscr B_{r, t}(\rho, q)$
over $R$ in  the sense of \cite{Ma2}, let us recall the cellular basis of $\mathscr B_{r, t}(\rho, q)$ in \cite{Rsong}. Write $\Lambda^+(n)$ as the set of all partitions of $n$. There is a dominance order on $\Lambda^+(n)$ such that    $\lambda\unlhd \mu$ if $\lambda, \mu\in \Lambda^+(n)$ and  $ \sum_{j=1}^i \lambda_j\le \sum_{j=1}^i \mu_j$ for all possible
$i$.    Write $\lambda\vartriangleleft \mu$ if
$\lambda\trianglelefteq \mu$ and $\lambda\ne \mu$. Let
$l(\lambda): =\max\{i\mid \lambda_i\neq0\}$. We also write $|\lambda|=\sum_{i=1}^d\lambda_i$ with $l(\lambda)=d$ and identify any partition $\lambda$ with its Young diagram $[\lambda]$.

 Suppose $\lambda\in \Lambda^+(n)$.
A standard $\lambda$-tableau $\S$  is  obtained by inserting the numbers $\{1,2,\ldots,n\}$ into the Young diagram $[\lambda]$ such that the entries in $\S$ increase both from  left to right along each row and from top to bottom along each  column.
 Let $\Std(\lambda)$ be the set of all standard $\lambda$-tableaux. It is known that a standard $\lambda$-tableau $\S $ can be identified uniquely with an up $\lambda$-tableau $\s=(\s_0,\s_1,\ldots, \s_n)$  such that $\s_0=\emptyset$ and $\s_i$ is obtained from
$\S$ by removing the boxes  containing  the entries which are  strictly greater than $i$. In this case,  $\s_i$ is a standard $\mu$-tableau for some   $\mu\in\Lambda^+(i)$. Abusing notation,
 we use   $\s_i$ to denote $\mu$.  So, $(\s_0,\s_1,\ldots, \s_n)$ can be considered as  a sequence of partitions.
   Following \cite{Ma}, let $\unlhd $ be the dominance order  on $\Std(\lambda)$   such that
$\S\unlhd \T$ if $\s_i\unlhd \t_i$, for all $1\le i\le n$.
The maximal element in  $\Std(\lambda)$ is $\T^\lambda$ which is  obtained from  $[\lambda]$ by inserting  $1, 2, \ldots, n$ from left to right
along the rows of $[\lambda]$. The minimal element in $\Std(\lambda)$ is $\T_\lambda$ which is  obtained from
$[\lambda]$ by inserting $1, 2, \ldots, n$ from top to bottom along the columns of $[\lambda]$. Since  the symmetric group  $\mathfrak S_n $  acts on
each $\S\in \Std(\lambda)$  by permuting its entries,
we write $w=d(\S)$ if $\T^\lambda w=\S$. Then  $d(\S)$ is
uniquely determined by $\S$.
\begin{example} Suppose that $\lambda=(3,1)\in\Lambda^+(4)$. Then
$\T^\lambda=\tiny\young(123,4)$, $\T_\lambda=\tiny\young(134,2)$ and  $\T^\lambda$ can be identified with the up $\lambda$-tableau $\tiny{(\emptyset,\young(1), \young(12),\young(123),\young(123,4)  )}$, which can be considered as the sequence of partitions $((0), (1), (2), (3),(3,1))$.
Moreover,
 $d(\T_\lambda)=s_3s_2$ and  $\T_\lambda$ is identified uniquely with the  up $\lambda$-tableau $\tiny{(\emptyset,\young(1), \young(1,2),\young(13,2),\young(134,2)  )}$, which can be considered as the sequence of partitions $((0), (1), (1,1), (2,1),(3,1))$.
\end{example}

For any $r,t\in\mathbb N$, let $\Lambda^+_{r,t} =\Lambda^+(r)\times \Lambda^+(t)$ and $\Lambda_{r, t}: =\{(f, \lambda)\mid 0\le f\le \min\{r, t\}, \lambda\in \Lambda^+_{r-f,t-f}\}.$
 Given $(f, \lambda), (\ell, \mu)\in \Lambda_{r, t}$, define \begin{equation}\label{pos1} (f, \lambda)\unlhd (\ell, \mu), \text{ if
either $f<\ell$ or $f=\ell$ and $\lambda^{(i)}\unlhd \mu^{(i)}$ for all $i=1, 2$.}\end{equation} Then $(\Lambda_{r, t}, \unlhd)$ is a poset.  If $(f, \lambda)\unlhd (\ell, \mu)$ and $(f, \lambda)\neq  (\ell, \mu)$, we write $(f, \lambda)\lhd (\ell, \mu)$.
 For any $(f, \lambda)\in \Lambda_{r, t}$, define
$$\Std(\lambda)=\Std(\lambda^{(1)})\times \Std(\lambda^{(2)}), $$ where  $\Std(\lambda^{(1)})$   (resp. $\Std(\lambda^{(2)})$) is
the set of  usual standard $\lambda^{(1)}$-tableaux (resp., standard $\lambda^{(2)}$-tableaux which are obtained from
usual standard tableaux by using $\overline {f+j}$ instead of $\bar j$).
Following \cite{Rsong}, define
 \begin{equation}\label{ifla} I(f,\lambda)=\Std(\lambda)\times \mathscr D^f_{r, t},\end{equation}
where $\mathscr{D}_{r,t}^f=\{ s_{r-f+1,i_f} s^*_{f, j_f} \cdots s_{r,i_1}s^*_{1,{j_1}}|  k \le {j_k}, 1 \le i_f< i_{f-1} < \ldots <i_1\le r \}$ for $f>0$ and $\mathscr D_{r,t}^0=\{1\}$.
 For any $(\S
, e), (\T, d)\in I(f, \lambda)$,  let \begin{equation}\label{cellbasis} C_{(\S, e)(\T, d)}=\sigma(g_e)
\sigma( g_{d(\S)}) n_\lambda g_{d(\T)} g_d,\end{equation}
where  \begin{enumerate} \item $n_\lambda=E^f y_{\lambda}$, $E^f=E_{r, 1} E_{r-1, 2} \cdots E_{r-f+1, f} $ if $0<f\le \min\{r, t\}$, and $E^0=1$,
 \item  $y_{\lambda}=y_{\lambda^{(1)}} y_{\lambda^{(2)}}$,    where $y_{\lambda^{(1)}}=\sum_{w\in \mathfrak S_{\lambda^{(1)}}}(-q)^{-\ell(w)} g_w$ and $y_{\lambda^{(2)}} $ is obtained from usual one by using $g_{f+j}^*$ instead of $g_j^*$, $\mathfrak  S_{\lambda^{(1)}}$ is the stabilizer of $\t^{\lambda^{(1)}}$ and  $\ell(\cdot)$ is the usual length function on symmetric group.
 \item    $g_{d(\S)}=g_{d(\S_1)} g^*_{d(\S_2)}$ if $\S=(\S_1, \S_2)$,
\item  $\sigma$ is the $R$-linear anti-involution on  $\mathscr B_{r, t}(\rho, q)$ in Lemma~\ref{three isom}(d).
\end{enumerate}

\begin{Theorem}\label{celb}\cite[Theorem~3.7]{Rsong}Let $\mathscr
B_{r,t}(\rho, q)$ be the quantized walled Brauer algebra over $R$. Then
$\mathcal C$ is a cellular $R$-basis of $\mathscr B_{r,t}(\rho, q)$ over
the poset $\Lambda_{r, t}$ in the sense of \cite{GL}, where $$\mathcal C=\cup_{(f, \lambda)\in
\Lambda_{r,t}} \{C_{(\S, e)(\T, d)}\mid (\S, e), (\T, d) \in I(f,
\lambda)\}.$$ \end{Theorem}

%We write $y_{\lambda}=y_{\lambda^{(1)}} y_{\lambda^{(2)}}$ if $\lambda=(\lambda^{(1)}, \lambda^{(2)})\in \Lambda^+(r-f)\times \Lambda^+(t-f)$.
By Graham-Lehrer's results on cellular algebras in  \cite{GL},   we have the right cell module
$C(f,\lambda)$ for each  $(f, \lambda)\in \Lambda_{r, t}$, which   has basis
 $\{ n_{\lambda} g_{d(\S)} g_d +\mathscr B_{r,t}(\rho, q)^{\rhd(f,\lambda)}|(\S, d) \in I(f,\lambda)\}$, where $\mathscr B_{r,t}(\rho, q)^{\rhd(f,\lambda)}$ is the free $R$-submodule of $\mathscr B_{r,t}(\rho, q)$ spanned by
 $\cup_{(\ell, \mu)\in\Lambda_{r,t}} \{C_{(\S, e)(\T, d)}\mid (\S, e), (\T, d) \in I(\ell,\mu), (\ell, \mu)\rhd (f, \lambda)\}$. In fact, $\mathscr B_{r,t}(\rho, q)^{\rhd(f,\lambda)}$ is
 a two-sided ideal of $\mathscr B_{r, t}(\rho, q)$.

We are going to  construct the Jucys-Murphy basis of  $\mathscr B_{r, t}(\rho, q)$
over $R$ in  the sense of \cite{Ma2}.
For any partition $\lambda$,  we write $\lambda\setminus p$ (resp., $\lambda \cup p$) as the partition obtained from $\lambda$ by removing the node $p $ (resp., adding the node $p $). The node $p$ is called a removable (resp., an addable) node of $\lambda$ if $\lambda\setminus p$  (resp.,$\lambda \cup p$) is also a partition.  Let $\mathscr R(\lambda)$ (resp., $\mathscr A(\lambda)$) the set of all removable nodes (resp., addable  nodes ) of $\lambda$.
Similarly, for any bi-partition $\lambda=(\lambda^{(1)},\lambda^{(2)})$, we have the notation $\mathscr R(\lambda)$ and $\mathscr A(\lambda)$.

We  fix a  set $I_{r,t}$, where \begin{equation}\label{brpath} I_{r,t}=\{(1,0),(2,0),\ldots,(r,0),(r,1),(r,2),\ldots,(r,t)\}.\end{equation}
For all $1\leq k\leq r+t$, there is a unique $(i,j)\in I_{r,t}$ such that $k=i+j$. Hence   we  can also  write $\Lambda_k$ (resp. $\mathscr B_k(\rho, q)$, $c_{k}$) instead of $\Lambda_{i,j}$ (resp., $\mathscr B_{i,j}(\rho, q)$, $c_{i,j}$) without any confusion.

 \begin{Defn}\label{updt}
Suppose that  $(f,\lambda)\in\Lambda_{r,t}$. An up-down $\lambda$-tableau with respect to  $I_{r, t}$   is a sequence of  bi-partitions
$\t=(\t_{0},\t_{1},\ldots,\t_{r+t})$
such that
\begin{enumerate}
\item $\t_0=((0), (0))$, $\t_{r+t}=\lambda$ and $((k-|\t_{k}^{(1)}|-|\t_{k}^{(2)}|)/2,\t_k)\in \Lambda_k$, for $k>0$;
 \item If $1\leq k\leq r$, set $\t^{(1)}_{k}=\t_{k-1}^{(1)}\cup  p$ for some  $p\in \mathscr A(\t_{k-1}^{(1)})$, and $ \t^{(2)}_k=(0)$;
\item If $r+1\leq k\leq r+t$, either $\t^{(1)}_{k} =\t^{(1)}_{k-1}$ and  $\t^{(2)}_{k}=\t^{(2)}_{k-1}\cup p$ for some $p\in \mathscr A(\t^{(2)}_{k-1})$ or  $\t^{(1)}_{k} =\t^{(1)}_{k-1}\setminus p$ for some $p\in \mathscr R(\t^{(1)}_{k-1})$ and $\t^{(2)}_{k} =\t^{(2)}_{k-1}$.
\end{enumerate}
\end{Defn}

Let ${\upd_{r,t}}(\lambda)$ be the set of all up-down $\lambda$-tableaux for $(f,\lambda)\in\Lambda_{r,t}$. We write
$\mu\rightarrow \lambda$ if $\lambda=\mu\setminus p$ or $\lambda=\mu\cup p$. So, $\t_{k-1}\rightarrow \t_k$ for any $\t\in{\upd_{r,t}}(\lambda)$ and $1\leq k\leq r+t$. Recall   we can  also identify a partition with its Young diagram.
\begin{example}\label{exofupd}Suppose $\lambda=((1),(1))$.
There are 4 elements in ${\upd_{2,2}}(\lambda)$ as follows.
$$\begin{aligned}
\t =&\tiny{\left((\emptyset,\emptyset), \left(\young(\ ),\emptyset\right),(\young(\ \ ),\emptyset),\left(\young(\ ),\emptyset\right),\left(\young(\ ),\young(\ )\right)\right)},\\
\u=&\tiny{\left((\emptyset,\emptyset),\left(\young(\ ),\emptyset\right),\left(\young(\ \ ),\emptyset\right),\left(\young(\ \ ),\young(\ )\right),\left(\young(\ ),\young(\ )\right)\right)},\\
\s=&\tiny{\left((\emptyset,\emptyset),\left(\young(\ ),\emptyset\right),\left(\young(\ ,\ ),\emptyset\right),\left(\young(\ ),\emptyset\right),\left(\young(\ ),\young(\ )\right)\right)},\\
\v=&\tiny{\left((\emptyset,\emptyset),\left(\young(\ ),\emptyset\right),\left(\young(\ ,\ ),\emptyset\right),\left(\young(\ ,\ ),\young(\ )\right),\left(\young(\ ),\young(\ )\right)\right)}.
\end{aligned}$$
\end{example}

For any $(i,j)\in I_{r,t}$ with $i+j=k$ and any $\lambda\in \Lambda_{i,j}$, we also have the  notation ${\upd_k}(\lambda):= {\upd_{i,j}}(\lambda)$.
 For any $\s\in \upd_{r,t}(\lambda)$ and $1\leq k \leq r+t$,
let $l_{\s_k}\in\mathbb N$ such that
$$k=2l_{\s_k}+|\s_k^{(1)}|+|\s_k^{(2)}|. $$
Then $(l_{\s_k},\s_k)\in\Lambda_k$.
For any $\s, \t\in \upd_{r,t}(\lambda)$, note that $\s_0=\t_0$ and $\s_1=\t_1$. Suppose $2\leq k\leq r+t-1$, we  write $\s\overset {k} \prec \t$ if there is an $l\geq k$ such that
$\s_j=\t_j$ for $l+1\le j\le r+t$ and $(l_{\s_l},\s_l)\lhd (l_{\t_l},\t_l)$ in the sense of \eqref{pos1}. We say $\s\prec\t$ if $\s\overset{k}\prec\t$ for some  $2\leq k\leq r+t-1$. Then $\prec$ is a total order on
${\upd_{r,t}}(\lambda)$. It is easy to see that  $\v\prec \u\prec \s\prec \t $ in example ~\ref{exofupd}.

%We write $p=\lambda\ominus \mu$ if either $\lambda=\mu\cup p$ or $\lambda=\mu\setminus p$.

% Then there is a partial order, say  $\unlhd$  on  ${\upd}(\lambda)$ in the sense that , for any $\s, \t\in \upd(\lambda)$,??
%\begin{equation} \label{uppar} \s\unlhd \t, \text{ if } (l_{\s_k},\s_k)\unlhd (l_{\t_k},\t_k) \text{ for all } 1\leq k\leq r+t .\end{equation}
%where $l_{\s_k}, l_{\t_k} $ are non-negative integers such that  $(l_{\s_k},\s_k),(l_{\t_k},\t_k)\in\Lambda_k$.
%Each $\t$ in ${\upd}(\lambda)$ is of form $(\t_0, \ldots, \t_{r+t})$ such that  each $\t_i$ is a pair of partitions $(\t_i^{(1)}, \t^{(2)}_i)$. Moreover,
%\begin{enumerate}

%\end{enumerate}

\begin{Defn}\label{definition of m} Given any $\t\in \upd_{r,t}(\lambda)$ for some  $(f,\lambda)\in\Lambda_{r,t}$.
Define   $\n_\t=\n_{\t_{r+t}}$ and $\n_{\t_{k}}$($1\le k\le r+t$) inductively as follows. If $k=1$, then we write $\n_{\t_1}=1$.   Suppose $\t_{k}=(\mu^{(1)}, \mu^{(2)})$.
Define
\begin{enumerate}
\item   $\n_{\t_k}=\sum_{j=a_{l-1}+1}^{a_l} (-q)^{j-a_l} g_{j,k} \n_{\t_{k-1}}$, if $ \t^{(1)}_{k-1}=\t^{(1)}_k\setminus (l,\mu_l^{(1)})$ and $ a_h=\sum _{i=1}^h\mu^{(1)}_i$, $h\in \{l,l-1\}$;
\item $\n_{\t_k}= \sum_{j=b_{l-1}+1}^{b_l}(-q)^{j-b_l}g^*_{j,k-r}\n_{\t_{k-1}}$,  if $ \t^{(2)}_{k-1}=\t_k^{(2)}\setminus (l,\mu_l^{(2)})$ and  $b_h=l_{\t_k}+\sum _{i=1}^h\mu^{(2)}_i$, $h\in \{l,l-1\}$;
\item $\n_{\t_k}= E_{r-l_{\t_k}+1,l_{\t_k}}g^*_{l_{\t_k},k-r}g^{-1}_{a_l+1, r-l_{\t_k}+1}\n_{\t_{k-1}}$, if $ \t^{(1)}_k=\t^{(1)}_{k-1}\setminus (l,\mu_l^{(1)}+1)$ and  $ a_l=\sum _{i=1}^l\mu^{(1)}_i$.

\end{enumerate}
\end{Defn}
In  Example~\ref{exofupd}, we have   $\n_{\t}=E(1-q^{-1} g_1)$, $\n_{\u}=Eg_1^*(1-q^{-1}g_1)$, $\n_{\s}=E$, $\n_{\v}=Eg_1^*$.
%The following result can be checked directly.

\begin{Lemma}\label{bt2} Suppose $\t\in \upd_{r,t}(\lambda)$ for some $(f, \lambda)\in \Lambda_{r, t}$. Then there is a $b_\t\in\mathscr H_{r, t}$ such that $\n_\t=n_\lambda b_\t$, where $n_\lambda$ is given in \eqref{cellbasis}.  \end{Lemma}
 \begin{proof} By Definition~\ref{definition of m},  $b_\t=b_{\t_{r+t}}$ can be obtained inductively with  $b_{\t_1}=1$, $b_{\t_k}=h_k b_{\t_{k-1}}$ and
 \begin{equation}\label{htdddd}
h_k=\left\{
      \begin{array}{ll}
         g_{a_l,k}, & \hbox{if $ \t^{(1)}_{k-1}=\t^{(1)}_k\setminus (l,\mu_l^{(1)})$;} \\
        g^*_{b_l,k-r}, & \hbox{if $ \t^{(2)}_{k-1}=\t^{(2)}_k\setminus (l,\mu_l^{(2)})$;} \\
        g^*_{l_{\t_k},k-r}g^{-1}_{a_l+1, r-l_{\t_k}+1}y, & \hbox{if $ \t^{(1)}_k=\t^{(1)}_{k-1}\setminus (l,\mu_l^{(1)}+1)$,}
      \end{array}
    \right.
\end{equation}
where $y=\sum_{j=a_{l-1}+1}^{a_l+1}(-q)^{-(a_l+1-j)}g_{a_l+1,j}$, $b_l=l_{\t_k}+\sum _{i=1}^l\mu^{(2)}_i$ and $ a_h=\sum _{i=1}^h\mu^{(1)}_i$.
 \end{proof}
For $\u$ in   Example~\ref{exofupd}, $\n_\u=n_\lambda b_\u$ with $n_\lambda=E$, $b_\u= g_1^*(1-q^{-1}g_1)$, and $h_{4}=g_1^*(1-q^{-1}g_1)$, $h_3=h_2=h_1=1$.
%  The following  two special up-down $\lambda$-tableaux $\t^\lambda$ and $\t_\lambda$ will be very useful later.

For any $(f,\lambda)\in\Lambda_{r,t}$, there are two special element  $\t^\lambda$ and $\t_\lambda$ in $\upd_{r,t}(\lambda) $ as follows. Moreover,
 $\t^\lambda$ is the  unique maximal element $\t^\lambda$ in $\upd_{r,t}(\lambda) $ with respect to the order $\prec$.
\begin{example}\label{exampleoft}Let  $(f,\lambda)\in\Lambda_{r,t}$. If $f>0$, let $\mu, \nu\in \Lambda^+(r)$ such that
  $\mu_1=\lambda^{(1)}_1+f$, $\mu_i=\lambda_i^{(1)}$ for $ i>1$, and   $\nu_i=\lambda^{(1)}_i$ for $1\leq i\leq l(\lambda^{(1)})$, $\nu_i=1$ for $l(\lambda^{(1)})+1\leq i\leq l(\lambda^{(1)})+f$.
If $f=0$, let  $\mu, \nu\in \Lambda^+(r)$ such that $\nu=\mu=\lambda^{(1)}$.

\begin{enumerate} \item  Let  $\t^\lambda=(\t_0, \ldots, \t_{r+t}) \in{\upd_{r,t}}(\lambda)$ such that
\begin{enumerate}
\item    $( \t_1^{(1)}, \ldots, \t_{r}^{(1)})=\mathbf t^\mu$, $(\t_{r+f+1}^{(2)}, \ldots, \t_{r+t}^{(2)})=\mathbf t^{\lambda^{(2)}}$, and $\t_k^{(1)}=\lambda^{(1)}$ for $r+f+1\leq k\leq r+t$,
 \item     $\t^{(1)}_{r+k+1}=\t^{(1)}_{r+k}\setminus (1,\lambda^{(1)}_1+f-k)$,  $0\leq k\leq f-1$ if $f>0$.

\end{enumerate}
\item  Let $\t_\lambda=(\s_0, \ldots, \s_{r+t}) \in{\upd_{r,t}}(\lambda)$ such that
\begin{enumerate}
\item   $ ( \s_1^{(1)}, \ldots, \s_{r}^{(1)})=\mathbf t^\nu$, $\s_k=(\t^\lambda)_k$, for $r+f+1\leq k\leq r+t$,

\item $\s_{r+k}=\s_{r-k}$, for $1\leq k\leq f$ if $f>0$.
\end{enumerate}
\end{enumerate}
For ${\upd_{2,2}}(\lambda)$ given in Example~\ref{exofupd}, we have  $\t^\lambda=\t$ and $\t_\lambda=\s$.

 By Lemma~\ref{bt2}, we can know that $\n_{\t_{\lambda}}=n_\lambda $ for any $(f,\lambda)\in\Lambda_{r,t}$. Moreover,  $\n_{\t^\lambda}=n_\lambda$ if $f=0$ and $\n_{\t^\lambda}=n_\lambda b_{\t^\lambda}$ if $f>0$,  where
\begin{equation}\label{btlambda}
b_{\t^\lambda}=\prod_{j=1}^f g^{-1}_{\lambda^{(1)}_1+f-j+1,r-j+1}\sum_{l=1}^{\lambda^{(1)}_1+f+1-j}(-q)^{-(\lambda^{(1)}_1+f+1-j-l)}g_{\lambda^{(1)}_1+f+1-j, l }.
\end{equation}\end{example}

In order to construct the Jucys-Murphy basis of $\mathscr B_{r,t}(\rho, q)$ over $R$, we need the following notations.
Suppose $(f, \lambda)\in \Lambda_{r, t}$.  Let $(f_1,\alpha^{(1)}),\ldots,(f_n,\alpha^{(n)})$ be all the elements in $\Lambda_{r+t-1}$ such that  $\alpha^{(i)}\rightarrow\lambda$, for all $1\le i\le n$. We  arrange them such that $(f_1,\alpha^{(1)})\rhd\ldots\rhd(f_n,\alpha^{(n)})$.
Define
$$\begin{aligned}
N^{\unrhd \alpha^{(k)}}=&R\text{-span}\{\n_\t+ \mathscr B_{r,t}(\rho, q)^{\rhd(f,\lambda)}\mid\t\in{\upd_{r,t}}(\lambda),\t_{r+t-1}=\alpha^{(j)}, j\le k\},\\
N^{\rhd \alpha^{(k)}}=&R\text{-span}\{\n_\t+\mathscr  B_{r,t}(\rho, q)^{\rhd(f,\lambda)}\mid\t\in{\upd_{r,t}}(\lambda), \t_{r+t-1}=\alpha^{(j)}, j< k \}.
\end{aligned}
$$

\begin{Theorem} \label{mcell}   Fix  $(f, \lambda)\in \Lambda_{r, t}$. Then \begin{enumerate} \item The set $\{\n_\t+ \mathscr B_{r,t}(\rho, q)^{\rhd(f,\lambda)}\mid\t\in{\upd_{r,t}}(\lambda)\}$ is an $R$-basis of $C(f,\lambda)$.
\item Both  $N^{\unrhd \alpha^{(k)}}$ and $N^{\rhd \alpha^{(k)}}$ are $\mathscr  B_{r+t-1}(\rho, q)$-submodules of $C(f,\lambda)$.
\item The $R$-linear map $\phi: N^{\unrhd \alpha^{(k)}}/N^{\rhd \alpha^{(k)}}\rightarrow C(f_k,\alpha^{(k)})$ sending $\n_\t+N^{\rhd \alpha^{(k)}}$
to $\n_{\t_{r+t-1}}+ \mathscr B_{r+t-1}(\rho, q)^{\rhd(f_k,\alpha^{(k)})}$
is an isomorphism of  $\mathscr  B_{r+t-1}(\rho, q)$-modules.

\end{enumerate}
\end{Theorem}

\begin{proof}  The main idea of this result is motivated by \cite{Enyang1}. When $t=0$, $\mathscr B_{r, 0}(\rho, q)=\mathscr H_r$. In this case, (a)-(c) follows from the  branching  rules  for Hecke algebras in  \cite{Ma}.
%In the following argument we assume  $t>0$.

If $t\neq 0$, for any $\t\in \upd_{r,t}(\lambda)$ with $\t_{r+t-1}=\alpha^{(k)}$, let $\u=(\t_0, \ldots, \t_{r+t-1})\in \upd_{r+t-1}(\alpha^{(k)})$. By Definition~\ref{definition of m}, we have $\n_\t=y_{\alpha^{(k)}}^\lambda b_\u$,  where $y_{\alpha^{(k)}}^\lambda =n_\lambda h_{r+t}$, $b_\u$ and $h_{r+t}$ are  defined in the proof of Lemma~\ref{bt2}.
By  \cite[Theorem~4.15]{Rsong}, there is a filtration
$$0=N_0\subset N_1\subset N_2\subset\ldots\subset N_n=C(f,\lambda) $$
of $ \mathscr B_{r+t-1}(\rho, q)$-modules such that  $N_k=\sum_{j=1}^k y_{\alpha^{(j)}}^\lambda \mathscr B_{r+t-1}(\rho, q)+\mathscr B_{r,t}(\rho, q)^{\rhd (f,\lambda)}$, for $1\leq k\leq n$. Moreover,
\begin{equation}\label{br1}
 \phi:  C(f_k,\alpha^{(k)}) \cong N_k/N_{k-1},
 \end{equation} and  $\phi$ sends
 $n_{\alpha^{(k)}} g_{d(\S)}g_d +\mathscr B_{r+t-1}(\rho, q)^{\rhd (f_k,\alpha^{(k)})}$ to  $ y_{\alpha^{(k)}}^\lambda g_{d(\S)}g_d+N_{k-1}$, for all $ (\S,d)\in I(f_k, \alpha^{(k)})$. Hence $\phi$ sends $\n_\u+\mathscr B_{r+t-1}(\rho, q)^{\rhd (f_k,\alpha^{(k)})}$ to $\n_\t+N_{k-1}$.
By induction assumption on $\mathscr B_{r+ t-1}(\rho, q)$, the cell module  $C(f_k,\alpha^{(k)})$ has basis $\{\n_\u+\mathscr B_{r+t-1}(\rho, q)^{\rhd (f_k,\alpha^{(k)})} \mid \u\in \upd_{r+t-1}(\alpha^{(k)})\}$. Therefore,  $\{\n_\t+ N_{k-1}\mid\t\in{\upd_{r,t}}(\lambda),\t_{r+t-1}=\alpha^{(k)}\}$ is a basis of $N_{k}/N_{k-1} $. Since $N_0=0$, by induction assumption  on $k-1$, it is shown that $N_k=N^{\unrhd \alpha^{(k)}}$. We complete the proof of (a)--(c).
\end{proof}

%By Theorem~\ref{induction} and (\ref{ylambda}), we can start to construct the Jucys-Murphy basis of $\mathscr B_{r, t}(\rho,q)$ over $R$ now.

The following result follows from standard arguments (cf. \cite[Theorem~2.7]{RS}) by using Theorem~\ref{mcell}.
\begin{Theorem}\label{mm}    For any  $\s,\t\in {\upd_{r,t}}(\lambda)$, and $(f,\lambda)\in\Lambda_{r,t}$,  let
 $\n_{\s\t}=\sigma(b_\s)n_\lambda b_\t$, where $\sigma$ is the anti-involution given in Lemma~\ref{three isom}(d). Then
\begin{enumerate}
\item $\mathscr M_{r,t}=\big\{\n_{\s\t}\mid\s,\t\in{\upd_{r,t}}(\lambda),(f,\lambda)\in\Lambda_{r,t}, 0\leq f\leq\min\{r,t\}\big\}$
is a free $R$-basis of $ \mathscr B_{r,t}(\rho, q)$.
\item $\sigma(\n_{\s\t})=\n_{\t\s}$ for all $\s,\t\in{\upd_{r,t}}(\lambda),(f,\lambda)\in\Lambda_{r,t}.$
\item Let $\tilde{ \mathscr B}_{r,t}(\rho, q)^{\rhd(f,\lambda)}$ be the free $R$-module of $ \mathscr B_{r,t}(\rho, q)$ generated by $\n_{\tilde{\s}\tilde{\t}}$
with $\tilde{\s},\tilde{\t}\in{\upd_{r,t}}(\mu),(k,\mu)\rhd(f,\lambda)$, then $\tilde{ \mathscr B}_{r,t}(\rho, q)^{\rhd(f,\lambda)}= \mathscr B_{r,t}(\rho, q)^{\rhd(f,\lambda)}$.
\item Fix $\s,\t\in{\upd_{r,t}}(\lambda)$ with $(f,\lambda)\in\Lambda_{r,t}$, for any $h\in \mathscr B_{r,t}(\rho, q)$, there exist $a_\u\in R$ which are independent of $\s$, such that $$\n_{\s\t}h\equiv\sum_{\u\in{\upd_{r,t}}(\lambda)}a_\u\n_{\s\u}(\text{mod } \mathscr B_{r,t}(\rho, q)^{\rhd(f,\lambda)}).$$
\end{enumerate}
\end{Theorem}

Let  $[k]:=\frac{q^k-q^{-k}}{q-q^{-1}}$ for $k\in\mathbb Z$. If   $p$ is in the $i$th row and $j$th column of a Young diagram, we write \begin{equation} \label{ccont}c(p)=j-i,\end{equation} and call it  the  content of the box  $p$.
For any  $\t\in{\upd_{r,t}}(\lambda)$ and $1\leq k\leq r+t$,  we define
\begin{equation} \label{content}c_\t(k)=\begin{cases}{\delta}-\varepsilon\rho^\varepsilon q^{-\varepsilon c(p)}[\varepsilon c(p)],&\text{ if } \t_{k-1}=\t_{k}\cup p,\\
\varepsilon \rho^\varepsilon q^{-\varepsilon c(p)}[\varepsilon c(p)],&\text{ if } \t_{k-1}=\t_{k}\setminus p,
\end{cases}
\end{equation}
where $ \varepsilon=-1$ (resp., $1$) if $p$ is in either $\t_{k}^{(1)}$ or $\t_{k-1}^{(1)}$ (resp., either  $\t_{k}^{(2)}$ or
$ \t_{k-1}^{(2)}$).

\begin{Theorem}\label{ll} Suppose $\s,\t\in{\upd_{r,t}}(\lambda)$ with $(f,\lambda)\in\Lambda_{r,t}$ and $1\leq k\leq r+t$. Then
\begin{equation}\label{lll} \n_{\s\t}x_{k}\equiv c_\t(k)\n_{\s\t}+\sum_{\u\in{\upd_{r,t}}(\lambda)\atop\t\overset{k-1} \prec \u}a_\u\n_{\s\u}\quad (\text{mod }\mathscr B_{r,t}^{\rhd(f,\lambda)}) \end{equation}
for some $a_\u\in R$.
\end{Theorem}
\begin{proof}   By \cite[Lemma~6.3]{Rsong}, the central element  $c_{r, t}$ in \eqref{cenele} acts on $C(f, \lambda)$ as the scalar
  \begin{equation}\label{crtlambda}
c_{r+t}(\lambda):=f\delta+\rho^{-1}q^{c(p)}[c(p)]-\rho q^{-c(p)}[c(p)].
 \end{equation}
 Similarly, by  Theorem~\ref{mcell}(c), $c_{r, t}-x_{r+t}$ acts on the quotient module $N^{\unrhd \t_{r+t-1}}/N^{\rhd \t_{r+t-1}}$ as the  scalar $c_{r+t-1}(\t_{r+t-1})$, too. This implies  \eqref{lll} for $k=r+t$. Repeating above arguments yields   \eqref{lll}  for general $k$.
 \end{proof}

By Lemma~\ref{y}(b) and Theorem~\ref{ll}, we see that the basis of $\mathscr B_{r, t}(\rho, q)$ given  in Theorem~\ref{mm}(a) is the Jucys-Murphy basis in the sense of \cite{Ma2}.
Later on, we  use $\n_\t$ to denote  $\n_\t+\mathscr B_{r,t}(\rho, q)^{\rhd(f,\lambda)} $ so as  to simplify the notation.
From \cite{GL}, there is a symmetric invariant bilinear form $\langle \ , \ \rangle:C(f,\lambda)\times C(f,\lambda)\rightarrow R$.
 More explicitly, the scalar $\langle\n_\s,\n_\t\rangle\in R$ is determined by
\begin{equation}\label{comofm} \n_{\tilde{\s}\s}\n_{\t\tilde{\t}}\equiv\langle\n_\s,\n_\t\rangle\n_{\tilde{\s}\tilde{\t}}\quad(\text{mod } \mathscr B_{r,t}(\rho, q)^{\rhd(f,\lambda)}),
\text{ for all } \tilde{\s},\tilde{\t}\in{\upd_{r,t}}(\lambda).\end{equation}
By general results on cellular algebras in \cite{GL}, the above symmetric invariant bilinear form is independent of $\tilde{\s},\tilde{\t}\in{\upd_{r,t}}(\lambda)$.
 % We will compute $\det G_{f, \lambda}$ in Section 5, where $G_{f,\lambda}$ is  the corresponding  Gram matrix  associated to $C(f,\lambda)$.

\section{ Orthogonal representations  }
For any field $\kappa$ containing invertible elements $\rho,q$ and $q-q^{-1}$, we have the algebra $ $
$B_\kappa:=B\otimes_R\kappa$  for any free $R$-algebra $B$ (with finite rank).  We will leave out $\kappa$ in the subscripts.
In this section, we  assume  $\kappa=\mathbb C(\rho ,q)$. By \cite[Theorem~6.10]{Rsong}, $\mathscr B_{r,t}(\rho, q)$ is semisimple over $\kappa$.

\begin{Lemma}\label{conditionofcontent} Suppose that $\t\in{\upd_{r,t}}(\lambda)$ and $\s\in{\upd_{r,t}}(\mu)$ for some $(f,\lambda), (l,\mu)\in\Lambda_{r,t}$.
\begin{enumerate}
\item $\s=\t$ if and only if $c_\s(k)=c_\t(k)$ for all $1\leq k\leq r+t$.
\item If  $k\in \{1, \ldots, r-1, r+1, \ldots, r+t-1\}$, then $c_\t(k)- c_\t(k+1)\neq 0$.
\item If $\t_{r-1}\neq \t_{r+1}$, then $c_\t(r)+ c_\t(r+1)\neq \delta $.
%\item If $\t_{r-1}=\t_{r+1}$, then $c_\t(r)\neq \delta$.
\end{enumerate}
\end{Lemma}

\begin{proof} In order to show (a), we need to verify  $\s=\t $ if $c_\s(k)=c_\t(k)$ for all $1\leq k\leq r+t$.
 By induction on $r+t$, we have  $ \s_{k}=\t_{k}$, $1\leq k\leq r+t-1$. Since $\s$ and $\t$ play the similar role,  there are  four   cases that  we need to discuss as follows:
 \begin{enumerate}
 \item[(1)] $\t_{r+t}=\t_{r+t-1}\cup p_1$, $\s_{r+t}=\t_{r+t-1}\cup p_2$, $p_1,p_2\in\mathscr A(\t_{r+t-1}^{(2)})$ with $t>0$,
 \item[(2)] $\t_{r+t}=\t_{r+t-1}\cup p_1$, $\s_{r+t}=\t_{r+t-1}\cup p_2$, $p_1,p_2\in\mathscr A(\t_{r+t-1}^{(1)})$ with $t=0$,
 \item[(3)] $\t_{r+t}=\t_{r+t-1}\setminus p_1$, $\s_{r+t}=\t_{r+t-1}\setminus p_2$, $p_1,p_2\in\mathscr R(\t_{r+t-1}^{(1)})$ with $t>0$,
 \item[(4)] $\t_{r+t}=\t_{r+t-1}\cup p_1$, $\s_{r+t}=\t_{r+t-1}\setminus p_2 $, $p_1\in\mathscr A(\t_{r+t-1}^{(2)})$, $p_2\in \mathscr R(\t_{r+t-1}^{(1)})$ with $t>0$.
 \end{enumerate} In the  cases (1)-(3),  since $\rho$ and $q$ are indeterminates, by ~\eqref{content}, we have   $c(p_1)=c(p_2)$, forcing $p_1=p_2$. So  $\t_{r+t}=\s_{r+t}$.
 We claim that (4) never happens. Otherwise,    $\rho^2=q^{2(c(p_1)+c(p_2))}$ in $\mathbb C(\rho,q)$. This is  a  contradiction.
Finally, one can check  (b)-(c) directly.
\end{proof}

%Later on, we will denote  $p$ by  $\lambda\ominus\mu$ no matter whether $\lambda=\mu\setminus p$ or $\lambda=\mu\cup p$.

\begin{Defn}\label{8}
%Let $\mathscr B_{r,t}(\rho, q)$ be the quantized walled Brauer algebra over $\kappa$.
 For any $\t,\s\in{\upd_{r,t}}(\lambda)$ with $(f,\lambda)\in\Lambda_{r,t}$, define
\begin{enumerate}
\item $\mathscr R(k)=\{c_\t(k)\mid\t\in{\upd_{r,t}}(\lambda),(f,\lambda)\in\Lambda_{r,t}\}$, $1\leq k\leq r+t$,
\item $F_\t=\prod\limits_{k=1}^{r+t} F_{\t, k}$ where $F_{\t, k}=\prod\limits_{a\in \mathscr R(k)\atop c_\t(k)\neq a}\frac{x_k-a}{c_\t(k)-a},$
\item $f_{\s\t}=F_\s\n_{\s\t}F_\t$,
\item $f_\s=\n_\s F_\s$,
\end{enumerate}
\end{Defn}
We remark that in the case  (d), we omit $( \text{mod } \mathscr B_{r,t}(\rho, q)^{\rhd(f,\lambda)})$  for the
simplification of notation. Therefore, (d) should be read as $f_\s\equiv\n_\s F_\s\ ( \text{mod } \mathscr B_{r,t}(\rho, q)^{\rhd(f,\lambda)})$.
By  general results on orthogonal form of cellular algebras in   \cite{Ma2},  $\mathscr B_{r,t}(\rho, q)$ has  orthogonal basis $\{f_{\s\t}\mid \s,\t\in{\upd_{r,t}}(\lambda),(f,\lambda)\in\Lambda_{r,t}\}$ in the sense that
  $f_{\s\t}f_{\u\v}=\delta_{\t,\u}\langle f_\t,f_\t\rangle f_{\s\v}$, where $\delta_{\t,\u}$ is the Kronecker function and $\langle \ , \ \rangle$ is the invariant form defined on the cell module $C(f, \lambda)$ in \eqref{comofm}. Further,
    \begin{equation} \label{eigenv1} f_{\s\t}x_k=c_\t(k)f_{\s\t}, \text{ $f_\t x_k=c_\t(k)f_\t$,\ \
  $ f_{\s\t}F_\u=\delta_{\t,\u}f_{\s\t}$, and  $f_\t F_\u=\delta_{\t,\u}f_\t$}\end{equation}
  for all  $1\leq k\leq r+t$, and all $\s, \t\in \upd_{r,t}(\lambda)$ and $\u\in \upd_{r,t}(\mu)$.
  The cell module $C(f, \lambda)$ has orthogonal basis $\{f_\s\mid \s\in \upd_{r,t}(\lambda)\}$ and  the transition matrix between two bases $\{\n_\t\mid \t\in \upd_{r,t}(\lambda)\}$ and $\{f_\t\mid \t\in \upd_{r,t}(\lambda)\}$ is upper-unitriangular (see Theorem~\ref{ll}).
  The following result follows from    Theorem~\ref{mcell} and Theorem~\ref{ll}.

\begin{Cor} \label{fmcell}   Suppose $(f, \lambda)\in \Lambda_{r, t}$. Keep the setup in Theorem~\ref{mcell}.
 \begin{enumerate}
 \item   $\{f_\t \mid\t\in{\upd_{r,t}}(\lambda), \t_{r+t-1}\unrhd \alpha^{(k)} \}$ is a $\kappa$-basis of $N^{\unrhd \alpha^{(k)}}$.
\item $\{f_\t \mid\t\in{\upd_{r,t}}(\lambda), \t_{r+t-1}\rhd \alpha^{(k)} \}$ is a $\kappa$-basis of $N^{\rhd \alpha^{(k)}}$.
\item The isomorphism in Theorem~\ref{mcell}(c) sends $f_\t+N^{\rhd \alpha^{(k)}}$
to $f_{\t_{r+t-1}}$.

\end{enumerate}
\end{Cor}

   Let $G_{f,\lambda}$ (resp., $\tilde{G}_{f,\lambda}$) be the Gram matrix associated to the cell module $C(f,\lambda)$,   which is defined via its Jucys-Murphy basis in Theorem~\ref{mm} (resp., orthogonal basis). Then \begin{equation}\label{det2} \det G_{f,\lambda}=\det \tilde{G}_{f,\lambda}=\prod_{\t\in \upd_{r,t}(\lambda)}
    \langle f_\t,f_\t\rangle.\end{equation}
   Note that $\langle f_\t,f_\t\rangle\neq0 $ for any $\t\in \upd_{r,t}(\lambda)$ under the assumption that $\kappa=\mathbb C(\rho,q)$ (since $\mathscr B_{r,t}(\rho,q)$ is semisimple).
We will compute $\det G_{f, \lambda}$ via \eqref{det2} in Section 4, later on.
In the remaining part of this section, we describe the actions of the generators of $\mathscr B_{r,t}(\rho, q)$ on $f_\t$ for any $\t\in \upd_{r,t}(\lambda)$.
 We  write

\begin{equation}\label{eststsk}
 f_\t g_k=\sum\limits_{\s\in{\upd_{r,t}}(\lambda)}S_{\t,\s}(k)f_\s,\ \
f_\t g^*_l=\sum\limits_{\s\in{\upd_{r,t}}(\lambda)}S_{\t,\s}(r+l)f_\s, \text{ and  $f_\t E=\sum\limits_{\s\in{\upd_{r,t}}(\lambda)}E_{\t,\s}(r)f_\s$, }\end{equation}
where $E_{\t,\s}(r), S_{\t,\s}(k),S_{\t,\s}(r+l)\in \kappa$ for all $1\le k\le r-1$, $1\le l\le t-1$.

\begin{Defn} Suppose  $\t,\s\in{\upd_{r,t}}(\lambda)$ for some  $(f,\lambda)\in\Lambda_{r,t}$. Fix $1\leq k\leq r+t$, we   write  $\t\overset{k}\sim\s$ if
  $\t_j=\s_j$ for all  $j\neq k$.
\end{Defn}

\begin{Lemma}\label{co} Suppose that $\s, \t\in{\upd_{r,t}}(\lambda)$ for some  $(f,\lambda)\in\Lambda_{r,t}$.
\begin{enumerate}
\item If $S_{\t,\s}(k)\neq0$ and $k\neq r$, then  $ \s\overset{k}\sim\t$.
\item  If $E_{\t,\s}(r)\neq0$, then $\s\overset{r}\sim\t$ and $\t_{r-1}=\t_{r+1}$.
\end{enumerate}
\end{Lemma}
\begin{proof}
By Lemma~\ref{y}  (c)(2), (d)(2) and \eqref{eigenv1}, we have $c_\s(j)=c_\t(j)$ for $j\neq k,k+1$ when $S_{\t,\s}(k)\neq0$. Applying Lemma~\ref{conditionofcontent}(a) to both  $(\s_0,\ldots,\s_{k-1})$
and $(\s_{k+1},\ldots,\s_{r+t})$, we have $\s_{i}=\t_i$ for $i\neq k$. Then (a) is proved.
If  $\t_{r-1}\neq \t_{r+1}$, then  $f_\t E= 0$. Otherwise, by Lemma~\ref{y}(e) and \eqref{eigenv1}, $c_\t(r)+c_\t(r+1)=\delta$, a contradiction with Lemma~\ref{conditionofcontent}(c). By Lemma~\ref{y}(e) and \eqref{eigenv1},
  $c_\s(j)=c_\t(j)$ for $j\neq r,r+1$ if $E_{\t,\s}(r)\neq0$, forcing $\s\overset{r}\sim\t$.\end{proof}

For any bipartition $\lambda$,
we say two nodes of $[\lambda]$ are in the same row (resp., same column) if either they are both in the same row (resp., same column) of $[\lambda^{(1)}]$ or they are both in the same row (resp., same column) of $[\lambda^{(2)}]$. If $\mu=\lambda\setminus p$ or $\mu=\lambda\cup p$, we denote   the node $p$ by $\lambda\ominus\mu$.
\begin{Lemma}\label{w} Suppose that $\t,\s\in{\upd_{r,t}}(\lambda)$ and  $\s\overset {k} \sim \t$ for some $k$ with  $k\neq r$.
\begin{enumerate}
\item If $\t_{k}\ominus\t_{k-1}$ and  $\t_{k+1}\ominus\t_{k}$  are either in the same row or in  the same column, then $\s=\t$.
\item If $\t_{k}\ominus\t_{k-1}$ and  $\t_{k+1}\ominus\t_{k}$  are neither  in the same row nor in  the same column,  then there is a unique $\s$ such that $\s\overset{k}\sim\t$ and $\s\neq \t$.
Moreover,
 $c_\t(k)=c_{\s}(k+1)$ and  $c_\t(k+1)=c_{\s}(k)$.
 \end{enumerate}
\end{Lemma}
\begin{proof} (a) follows directly from  Definition~\ref{updt}.
In case (b), $\s$ is the unique element such that $\t_{k}\ominus\t_{k-1}=\s_{k+1}\ominus \s_k$ and  $\t_{k+1}\ominus\t_{k}=\s_k\ominus\s_{k-1}$.\end{proof}
We will denote $\s$ in Lemma~\ref{w} (b)  by $\t s_k$ (resp., $\t s^*_{k-r}$) if $1\le k\le r-1$ (resp., $r+1\le k\le r+t-1$).
In Example~\ref{exofupd}, $\t_4\ominus \t_3$ and $\t_3\ominus \t_2$ are neither in the same row nor in the same column. Moreover,  $\t s^*_1=\u$ and $\u s^*_1=\t$.
The following result will be used to compute  $S_{\t,\s}(k)$ in \eqref{eststsk}.
\begin{Lemma} \label{ntgk}
Let $\t\in{\upd_{r,t}}(\lambda)$ for some $(f,\lambda)\in\Lambda_{r,t}$.
\begin{enumerate}
\item Suppose $\t s_k$ exists for some  $1\leq k<r$ and $\t s_k\prec \t$.  Then $\n_{\t}g_k=\n_{\t s_k}$.
\item Suppose $\t s^*_{k-r}$ exists for some $r+1\leq k\leq r+t-1$ and $\t s^*_{k-r}\prec\t$.
\begin{enumerate}
\item If either
$\t_{k-1}\subseteq \t_k\subseteq \t_{k+1} $ or $ \t_k\subseteq \t_{k-1}$ and $\t_k\subseteq \t_{k+1}$, then $\n_{\t}g^*_{k-r}=\n_{\t s^*_{k-r}}$.
\item If $\t_{k+1}\subseteq \t_k\subseteq \t_{k-1}$, then $\n_{\t}=\n_{\t s^*_{k-r}}g_{k-r}^*$.
\end{enumerate}
\end{enumerate}

\end{Lemma}
\begin{proof}By the definition of $\n_\t$ in Definition~\ref{definition of m}, we can assume $k=r-1$ and $t=0$ without loss of generality when we compute $\n_\t g_k$. Suppose $\t_{r}=(\lambda^{(1)},(0))$ and $\t_{r-2}=(\mu^{(1)},(0))$.
Then $\t_{r-1}=(\mu^{(1)}\cup p_1,(0))$ and $\lambda^{(1)}=\mu^{(1)}\cup{p_1}\cup {p_2}$.
 Write $p_1=(l,\lambda^{(1)}_l)$, $p_2=(h,\lambda^{(1)}_h)$. Since  $\t s_{r-1}\prec \t$, we have $h>l$
 and $(\t s_{r-1})_{r-1}=(\mu^{(1)}\cup p_2,(0))$.
By Definition~\ref{definition of m} and \eqref{htdddd}, $\n_\t=n_\lambda g_{a_h,r }g_{a_l, r-1}b_{\t_{r-2}}$ for some $b_{\t_{r-2}}\in\mathscr H_{r-2}$, where $a_j=\sum_{i=1}^j \lambda^{(1)}_i$ for $j=h$ or $l$.
Using braid relations, we have  $\n_\t g_{r-1}= n_\lambda g_{a_l, r}g_{a_h-1,r-1 }b_{\t_{r-2}}=\n_{\t s_{r-1}}$. This completes the proof of (a).

By the same  reason as above,
 we can  assume  $k=r+t-1 $  when we prove (b).
 For the first case in (i), one can check the formula  via  arguments similar to those  in the proof of (a).
The second case in (i),  one can also use Definition~\ref{definition of m} and \eqref{htdddd} to check that $\n_\t g^*_{t-1}=\n_{\t s^*_{t-1}}$.
It remains to consider that  $\t_{r+t}\subseteq \t_{r+t-1}\subseteq \t_{r+t-2}$.
In this case, $f\geq2$. Write $\t_{r+t-2}=(\mu^{(1)},\mu^{(2)})$, $\t_{r+t-1}=(\mu^{(1)}\setminus  p_1,\mu^{(2)})$ and $ \t_{r+t}=(\mu^{(1)}\setminus p_1\setminus p_2,\mu^{(2)})=(\lambda^{(1)},\lambda^{(2)})$ with $p_1=(l,\lambda_l^{(1)}+1)$, $p_2=(h,\lambda_h^{(1)}+1)$.  Since $\t s^*_{t-1} \prec \t$,  we have $h<l$ and $(\t s^*_{t-1})_{r+t-1}=(\mu^{(1)}\setminus p_2, \mu^{(2)})$. Let $a_j=\sum_{i=1}^j \lambda^{(1)}_i$.
By Definition~\ref{definition of m} and \eqref{htdddd},
\begin{equation} \label{ssss}\begin{aligned}\n_{\t s^*_{t-1} }g^*_{t-1}
 &= E_{r-f+1,f}E_{r-f+2,f-1}g_{f-1,t}^*
g^*_{f-1,t-1}g^{-1}_{a_h+1,r-f+2} g^{-1}_{a_l+2,r-f+2}\n_{\t_{r+t-2}} b_{\t_{r+t-2}} \\
 &=E_{r-f+1,f}E_{r-f+2,f-1}g_{r-f+1}g_{f,t}^*
g^*_{f-1,t-1}g^{-1}_{a_h+1,r-f+2} g^{-1}_{a_l+2,r-f+2} \n_{\t_{r+t-2}} b_{\t_{r+t-2}} \\
&=E_{r-f+1,f}E_{r-f+2,f-1}g_{f,t}^*
g^*_{f-1,t-1}g^{-1}_{a_h+1,r-f+1} g^{-1}_{a_l+2,r-f+2} \n_{\t_{r+t-2}} b_{\t_{r+t-2}} =\n_{\t },
\end{aligned}\end{equation}
where
the  second  equality follows from        $E_{r-f+1,f}E_{r-f+2,f-1}g^*_{f-1}=E_{r-f+1,f}E_{r-f+2,f-1}g_{r-f+1}$ (cf.\cite[Lemma~2.4(e)]{Rsong}).
\end{proof}

Let
$[\delta+k]:=\frac{\rho q^k-\rho^{-1}q^{-k}}{q-q^{-1}}$, for  $k\in \mathbb Z$.
For $\t\in{\upd_{r,t}}(\lambda)$,  $p_1=\t_{k}\ominus\t_{k-1}$ and $p_2= \t_{k+1}\ominus\t_{k}$, define
$$\small{\begin{aligned}
S(k)_1&=\frac{q^{c(p_1)-c(p_2)}}{[c(p_1)-c(p_2)]},~S(k)_2=\frac{[c(p_1)-c(p_2)+1][c(p_1)-c(p_2)-1]}{[c(p_1)-c(p_2)]^2},\\
~ S(k)^*_1&=\frac{[\delta- c(p_1)-c(p_2)+1][\delta-c(p_1)-c(p_2)-1]}{[\delta-c(p_1)-c(p_2)]^2},
 ~S(k)^*_2=\frac{[c(p_1)-c(p_2)+1][c(p_1)-c(p_2)-1]}{[c(p_1)-c(p_2)]^2}.
\end{aligned}}$$
\begin{Lemma}\label{sks}  Let   $\t\in{\upd_{r,t}}(\lambda)$,  $p_1=\t_{k}\ominus\t_{k-1}$ and $p_2= \t_{k+1}\ominus\t_{k}$.
\begin{enumerate}
\item Suppose $1\le k\le r-1$.
\begin{enumerate}
\item If $\t s_k$ does not exist,  then   $f_\t g_k=-q^{-1}f_\t$ (resp., $qf_\t$) when $p_1$ and $p_2$ are in the same row
 (resp., the same column).
\item If $\t s_k$ exists , then
$f_\t g_k=S_{\t,\t}(k)f_\t+S_{\t,\t s_k}(k)f_{\t s_k}, $
where $S_{\t,\t}(k)=S(k)_1 $ and
 \begin{equation} \label{sttsk1}
 S_{\t,\t s_k}(k)= \left\{   \begin{array}{ll}
                                         1, & \hbox{ if $\t s_k\prec \t $;} \\
                                         S(k)_2, & \hbox{if $\t\prec \t s_k $.}
                                       \end{array}
                                     \right.
\end{equation}
\end{enumerate}
\item Suppose $k \geq r+1$.

\begin{enumerate}
\item If  $\t s_{k-r}^*$ does not exist,
then $f_\t g^*_{k-r}=-q^{-1}f_\t$ (resp., $qf_\t$) when $p_1$ and $p_2$ are in the same row
 (resp., the same column).
\item If  $\t s_{k-r}^*$   exists, then
 $f_\t g^*_{k-r}=S_{\t,\t }(k)f_\t+S_{\t,\t s^*_{k-r}}(k)f_{\t s^*_{k-r}}, $
where $S_{\t,\t}(k)=\frac{\rho-(q-q^{-1}) c_\t(k+1)}{c_\t(k)-c_\t(k+1)}$ and
 \begin{equation} \label{sttsk3}
 S_{\t,\t s^*_{k-r}}(k)=\begin{cases}
1, & \text{  if $\t s^*_{k-r}\prec \t$, }\\
               S(k)^*_1 , & \text{ if $\t\prec  \t s^*_{k-r}$, $\t_{k-1}\subseteq\t_k$ and $ \t_{k+1}\subseteq \t_k$ ;} \\
               S(k)^*_2 , & \text{otherwise.}\\
               \end{cases}
\end{equation}
\end{enumerate}\end{enumerate}
\end{Lemma}

\begin{proof}(a) By Lemma~\ref{co}(a), we  can write $f_\t g_k=\sum_{\s\overset{k}\sim\t}S_{\t,\s}(k)f_\s$.
If $\t s_k$ does not exist, then   $p_1$ and $p_2$ are either in the same row or in the same column. By Lemma~\ref{w}(a),
  $f_\t g_k=S_{\t,\t}(k)f_\t$ and $f_\t g_k^2=S^2_{\t,\t}(k)f_\t$.
 So $S_{\t,\t}(k)\in\{q,-q^{-1}\}$.
By  Lemma~\ref{y}(c),
$S^{-1}_{\t,\t}(k)c_\t(k)=S_{\t,\t}(k)c_\t(k+1)+\rho^{-1}$.
 If $p_1$ and $p_2$ are in the same row
 (resp. the same column), by ~\eqref{content},  we have $-qc_\t(k)=-q^{-1}c_\t(k+1)+\rho^{-1}$ ( resp., $q^{-1}c_\t(k)=qc_\t(k+1)+\rho^{-1}$), forcing $S_{\t,\t}(k)=-q^{-1}$ (resp., $S_{\t,\t}(k)=q$).

Suppose that  $\t s_k$ exists. There are two cases $\t s_k\prec \t $ or $\t \prec \t s_k$ considered as follows. If $\t s_k\prec \t $,
by Lemma~\ref{y}(c)(2) and \eqref{eigenv1},  $f_\t g_kx_k= (q-q^{-1})c_\t(k) f_\t + c_\t(k+1)f_\t g_k+\rho^{-1}f_\t$. Comparing the coefficient of $f_\t$ yields that
$$S_{\t,\t}(k)=\frac{(q-q^{-1}) c_\t(k)+\rho^{-1}}{c_\t(k)-c_\t(k+1)}=\frac{q^{c(p_1)-c(p_2)}}{[c(p_1)-c(p_2)]}.$$
We remark that this  formula makes sense by Lemma~\ref{conditionofcontent}(b). By Lemma~\ref{co}(a) and Lemma~\ref{w}(b), $\s=\t s_k$ if $\s\neq \t$ and $S_{\t, \s}(k)\neq 0$.
 By Theorem~\ref{ll} and Definition~\ref{8},
\begin{equation}\label{nffn}
f_\t =\n_\t+\sum_{\t\prec \u} a_\u \n_\u,\quad \n_\v=f_\v+\sum_{\v\prec \m}b_\m f_\m.
\end{equation}
for some scalars $a_\u$ and $b_\m$.  By Lemma~\ref{ntgk}(a), $\n_\t g_k=\n_{\t s_k}$. So,
$$f_\t g_k=\n_{\t s_k}+  \sum_{\t\prec \u} c_\u f_\u g_k= f_{\t s_k }+ \sum_{\t s_k\prec  \v }d_\v f_{\v} +  \sum_{\t\prec \u} c_\u f_\u g_k$$
for some scalars $c_\u$ and $d_\v$.
If $S_{\u,\t s_k}(k)\neq 0$, then $   \u\overset {k}\sim \t s_k\overset {k} \sim \t $, forcing $\u=\t s_k$, which contradicts the assumption that $\t s_k\prec \t $.
So, $S_{\u,\t s_k}(k)= 0$ and hence $S_{\t,\t s_k}(k)=1$.
Using the equality
$g_k^2=(q-q^{-1}) g_k+1$ and switching the roles of $\t$ and $\t s_k$, one can check the formula on $S_{\t, \t s_k}(k)$   if $\t s_k$ exists and $\t\prec \t s_k$.
 % This completes the proof of (a).

(b)
By similar arguments as above, one can check the results if $\t s^*_{k-r}$ does not  exist.
In the remaining cases, $S_{\t,\t}(k)$ is obtained by previous arguments in the proof of (a) and Lemma~\ref{y}(d)(2). Since $(g^*_{k-r})^2=(q-q^{-1}) g_{k-r}^*+1$, it is enough to show that $S_{\t, \t s_{k-r}^*}(k)=1$ if $\t s^*_{k-r} \prec \t$. There are three subcases in Lemma~\ref{ntgk}(b) to deal with.
For the first two cases in Lemma~\ref{ntgk}(b)(i), we have  $\n_\t g^*_{k-r}=\n_{\t s^*_{k-r}}$.
For the third case in Lemma~\ref{ntgk}(b)(ii), we have $\n_\t g^*_{k-r}=\n_{\t s^*_{k-r}}+(q-q^{-1}) \n_\t$.
 In all these cases, one can use \eqref{nffn}  and the previous arguments on $S_{\t, \t s_k}(k)$ to obtain the results on $S_{\t, \t s_{k-r}^*}(k)$. \end{proof}

%We close this section by computing  $f_\t E$ when  $\t_{r-1}=\t_{r+1} $.

By Lemma~\ref{e6} and \eqref{eigenv1},    $\omega_r^{(a)}$ acts on each $f_\t$
as a scalar. This enables us to    define $W_r(u,\t)\in \kappa[[u,u^{-1}]]$ satisfying  $f_\t W_r(u)= W_r(u,\t)f_\t$, where $W_r(u)$ is given in Lemma~\ref{forofwu}.

\begin{Lemma}\label{equation of E}  Suppose $\s, \t , \v\in \upd_{r,t}(\lambda)$ for some $(f, \lambda)\in \Lambda_{r, t}$.

\begin{enumerate}
 \item If $E_{\t,\t}(r)\neq0$, then $E_{\t,\t}(r)=\text{Res}_{u=c_\t(r)}\frac{W_r(u,\t)}{u}$, where $ \text{Res}_{u=a}g(u)$ is the residue of the rational function  $g(u)$ at $u=a$.
\item  If $\t \overset{r}\sim\s\overset{r}\sim\v$ and $E_{\v,\v}(r)\neq0$, then $E_{\t,\s}(r)E_{\v,\v}(r)=E_{\v,\s}(r)E_{\t,\v}(r)$.
\item If $\t_{r-1}=\t_{r+1}$, $\s\overset{r}\sim\t$,  then $\langle f_\s,f_\s\rangle E_{\t,\s}(r)=\langle f_\t,f_\t\rangle E_{\s,\t}(r)$.
 \end{enumerate}

\end{Lemma}

 \begin{proof}  We remark that (a) and (b) follow from standard arguments in the proof of \cite[Lemma~3.16]{RS}. Finally, (c)
 follows from the fact that $\langle f_\s, f_\t\rangle =\delta_{\s, \t} \langle f_\t, f_\t\rangle$ and the equality
 $\langle f_\t E,f_\s\rangle=\langle f_\t,f_\s\sigma(E)\rangle=\langle f_\t,f_\s E\rangle$.
 \end{proof}
%
%The scalars $E_{\t,\t}(r)$ will be computed in Proposition~\ref{f and E} via  mixed Schur-Weyl duality between quantum general linear groups  and $\mathscr B_{r, t}(q^n, q)$ over $\mathbb C(q)$.
%\section{Explicit formula on   $E_{\t,\t}(r)$ }
The remaining part  of this section is to give explicit  formulae on  $E_{\t,\t}(r)$ via mixed Schur-Weyl duality.
Let $U_q(\mathfrak {gl}_n)$ be the quantum general linear group over $\mathbb C(q)$, where  $ \mathfrak {gl}_n$ is  the general linear Lie algebra consisting of all $n\times n$ matrices over $\mathbb C$. The  standard Cartan subalgebra $\mathfrak h$ of $\mathfrak{gl}_n$ is the $\mathbb C$-space spanned by  $\{e_{i, i}\mid  1\le i\le n\}$, where $e_{i, j}$'s are  the usual matrix units.   The dual space $\mathfrak h^*$ of $\mathfrak h$ has  basis  $\{\epsilon_i\mid 1\le i\le n\} $ such that $\epsilon_i(e_{j,j})=\delta_{i, j}$. Each $\alpha=\sum_{i=1}^n \alpha_i \epsilon_i$ is called a weight and is also denoted by $(\alpha_1,\ldots,\alpha_n)$. If $\alpha\in\mathbb Z^n$ such that $\alpha_1\geq\alpha_2\geq\ldots\geq\alpha_n$, then $\alpha$ is called an integral dominant weight.
 %It is known that the finite dimensional irreducible $U_q(\mathfrak {gl}_n)$-modules are parameterized by  all integral dominant weights. Later on,
 Let
$V_\alpha$ be the finite dimensional irreducible $U_q(\mathfrak {gl}_n)$-module with highest weight $\alpha$, where $\alpha$ is an integral dominant weight.
For example, the natural module $V:=\mathbb C(q)^n$ of $U_q(\mathfrak {gl}_n)$ is isomorphic to $V_{(1,0^{n-1})}$, whereas its linear  dual $V^*$ is isomorphic to  $V_{(0^{n-1},-1)}$.

Following \cite{KM},
 for any integral dominant weight $\alpha$, let  $\omega_\alpha(q):=s_\alpha(q^{-n+1},q^{-n+3},\ldots,q^{n-1})$,
    where $s_\alpha(u_1,u_2,\ldots,u_n)=a_{\alpha+\gamma}(u_1,u_2,\ldots,u_n)/a_\gamma(u_1,u_2,\ldots,u_n)$ (i.e., the Schur polynomial), $\gamma=(n-1,n-2,\ldots,1,0)$ and $ a_{\alpha}(u_1,u_2,\ldots,u_n)=\text{det}[u_i^{\alpha_j}]$.
  If $\alpha\in \mathbb N^n$  is an integral dominant weight, then $\alpha$ can be considered as a partition of $\sum_{i=1}^n \alpha_i$.
    For any box $p=(i,j)\in[\alpha]$, the corresponding hook-length is
 $h^{\alpha}_p=\alpha_i+\alpha'_j+1-i-j$, where $\alpha'$ is the dual partition of $\alpha$, i.e. the number of nodes in the $i$th row of $[\alpha']$ equals to the number of nodes in the $i$th column of $[\alpha]$, for all possible $i$.

 \begin{Lemma}\label{qdim}\cite[Prop.~3.8]{KM}If $\alpha\in \mathbb N^n$  is an integral dominant weight, then
  $\omega_\alpha(q)=\prod_{p\in[\alpha]}\frac{[n+c(p)]}{[h_p^\alpha]}$.
  \end{Lemma}

 %In general, for any integral  dominant weight $\alpha$, the corresponding quantum dimension $\omega_{\alpha}(q)$ satisfies
 %  $\omega_{\alpha}(q)=\omega_{\alpha+k\omega}(q)$ for  all $k\in \mathbb Z$, where $\omega=(1^n)\in \mathbb Z^n$.
%We remark that the usual dimension of $V_\alpha$ can be computed as
%$ \lim_{q\rightarrow 1} \omega_{\alpha}(q)$.

Now, we assume  $n\in \mathbb N$ such that $n>r+t$. By \cite[Theorem~6.10]{Rsong}, $\mathscr B_{r,t}(q^n,q)$ is semisimple over $\mathbb C(q)$. Moreover,
$C(f,\lambda)$ is simple for all $(f,\lambda)\in\Lambda_{r,t}$ and
$\mathscr  B_{r,t}(q^n,q)\cong \oplus_{(f,\lambda)\in\Lambda_{r,t}}\End(C(f,\lambda)).$
It is known that any trace function $tr$ on $\mathscr  B_{r,t}(q^n,q)$ is determined by its weight vector $(w_{f,\lambda})_{(f,\lambda)\in\Lambda_{r,t}}$
such that $ tr=\sum_{(f,\lambda)\in\Lambda_{r,t}} w_{f,\lambda}\chi_{f,\lambda}$, where $ w_{f,\lambda}\in \mathbb C(q)$ and $\chi_{f,\lambda}$ is the usual matrix trace on $ \End(C(f,\lambda))$.
For any $(f,\lambda)\in\Lambda_{r,t}$, let  \begin{equation}\tilde{\lambda}=(\lambda^{(1)}_1,\ldots,\lambda^{(1)}_l,0,\ldots,0,-\lambda^{(2)}_k,\ldots,-\lambda^{(2)}_1)\in\mathbb Z^n, \end{equation}
where   $l=l(\lambda^{(1)})$ and $k=l(\lambda^{(2)})$.
It is known that   $V^{r,t}:= V^{\otimes r}\otimes (V^*)^{\otimes t}$ is
  a $(U_{q}(\mathfrak {gl}_n), \mathscr B_{r,t}(q^n, q))$-bimodule and the corresponding decomposition has been given in \cite{KM}.
  Let $\eta$ be the representation of $\mathscr B_{r,t}(q^n, q) $ corresponding to $V^{r,t}$.
  Using   results on the classification of highest weight vectors of $V^{r, t}$ in \cite[Proposition~5.2, Theorem~4.13]{Rsong2}, we have  the following $(U_{q}(\mathfrak {gl}_n), \mathscr B_{r,t}(q^n, q))$-isomorphism
\begin{equation} \label{decomp123} V^{r,t}\cong \bigoplus_{(f,\lambda')\in\Lambda_{r,t}}V_{\widetilde{\lambda}}\otimes C(f,\lambda'),\end{equation}
where $\lambda':=(\lambda^{(1)'},\lambda^{(2)'})$ for any bi-partition $\lambda=(\lambda^{(1)},\lambda^{(2)} )$.
Suppose  $d =D^{\otimes r}\otimes (D^{-1})^{\otimes t}\in \End(V^{r,t})$, where  $D=\text{diag}(q^{-n+1},q^{-n+3},\ldots, q^{n-1})\in \End(V)$.
Let $\chi_{r,t}$ be the usual matrix trace on $\End(V^{r,t})$. Define a linear map  $\tau_{r,t}: \mathscr B_{r,t}(q^n, q)\rightarrow \mathbb C(q)$ such that
\begin{equation}\label{trace}
\tau_{r,t}(b):=\chi_{r,t}(d\eta(b)), \text{ for any } b\in\mathscr B_{r,t}(q^n, q).
\end{equation}
\begin{Lemma}\label{newtrace} \cite[Lemma~6.3]{KM}
 $\frac 1 {[n]^{r+t}}\tau_{r,t}$ is a trace function on $\mathscr B_{r, t}(q^n, q)$, which
satisfies the Markov property in the sense of \cite[Proposition~5.2]{KM}. The weight vector of  $\tau_{r,t}$ is
$(\omega_{\tilde{\lambda}}(q))_{(f,\lambda')\in \Lambda_{r,t}}$.
\end{Lemma}

%Later on, we denote  $\omega_{\tilde{\lambda}}(q)$ by $\omega_{{\lambda'}}(q)
%$ if $(f, \lambda)\in \Lambda_{r, t}$.
For any partition $\alpha$, let $\omega_\alpha(\rho,q)= \prod_{p\in[\alpha]}\frac{[\delta+c(p)]}{[h_p^\alpha]}$.
Thanks to Lemma~\ref{qdim}, we have   \begin{equation}\label{qdimrho1} \omega_\alpha(q^n,q)=\omega_{\alpha}(q), \text{ if $\alpha\in \mathbb N^n$ and $\alpha$ is a partition.}
\end{equation}

\begin{Prop}\label{f and E} Let $\mathscr B_{r, t}(\rho, q)$ be the quantized walled Brauer algebra over $\mathbb C(\rho, q)$ and $\t\in{\upd_{r,t}}(\lambda)$ for some  $(f,\lambda)\in\Lambda_{r,t}$. If $\t_{r-1}=\t_{r+1}$, then $E_{\t,\t}(r)=\frac{\omega_{\nu'}(\rho,q)}{\omega_{\mu'}(\rho,q)}\neq0$, where $\mu=\t_{r-1}^{(1)}$ and $\nu=\t_{r}^{(1)}$.
 \end{Prop}
 \begin{proof} We first assume that $\rho=q^n$ with $n>r+t$ and consider $\mathscr B_{r, t}(q^n, q)$ over $\mathbb C(q)$. By Corollary ~\ref{fmcell}(c),
we can assume $t=1$ when we compute $E_{\t,\t}(r)$. We claim that
\begin{equation}\label{7n}
  \tau_{r,1}(Eb)=\tau_{r,0}(b),\quad \text{for all $ b\in\mathscr H_{r}$.}
\end{equation}
In fact, if $b\in\mathscr   H_{r-1}$, \eqref{7n} is given in  \cite[Prop.~5.2(4)]{KM}.
 It is well-known that any $b\in\mathscr H_r$ is of form
$u x v$ such that  $u,v\in\mathscr H_{r-1}$ and $x\in \{1, g_{r-1}\}$. So we  need to consider the element $u g_{r-1} v$.
  By Definition~\ref{qwb}, we have $uE=Eu$, $vE=Ev$,  $E^2=[n]E$ and  $Eg_{r-1}E=q^nE$. So,
$$[n]\tau_{r,1}(Eug_{r-1}v)=\tau_{r,1}(Eug_{r-1}vE )=\tau_{r,1}(vu E g_{r-1}E)=q^n \tau_{r,1}(vu E)={q^n} \tau_{r,0}(vu )=[n]\tau_{r,0}(ug_{r-1}v), $$
 proving \eqref{7n}. We remark that   the final equality follows from \cite[Prop.~5.2(3)]{KM}.

 Let $\hat \t=(\t_0, \t_1, \ldots, \t_{r})$.
Then $F_{\hat \t}\in \mathscr H_{r}$. By general result  in \cite[Theorem~3.16]{Ma2},   $F_{\hat \t}$ is a primitive idempotent of $\mathscr H_{r}$. Thanks to \eqref{eigenv1},
$f_\u F_{\hat \t}=\delta_{\u,\hat \t}f_\u$ for any
$\u\in{\upd_{r,0}}(\gamma)$, where $\gamma\in\Lambda_{r,0}$. So, $\chi_{0,\gamma}(F_{\hat\t})=\delta_{\gamma,\nu}$.
Using
 Lemma~\ref{newtrace} yields
 \begin{equation}\label{ro0s}
\tau_{r,0}(F_{\hat \t})=\omega_{\nu'}(q).
\end{equation}
On the other hand, for any $\v\in{\upd_{r,1}}(\beta)$ with $(l,\beta)\in\Lambda_{r,1}$,
we have $f_\v E=0 $ if $\v_{r-1}\neq\v_{r+1}$(see Lemma~\ref{co}).
If $\v_{r-1}=\v_{r+1}$, then $f_\v EF_{\hat \t}=\delta_{(l,\beta),(1,\lambda)}E_{\v,\t}(r)f_\t$. So, $\chi_{l,\beta}(EF_{\hat \t})= \delta_{(l,\beta),(1,\lambda)}E_{\t,\t}(r)$.
Using Lemma~\ref{newtrace}  again yields
 \begin{equation}\label{ro0st}
\tau_{r,1}(EF_{\hat \t})= E_{\t,\t}(r)\omega_{\mu'}(q).
\end{equation}
By \eqref{7n}--\eqref{ro0st},
   $E_{\t,\t}(r)=\frac{\omega_{\nu'}(q)}{\omega_{\mu'}(q)}=\frac{\omega_{\nu'}(\rho,q)}{\omega_{\mu'}(\rho,q)} \neq0$ (see \eqref{qdimrho1}) with $\rho=q^n$.

 Finally,  we consider $\mathscr B_{r,t}(\rho, q)$ over $\mathbb C(\rho,q)$. We have formula on  $E_{\t,\t}(r)$ when $\rho= q^n$ for $n>r+t$. In general, it is a rational function of $\rho$.
    Using the Fundamental Theorem of algebra, we obtain the result over $\mathbb C(\rho,q)$.\end{proof}

\section{ Gram determinants}
The aim of this  section is to compute Gram determinants $\text{det }G_{f,\lambda}$ of $\mathscr B_{r, t}(\rho, q)$ over an arbitrary field, for all $(f, \lambda)\in \Lambda_{r, t}$. Unless otherwise stated, we assume $\kappa=\mathbb C(\rho,q)$.
By (\ref{det2}), it is enough to compute scalars $\langle f_\t, f_\t\rangle$ for all $\t\in \upd_{r,t}(\lambda)$.
Recall $\t^\lambda$ and $\t_\lambda$ are illustrated in Example~\ref{exampleoft}.
\begin{Defn}
Suppose  $\s\in{\upd_{r,t}}(\lambda)$ such that  $\s_{r+t-1}=\mu$. Define
\begin{enumerate}
\item $\hat{\s}\in{\upd_{r+t-1}}(\mu)$ such that $\hat{\s}_j=\s_j$ for $1\leq j\leq r+t-1$;
\item $\tilde{\s}\in{\upd_{r,t}}(\lambda)$ such that $\tilde{\s}_i=(\t^{\mu})_i$ for $1\leq i\leq r+t-1$ and $\tilde{\s}_{r+t}=\lambda$.
\end{enumerate}
\end{Defn}
In Example~\ref{exofupd}, $\hat\s=(\s_1,\s_2,\s_3)$ and $\tilde \s=\t$.
Motivated by \cite[Proposition~4.2]{RS}, we can get the following result.

\begin{Prop}\label{2} Suppose  $\t\in{\upd_{r,t}}(\lambda)$ for some $(f, \lambda)\in \Lambda_{r, t}$ such that   $\t_{r+t-1}=\mu$  with  $(l,\mu)\in\Lambda_{r+ t-1}$. Then
$\langle f_\t,f_\t\rangle=\frac{1}{\langle f_{\t^\mu}, f_{\t^\mu}\rangle}\langle f_{\hat{\t}},f_{\hat{\t}}\rangle\langle f_{\tilde{\t}},f_{\tilde{\t}}\rangle$.
\end{Prop}
From Proposition \ref{2}, one can compute  $\langle f_\t, f_\t\rangle$ recursively by computing  ${\langle f_{\t^\mu}, f_{\t^\mu}\rangle}$ and $\langle f_{\tilde{\t}},f_{\tilde{\t}}\rangle$.
\begin{Lemma}\label{down}
Suppose  $(f,\mu)\in\Lambda_{r,t}$ and $(\t^\mu)_{r+t-1}=\nu$.  If $\nu=\mu \cup p$ with $p=(1,\mu_1^{(1)}+1)\in\mathscr A(\mu^{(1)})$, then
\begin{equation} \label{n1}n_\mu b_{\t^\mu} \sigma (b_{\t^\mu} ) n_\mu \equiv \langle f_{\t^\nu},f_{\t^\nu}\rangle (q^{-\mu_1^{(1)}}[\delta-\mu_1^{(1)}] f_{\t_\mu}+\sum_{\v\neq \t_\mu} a_\v f_\v )  \quad (\text{mod }\mathscr B_{r,t}^{\rhd(f,\mu)}),
\end{equation}
where $\sigma$ is given in Lemma~\ref{three isom}(d) and  $a_\v$'s are some scalars in $\mathbb C(\rho,q)$.
\end{Lemma}
\begin{proof} The proof is elaborated in Section~\ref{proof of lemma 5.3}.
\end{proof}
Let $\lceil n\rceil:=1+q^{-2}+\ldots+q^{-(2n-2)}=q^{-n+1}[n]$ if $n\in\mathbb Z_{>0}$ and $\lceil 0\rceil=1$.
\begin{Lemma}\label{lemmaoftlambda}
Suppose $(f,\mu)\in\Lambda_{r,t}$ and $(\t^\mu)_{r+t-1}=\nu$.  Then $\langle f_{\t^\mu},f_{\t^\mu}\rangle=  \langle f_{\t^\nu},f_{\t^\nu} \rangle A$, where
\begin{equation} \label{m}  A= \begin{cases}\lceil\mu_k^{(1)}\rceil, & \text{ if $\mu=\nu \cup p$ with $p=(k,\mu_k^{(1)})\in \mathscr R(\mu^{(1)})$,}\\
\lceil\mu_k^{(2)}\rceil, & \text{ if $\mu=\nu \cup p$ with $p=(k,\mu_k^{(2)})\in\mathscr R(\mu^{(2)})$,}\\
q^{-\mu_1^{(1)}}[\delta-\mu_1^{(1)}] , & \text{ if $\nu=\mu \cup p$ with $p=(1,\mu_1^{(1)}+1)\in\mathscr A(\mu^{(1)})$.}
\end{cases}
\end{equation}

\end{Lemma}
\begin{proof}By Theorem~\ref{ll}, $\n_{\t^\mu}=f_{\t^\mu}$ since $\t^\mu$ is maximal with respect to $\prec$.
If $\mu=\nu \cup p$ with $p=(k,\mu_k^{(2)})\in\mathscr R(\mu^{(2)})$, then  by Definition~\ref{definition of m},   $\n_{\t^\mu}=\sum_{j=b_{k-1}+1}^{t} (-q)^{j-t}g^*_{j,t}\n_{\t^\nu}$, where $b_{k-1}=f+\sum _{j=1}^{k-1}\mu^{(2)}_j$.
 Using  Theorem ~\ref{mcell}(c) and \eqref{btlambda}, we have
$$\begin{aligned}\langle f_{\t^\mu},f_{\t^\mu}\rangle n_\mu &\equiv n_\mu b_{\t^\mu} \sigma (b_{\t^\mu} ) n_\mu \equiv \langle f_{\t^\nu},f_{\t^\nu}\rangle \sum_{j=b_{k-1}+1}^{t} (-q)^{j-t}g^*_{j,t} n_\nu\sum_{j=b_{k-1}+1}^{t} (-q)^{j-t}g^*_{t,j} \\
&\equiv \langle f_{\t^\nu},f_{\t^\nu}\rangle n_\mu  \sum_{j=b_{k-1}+1}^{t} (-q)^{j-t}g^*_{t,j}
\equiv  \lceil \mu_k^{(2)} \rceil\langle f_{\t^\nu},f_{\t^\nu}\rangle n_\mu(\text{mod }\mathscr B_{r,t}^{\rhd(f,\mu)}) ,
\end{aligned}
$$
where the third equality can be computed directly by Definition~\ref{qwb}(see also \eqref{xyxyx}).  The result for  $\mu=\nu \cup p$ with $p=(k,\mu_k^{(1)})\in\mathscr R(\mu^{(1)})$ can be proved similarly as above.

For the third case in \eqref{m},  by   Lemma~ \ref{down} we have
$$\begin{aligned} \langle f_{\t^\mu},f_{\t^\mu}\rangle f_{\t_\mu\t_\mu} &\equiv F_{\t_\mu} n_\mu b_{\t^\mu} \sigma (b_{\t^\mu} ) n_\mu  F_{\t_\mu}
\equiv \langle f_{\t^\nu},f_{\t^\nu}\rangle F_{\t_\mu}(q^{-\mu_1^{(1)}}[\delta-\mu_1^{(1)}] f_{\t_\mu}+\sum_{\v\neq \t_\mu} a_\v f_\v ) F_{\t_\mu} \\&\equiv \langle f_{\t^\nu},f_{\t^\nu}\rangle  q^{-\mu_1^{(1)}}[\delta-\mu_1^{(1)}]  f_{\t_\mu\t_\mu} (\text{mod }\mathscr B_{r,t}^{\rhd(f,\mu)}).
\end{aligned}$$
\end{proof}
If $ \nu=(\nu_1,\nu_2,\ldots,\nu_k)$ is a partition, let $\lceil\nu\rceil!:=\lceil\nu_1\rceil!\lceil\nu_2\rceil!\ldots\lceil\nu_k\rceil!$, where $\lceil n\rceil!=\lceil n\rceil\lceil n-1\rceil\ldots\lceil2\rceil\lceil1\rceil$.
 For any bi-partition $\lambda$,  let $\lceil\lambda\rceil!:=\lceil\lambda^{(1)}\rceil! \lceil\lambda^{(2)}\rceil!$. By Lemma~\ref{lemmaoftlambda} and induction on $r+t$, we  get the explicit formula of ${\langle f_{\t^\mu}, f_{\t^\mu}\rangle}$ as follows.
\begin{Lemma}\label{tmumaximal}
 Suppose that $(f,\mu)\in\Lambda_{r,t}$. Let $\tilde\mu:=(\tilde\mu^{(1)},  \mu^{(2)})\in \Lambda^+_{r,t-f}$ such that $\tilde \mu^{(1)}_1=\mu_1^{(1)}+f$ and $\tilde\mu^{(1)}_i=\mu^{(1)}_i$ for all $i>1$. Then \begin{equation}\label{tmu}
\langle f_{\t^\mu},f_{\t^\mu}\rangle= \lceil\tilde \mu\rceil! \prod_{i=1}^f q^{-(\mu_1^{(1)}+i-1)}[\delta -(\mu_1^{(1)}+i-1)].
\end{equation}
\end{Lemma}

\begin{Lemma}\label{9} Let $\t\in{\upd_{r,t}}(\lambda)$ with $(f,\lambda)\in\Lambda_{r,t}$.
\begin{enumerate}
\item If $\t s_k$ exists and  $\t s_k\prec \t$ for some  $1\leq k\leq r-1$,
then $\langle f_{\t s_k},f_{\t s_k}\rangle=S_{\t s_k,\t}(k)\langle f_{\t},f_{\t}\rangle$;
\item If $\t s_k^*$ exists and  $\t s_k^*\prec\t$ for some  $1\leq k\leq t-1$, then $\langle f_{\t s^*_k},f_{\t s^*_k}\rangle=S_{\t s_k^*,\t}(r+k)\langle f_{\t},f_{\t}\rangle$.
\end{enumerate}
where $S_{\t s_k,\t }(k)$ (resp., $S_{\t s_k^*,\t }(r+k)$) is given in (\ref{sttsk1})(resp., (\ref{sttsk3})).
\end{Lemma}

\begin{proof}
The result $(a)$ (resp., $(b)$) follows from Lemma~\ref{sks}, $\langle f_{\t},f_{\t}g^2_k\rangle=\langle f_{\t}g_k,f_{\t}g_k\rangle$ (resp., $\langle f_{\t},f_{\t}(g_k^*)^2\rangle=\langle f_{\t}g_k^*,f_{\t}g_k^*\rangle$) and $\langle f_{\t},f_{\s}\rangle=0$ for $\t\neq \s$.
\end{proof}
%
%\begin{Defn}\cite[Definition~4.4]{RS}
%Let $\lambda$ be any partition, given a removable (resp. an addable) node $p=(k,\lambda_k)$ (resp. $(k,\lambda_{k}+1)$) of $\lambda$, define
%\begin{enumerate}
%\item $\mathscr R(\lambda)^{<p}=\{(l,\lambda_l)\in\mathscr R(\lambda)\mid l>k \}$.
%\item  $\mathscr A(\lambda)^{<p}=\{(l,\lambda_l+1)\in\mathscr A(\lambda)\mid l>k \}$.
%\end{enumerate}
%\end{Defn}
%Recall we have the notation $\mu\rightarrow\lambda$ in section~2.
\begin{Defn}\label{gamahhsu}
Suppose  $(l,\mu)\in\Lambda_{r+t-1},(f,\lambda)\in\Lambda_{r,t}$ such that $\mu\rightarrow\lambda$. For any $\t\in{\upd_{r,t}}(\lambda)$
with  $\hat{\t}=\t^\mu$, let $\gamma_{\lambda/\mu}\in \mathbb C(\rho,q)$ be defined by  $\gamma_{\lambda/\mu}:=\frac{\langle f_{\t},f_{\t}\rangle}{\langle f_{\t^\mu},f_{\t^\mu}\rangle  }$.
\end{Defn}
\begin{Prop}\label{15}Let $\t\in{\upd_{r,t}}(\lambda)$ such that  $\hat{\t}=\t^\mu$ with  $(f,\lambda)\in\Lambda_{r,t}$ and $\mu=\lambda\setminus p$.
\begin{enumerate}
\item If $t=0$, then  $p=(k,\lambda_k^{(1)})\in\mathscr R(\lambda^{(1)})$ for some $k$ and
$\gamma_{\lambda/\mu}= \lceil\lambda_k^{(1)}\rceil \prod_{j=a_k}^{r-1}S_{\t^\lambda s_{a_k,j+1} ,\t^\lambda s_{a_k,j}}(j)$,
where  $a_k=\sum_{i=1}^k\lambda_i^{(1)}$ and  $S_{\t^\lambda s_{a_k,j+1} ,\t^\lambda s_{a_k,j}}(j)$ is given in (\ref{sttsk1}).
\item If $t>0$, then   $p=(k,\lambda_k^{(2)})\in\mathscr R(\lambda^{(2)})$ for some $k$ and
$\gamma_{\lambda/\mu}= \lceil\lambda_k^{(2)}\rceil  \prod_{j=b_k}^{t-1}S_{\t^\lambda s^*_{b_k,j+1},\t^\lambda s^*_{b_k,j}}(j+r)$,
 where $b_k=f+\sum_{i=1}^k\lambda_i^{(2)}$ and  $S_{\t^\lambda s^*_{b_k,j+1},\t^\lambda s^*_{b_k,j}}(j+r)$  is given in (\ref{sttsk3}).
\end{enumerate}
\end{Prop}

\begin{proof}
By assumption, $\t=\t^\lambda s_{a_k,r}$ if $t=0$. Note that $\t^\lambda s_{a_k,r}\prec\t^\lambda s_{a_k,r-1}\prec\ldots\prec\t^\lambda$. By using Lemma~\ref{9}(a) repeatedly for pairs $\{f_{\t^\lambda s_{a_k,j+1}},f_{\t^\lambda s_{a_k,j}} \}$, $a_k\leq j\leq r-1$,  we have
$\langle f_{\t},f_{\t}\rangle=\langle f_{\t^\lambda},f_{\t^\lambda}\rangle\prod_{j=a_k}^{r-1}S_{\t^\lambda s_{a_k,j+1} ,\t^\lambda s_{a_k,j}}(j)$.
 Hence (a) follows from Lemma~\ref{tmumaximal}. By similar arguments for (a), (b) follows from  Lemma~\ref{9}(b) .
\end{proof}

Now we need to consider the case that $\t\in{\upd_{r,t}}(\lambda)$ with $(f,\lambda)\in\Lambda_{r,t}$ such that  $\hat{\t}=\t^\mu$  and $\mu=\lambda\cup p$. In this case, $t\geq f>0$
 and $p\in \mathscr A(\lambda^{(1)})$.

We first assume that $t=1$ (hence $f=1$) and fix  $\t\in{\upd_{r,t}}(\lambda)$   with $\hat{\t}=\t^\mu$. Let  $p=(k,\lambda_k^{(1)}+1)\in \mathscr A(\lambda^{(1)})$ such that $\mu=\lambda\cup p$. Note that $p$ and $k$ are uniquely  determined by $\lambda$ and $\mu$ (hence by $\t$).
Let  $\u$ be the unique element in $\upd_{r,1}(\lambda)$ (see Example~\ref{exofbaru}) such that
\begin{equation}  \label{conditionofu}
 \u_{r-1}=\u_{r+1}, ~\u_r\setminus\u_{r-1}= (k,\lambda^{(1)}_k+1), \text{~~} (\u_1^{(1)},\ldots,\u_{r-1}^{(1)})=\mathbf t^{\lambda^{(1)}} \text{ and } \u_i^{(2)}=(0), \text{$0\leq i\leq r+1$}.
  \end{equation}
For any partition $\nu$ of $n$, we denote by $\bar \nu$ the partition of $n-\sum_{i=1}^{k-1}\nu_i$ such that $\bar\nu_i=\nu_{k-1+i}$, $i\in\mathbb Z^{>0}$.
Let $ (1,\bar \lambda)\in\Lambda_{r-b,1}$  such that $\bar\lambda^{(2)}=(0)$ and $\bar \lambda^{(1)}=\bar{\lambda^{(1)} }$,
where  $ b=\sum_{i=1}^{k-1}\lambda_i^{(1)}$.
 \begin{Defn}\label{defiojjdied} Keep  $\u$  in \eqref{conditionofu} and $k,b$ as above.  Define
    \begin{enumerate}
\item ${\upd_{r,1}}(\lambda)_\u:=\{\s\in {\upd_{r,1}}(\lambda)\mid \s_i=\u_i \text{ for }\, 1\leq i\leq b \text{ and } r-1\le i\le r+1\}$.
\item ${\upd_{r-1,0}}(\lambda)_\u:=\{\v\in {\upd_{r-1,0}}(\lambda)\mid \v_i=\u_i \text{ for  }1\leq i\leq b\}$.
\item For any $\s\in{\upd_{r,1}}(\lambda) _\u$,  let $\bar \s\in {\upd_{r-b,1}}(\bar\lambda)$ such that $\bar \s^{(1)}_i=\bar{\s^{(1)}_{b+i}}$ , $1\leq i\leq r-b+1$.
\item For any $\v\in {\upd_{r-1,0}}(\lambda) _\u$, let $\bar \v\in {\upd_{r-b-1,0}}(\bar\lambda)$ such that $\bar\v_i^{(1)}=\bar{\v^{(1)}_{b+i}}$, $1\leq i\leq r-b-1$.
    \end{enumerate}
\end{Defn}

\begin{example}\label{exofbaru}
Suppose that $\lambda=((3,2,2,1),(0))$, $r=9,t=1$ with $(1,\lambda)\in \Lambda_{9,1}$.
Let
$$\tiny{\t=(\t_0, \t_1, \t_2, \t_3,
(\young(\ \ \ ,\ ), \emptyset) ,(\young(\ \ \ ,\ \ ), \emptyset)    ,(\young(\ \ \ ,\ \ \ ), \emptyset)    ,(\young(\ \ \ ,\ \ \ ,\ ), \emptyset), (\young(\ \ \ ,\ \ \ ,\ \ ), \emptyset), (\young(\ \ \ ,\ \ \ ,\ \ ,\ ), \emptyset), (\young(\ \ \ ,\ \ ,\ \ ,\ ), \emptyset)),
 }$$
where $(\t_1,\t_2,\t_3)=\tiny{ ((\young(\ ),\emptyset),(\young(\ \ ),\emptyset),(\young(\ \ \ ),\emptyset))}$.
Then $\hat\t=\t^\mu$ with $\mu=((3,3,2,1),(0))$, $p=(2,3)\in \mathscr A(\lambda^{(1)})$ and $k=2$, $b=3$. So, $\bar {\lambda^{(1)}}= (2,2,1)$, $\bar\lambda=((2,2,1), (0))$  and
$$\u=\tiny{(\t_0, \t_1, \t_2,\t_3,  (\young(\ \ \ ,\ ), \emptyset) ,(\young(\ \ \ ,\ \ ), \emptyset)    ,(\young(\ \ \ ,\ \ ,\ ), \emptyset)    ,(\young(\ \ \ ,\ \ ,\ \ ), \emptyset), (\young(\ \ \ ,\ \ ,\ \ ,\ ), \emptyset), (\young(\ \ \ ,\ \ \ ,\ \ ,\ ), \emptyset), (\young(\ \ \ ,\ \ ,\ \ ,\ ), \emptyset))}.$$
In this case,  ${\upd_{9,1}}(\lambda)_\u=\{\s,\u\} $ and  ${\upd_{8,0}}(\lambda)_\u=\{\v, (\u_0,\ldots,\u_8)\}$, where
$$\begin{aligned}\s&=\tiny{(\t_0, \t_1, \t_2,\t_3,  (\young(\ \ \ ,\ ), \emptyset) ,(\young(\ \ \ ,\ ,\ ), \emptyset)    ,(\young(\ \ \ ,\ \ ,\ ), \emptyset)    ,(\young(\ \ \ ,\ \ ,\ \ ), \emptyset), (\young(\ \ \ ,\ \ ,\ \ ,\ ), \emptyset), (\young(\ \ \ ,\ \ \ ,\ \ ,\ ), \emptyset), (\young(\ \ \ ,\ \ ,\ \ ,\ ), \emptyset))},\\
\v&=\tiny{(\t_0, \t_1, \t_2,\t_3,  (\young(\ \ \ ,\ ), \emptyset) ,(\young(\ \ \ ,\ ,\ ), \emptyset)    ,(\young(\ \ \ ,\ \ ,\ ), \emptyset)    ,(\young(\ \ \ ,\ \ ,\ \ ), \emptyset), (\young(\ \ \ ,\ \ ,\ \ ,\ ), \emptyset))}.
\end{aligned}$$
Moreover,   $\bar \u, \bar \s\in  {\upd_{6,1}}(\bar\lambda)$ and $\bar \v\in {\upd_{5,0}}(\bar\lambda)$, where
$$\begin{aligned}\bar\u&=\tiny{((\emptyset,\emptyset) ,  (\young(\ ), \emptyset) ,(\young(\ \ ), \emptyset)    ,(\young(\ \ ,\ ), \emptyset)    ,(\young(\ \ ,\ \ ), \emptyset), (\young(\ \ ,\ \ ,\ ), \emptyset), (\young(\ \ \ ,\ \ ,\ ), \emptyset), (\young(\ \ ,\ \ ,\ ), \emptyset))},\\
\bar\s&=\tiny{((\emptyset,\emptyset), (\young(\ ), \emptyset) ,(\young(\ ,\ ), \emptyset)    ,(\young(\ \ ,\ ), \emptyset)    ,(\young(\ \ ,\ \ ), \emptyset), (\young(\ \ ,\ \ ,\ ), \emptyset), (\young(\ \ \ ,\ \ ,\ ), \emptyset), (\young(\ \ ,\ \ ,\ ), \emptyset))},\\
\bar\v&=\tiny{( (\emptyset,\emptyset),  (\young(\ ), \emptyset) ,(\young(\ ,\ ), \emptyset)    ,(\young(\ \ ,\ ), \emptyset)    ,(\young(\ \ ,\ \ ), \emptyset), (\young(\ \ ,\ \ ,\ ), \emptyset))}.
\end{aligned}$$

\end{example}

  % Note that $\u=\t_\lambda$ if $\lambda^{(1)}_k=0$. Write $a_{j}=\sum_{i=1}^{j}\lambda_i^{(1)}$.

\begin{Lemma}\label{skss21}
 For any  $\v\in{\upd_{r,1}}(\lambda)_\u$ (resp., $\v\in{\upd_{r-1,0}}(\lambda)_\u$),  if $\v s_{j}$ exists  for some $j$ such that  $b+1\leq j\leq r-1 $(resp., $b+1\leq j\leq r-2$), then
$S_{\v, \v}(j)=S_{\bar \v,\bar \v }(j-b)$ and $S_{\v, \v s_j}(j)=S_{\bar \v,\bar \v s_{j-b}}(j-b)$.

\end{Lemma}
\begin{proof}Let $\v_j\ominus \v_{j-1}=p_1, \v_{j+1}\ominus \v_{j}=p_2$, $\bar\v_{j-b}\ominus \bar\v_{j-b-1}=\bar p_1$  and $\bar\v_{j-b+1}\ominus \bar\v_{j-b}=\bar p_2$.
Then $c( p_i)=c(\bar p_i )-k+1$, $i=1,2$ and   $c(p_1)-c(p_2)=c(\bar p_1)-c(\bar p_2)$.
Now the result follows from \eqref{sttsk1}.
\end{proof}

%For the $\s\in {\upd_{9,1}}(\lambda)_\u  $ and  $\v\in {\upd_{8,0}}(\lambda)_\u$ in Example~\ref{exofbaru},   $ \s s_5=\u$ and $\v s_5=(\u_0,\ldots,\u_8)$.

 %\begin{Cor}Suppose $(1,\lambda)\in\Lambda_{r,1}$.
%Let $\u\in{\upd}(\lambda)$    such that $\u$ satisfies the condition in \eqref{conditionofu}. If $\u_r\ominus\u_{r-1}=(k,\lambda_k^{(1)}+1)$ with $k=1$.
%Then $\langle f_{\u},f_{\u}\rangle= \frac{\lceil\lambda\rceil !E_{\u\s}^2}{E_{\u\u}}$
% \end{Cor}

  % Let $\s\in{\upd}(\lambda)$ such that
% \begin{equation}  \label{s}
% \s\overset{r}\sim\u \text{ and }\s_r\ominus \s_{r-1}=(\ell(\lambda)+1,1).
% \end{equation}

%\begin{equation}\label{s}\s_{r}=\begin{cases}\u_{r-1}\cup p \text{ with } p=(k+1,1), &\text{ if }\lambda^{(1)}_k>0;\\
%\u_{r}, &\text{ if }\lambda^{(1)}_k=0.
%\end{cases}\end{equation}

%We can get $\langle f_{\s},f_{\s}\rangle $ by Proposition~\ref{ss}. If $\lambda_k^{(1)}>0$, then $\langle f_{\u},f_{\u}\rangle$ can be computed by following Proposition~\ref{uu}.

%\begin{Lemma}\label{ss} Suppose $\s$ is defined in (\ref{s}). Then \begin{equation}\label{s1}\langle f_{\s},f_{\s}\rangle=E_{\s\s}\langle f_{\t^{\tilde\lambda}},f_{\t^{\tilde\lambda}}\rangle\end{equation}
%where $\tilde \lambda=\lambda$ and $(f-1,\tilde \lambda)\in \Lambda_{r-1,0}$.
%\end{Lemma}

%\begin{proof}

%\end{proof}

\begin{Prop}\label{uu}
For   $\u$   in \eqref{conditionofu} and $\bar \u$ in Definition~\ref{defiojjdied}(c),   we have
 \begin{equation}\label{ffeeded1}
\small{\langle f_{\u},f_{\u}\rangle= q^{-\lambda_k^{(1)}}\frac{\lceil\lambda^{(1)}_k+1\rceil\lceil\lambda\rceil !E_{\u,\u}(r)}{E_{\bar\u, \bar\u}(r-b)}
 [\delta-\lambda_k^{(1)}]A_\u},\end{equation}
where $A_\u=\prod_{j=h}^{r-b-1}S_{\t^{\bar\lambda}s_{h,j+1}, \t^{\bar\lambda} s_{h,j}}(j)$, $h=\lambda_k^{(1)}+1$ and $E_{\u,\u}(r)$ is given in Proposition~\ref{f and E}.
\end{Prop}

\begin{proof}First note that $\bar\u\in\upd_{r-b,1}(\bar\lambda)$ and $\bar \u=\t^{\bar \lambda}s_{h,r-b}$. By Lemmas~\ref{tmumaximal} and  \ref{9},
we have
   $\langle f_{\bar\u},f_{\bar\u}\rangle=q^{-\lambda_k^{(1)}}\lceil\lambda^{(1)}_k+1\rceil\lceil \bar\lambda\rceil! [\delta-\lambda_k^{(1)}]A_\u$. So, it remains to prove that
\begin{equation}\label{ffeeded}
\small{\langle f_{\u},f_{\u}\rangle= \frac{  \lceil\lambda\rceil !E_{\u,\u}(r)}{\lceil \bar\lambda\rceil!E_{\bar\u, \bar\u}(r-b)}\langle f_{\bar\u},f_{\bar\u}\rangle}.
\end{equation}
By Definition~\ref{definition of m},
$\n_\u=Ey_\lambda g_{a,r}^{-1} \sum_{j=b+1}^{a}(-q)^{j-a}g_{a,j}g_{a,r}$, where   $a=\sum_{i=1}^k\lambda^{(1)}_i+1$.
 So,
\begin{equation}\label{fee} \small{\begin{aligned}
f_\u E &= E y_\lambda g_{a,r}^{-1} \sum_{j=b+1}^{a}(-q)^{j-a}g_{a,j}g_{a,r}F_\u E
= Ey_\lambda F_\u E + Ey_\lambda g_{a,r}g_{a-1,r-1}^{-1} \sum_{j=b+1}^{a-1}(-q)^{j-a}g_{a-1,j} F_\u E
 \\
&= [y_\lambda E F_{\u,r} F_{\u,r+1}E+ y_\lambda g_{a,r-1}E g_{r-1}F_{\u,r} F_{\u,r+1}E g_{a-1,r-1}^{-1} \sum_{j=b+1}^{a-1}(-q)^{j-a}g_{a-1,j}
]Y_\u
\end{aligned}}\end{equation}
where $Y_\u=\prod_{1\leq k\leq r-1} F_{\u,k}$ (which   commutes with $E$ by Lemma~\ref{y}(f)).

 By definition of $\u$  \eqref{conditionofu}, $ (\u_0,\ldots,\u_{r-1})$ is the maximal element in ${\upd_{r-1,0}}(\lambda) $. So,
\begin{equation}\label{ylambdanur1}
y_{\lambda}=\n_{\u_{r-1}}\equiv f_{\u_{r-1}} (\text{mod }\mathscr H_{r-1}^{\rhd(0,\lambda)}).\end{equation}
Thanks to  Lemmas~\ref{e6} and   \ref{y}(c),  there are some  $\Phi_\u, \Psi_\u\in  \mathbb C(\rho,q)[x_1,x_2,\ldots,x_{r-1}]$ such that
\begin{equation}\label{fi}
EF_{\u,r}F_{\u,r+1}E=\Phi_\u E\quad  \text{  and  }\quad
Eg_{r-1}F_{\u,r} F_{\u,r+1}E=\Psi_\u E.
\end{equation}
Using Theorem~\ref{ll} and (\ref{eigenv1}), we  can define  $ \bar{\Phi}_\u, \bar{\Psi}_\u \in \mathbb C(\rho,q)$ such that
\begin{equation}\label{faift}
\n_{\u_{r-1}}\Phi_{\u}\equiv\bar{\Phi}_\u \n_{\u_{r-1}} ,~ \n_{\u_{r-1}}\Psi_\u\equiv\bar{\Psi}_\u \n_{\u_{r-1}},~
f_\u\Phi_\u\equiv\bar{\Phi}_\u f_\u, ~ f_\u\Psi_\u\equiv\bar{\Psi}_\u f_\u,
\end{equation}
where we omit $(\text{mod }\mathscr H_{r-1}^{\rhd(0,\lambda)})$ in the first two equations and  $(\text{mod }\mathscr B_{r,1}(\rho,q)^{\rhd(1,\lambda)})$ in the last two equations.
%\begin{equation}\label{faimt}\n_{\u_{r-1}} E F_{\u,r} F_{\u,r+1}E y
%\equiv \bar{\Phi}_\u \n_{\u_{r-1}} E y(\text{mod }\mathscr B_{r,t}^{\rhd(f,\lambda)})\equiv \bar{\Phi}_\u \n_\s y (\text{mod }\mathscr B_{r,t}^{\rhd(f,\lambda)}),
%\end{equation}
%where $\n_\s=\n_{\u_{r-1}} E $ and $\s$ is defined in (\ref{s}).
%In fact, by (\ref{eigenv1}) and (\ref{faimt}), we have
%\begin{equation}\label{faift}f_\u\Phi_\u=\bar{\Phi}_\u f_\u,  f_\u\Psi_\u=\bar{\Psi}_\u f_\u.
%\end{equation}
On the other hand, by (\ref{eigenv1}), we have $f_\v F_{\u,r}F_{\u,r+1}=0$ if $\v\overset{r}\sim\u,\v\neq\u$. We have
\begin{equation}\label{phiu} \small{
\bar\Phi_\u f_\u E\overset{\text{(\ref{fi})--(\ref{faift})}}\equiv f_\u EF_{\u,r}F_{\u,r+1}E
\overset{\text{Lem.~\ref{co}}}\equiv\sum_{\v\overset{r}\sim\u}E_{\u,\v}(r)f_\v F_{\u,r}F_{\u,r+1}E
\equiv E_{\u,\u}(r)f_\u E\quad (\text{mod }\mathscr B_{r,1}^{\rhd(1,\lambda)}).
}\end{equation}
By Proposition~\ref{f and E},  $E_{\u,\u}(r)\neq 0$. So, $f_\u E\neq 0$ and $E_{\u,\u}(r)=\bar{\Phi}_\u$.
Moreover,
\begin{equation}\label{eeetlam}
y_\lambda E F_{\u,r} F_{\u,r+1}EY_\u\equiv E_{\u,\u}(r) y_\lambda E\equiv E_{\u,\u}(r)f_{\t_\lambda}
+\sum_{\t_\lambda\prec \t} p_\t f_\t ~(\text{mod }\mathscr B_{r,1}^{\rhd(1,\lambda)}),
\end{equation}
where $p_\t$ are some scalars.

If $\t=\t_\lambda$, then   $\u=\t_\lambda$ by  \eqref{conditionofu}.
Applying \eqref{fee}--\eqref{eeetlam} to $\u=\t_\lambda$ and using  Definition \ref{definition of m}, we have
$$\begin{aligned}
f_{\t^\lambda\t_\lambda}f_{\t_\lambda\t^\lambda}=&F_{\t^\lambda}\n_{\t^\lambda\t_\lambda}F_{\t_\lambda}\n_{\t_\lambda\t^\lambda}F_{\t^\lambda}
\equiv F_{\t^\lambda}b_{\t^\lambda}  E\n_{\u_{r-1}} F_{\t_\lambda,r}F_{\t_\lambda,r+1}\n_{\u_{r-1}} E  b_{\t^\lambda}F_{\t^\lambda}\\
\equiv&F_{\t^\lambda}b_{\t^\lambda}  \n_{\u_{r-1}} E F_{\t_\lambda,r}F_{\t_\lambda,r+1}E\n_{\u_{r-1}} b_{\t^\lambda}F_{\t^\lambda}
\equiv F_{\t^\lambda}b_{\t^\lambda} \n_{\u_{r-1}} \Phi_{\t_\lambda} E \n_{\u_{r-1}} b_{\t^\lambda}F_{\t^\lambda}\\
\equiv&\bar{\Phi}_{\t_\lambda}F_{\t^\lambda}b_{\t^\lambda} E\n_{\u_{r-1}}\n_{\u_{r-1}} b_{\t^\lambda}F_{\t^\lambda} \equiv \lceil \lambda\rceil !E_{\t_\lambda,\t_\lambda}(r) f_{\t^\lambda \t^\lambda} \quad (\text{mod }\mathscr B_{r,1}^{\rhd(1,\lambda)}).
\end{aligned}
$$
%where $\Phi_\s\in F[x_1, x_2,\ldots,x_{r-1}]$. Let $\bar{\Phi_\s}$ be defined as (\ref{faimt}). Then
%$$\begin{aligned}F_{\t^\lambda} b_{\t^\lambda}\n_{\tilde{\lambda}} \Phi_\s E
%=.
%\end{aligned}
%$$
%By Lemma~\ref{ff} and Lemma~\ref{co}, for $\v\in{\upd}(\lambda), \v_{2s}=\u_{2s},\v_{2s-1}\rhd\lambda$,
%$$F_{\t^\lambda} f_\v s^*_{s,f}s_{r,f}F_{\t^\lambda}\equiv0 \quad (\text{mod }\mathscr B_{r,t}^{\rhd(f,\lambda)}).$$
%Therefore,
%$$\begin{aligned}f_{\t^\lambda\s}f_{\s\t^\lambda}\equiv&  \bar{\Phi_\s}F_{\t^\lambda}b_{\t^\lambda} E \n_{\tilde{\lambda}} \n_{\tilde{\lambda}} b_{\t^\lambda} F_{\t^\lambda}\quad (\text{mod }\mathscr B_{r,t}^{\rhd(f,\lambda)})\\
%\equiv&\langle f_{\t^{\tilde\lambda}},f_{\t^{\tilde\lambda}}\rangle\bar{\Phi_\s}F_{\t^\lambda}\m_\lambda b_{\t^\lambda} F_{\t^\lambda}\quad (\text{mod }\mathscr B_{r,t}^{\rhd(f,\lambda)})\\
%\equiv&\bar{\Phi_\s}\langle f_{\t^{\tilde\lambda}},f_{\t^{\tilde\lambda}}\rangle f_{\t^\lambda}\quad (\text{mod }\mathscr B_{r,t}^{\rhd(f,\lambda)}),
%\end{aligned}
%$$
%Note that $\bar{\Phi_\s}=E_{\s\s}$ is proved similarly by (\ref{fi})--(\ref{phiu}).
This yields  \begin{equation}\label{s1}\langle f_{\t_\lambda},f_{\t_\lambda}\rangle=\lceil\lambda\rceil!E_{\t_\lambda,\t_\lambda}(r).\end{equation}
%and $\langle f_{\t^{\tilde\lambda}},f_{\t^{\tilde\lambda}}\rangle$ can be obtained by Lemma~\ref{tmumaximal}.
Note that for any  $\u$ in \eqref{conditionofu}, we have $\u \overset{r}\sim \t_\lambda$.
By Lemma~\ref{equation of E}(b)--(c) and \eqref{s1},
 \begin{equation}\label{ff}\langle f_{\u},f_{\u}\rangle= \langle f_{\t_\lambda},f_{\t_\lambda}\rangle\frac{E_{\u,\t_\lambda}(r)^2}{E_{\u,\u}(r)E_{\t_\lambda,\t_\lambda}(r)}=\frac{\lceil\lambda\rceil!E_{\u,\t_\lambda}(r)^2}{E_{\u,\u}(r)}.\end{equation}
Similarly, we have  \begin{equation} \langle f_{\bar\u},f_{\bar\u}\rangle= \langle f_{\t_{\bar\lambda}},f_{\t_{\bar\lambda}}\rangle\frac{E_{\bar\u,\t_{\bar\lambda}}(r-b)^2}{E_{\bar\u,\bar\u}(r-b)E_{\t_{\bar\lambda},\t_{\bar\lambda}}(r-b)}
=\frac{\lceil\bar\lambda\rceil!E_{\bar\u,\t_{\bar\lambda}
}(r-b)^2}{E_{\bar\u,\bar\u}(r-b)}.\end{equation}

We claim that  \begin{equation}\label{emededddd}
E_{\u, \t_\lambda}(r)=E_{\bar\u,\t_{\bar\lambda}}(r-b)\frac{E_{\u, \u}(r)}{E_{\bar\u, \bar\u}(r-b)}.
\end{equation}
Note that \eqref{ffeeded} follows from  \eqref{ff}--\eqref{emededddd}.
%If $\u\neq \s$ and $k=\ell(\lambda)$, then $a=r$, $\lambda_k^{(1)}>0$ and   $ \u_r \ominus \u_{r-1}$ and $\u_{r-1}\ominus \u_{r-2}$ are in the some row. Hence $f_\u g_{r-1}=(-q)^{-1}f_\u$ by  Lemma~\ref{sks}(a). By (\ref{fi}) and using $ \Psi_\u$ instead of $\Phi_\u$ in \eqref{phiu}, we have   $\bar{\Psi_\u} f_\u E\equiv -q^{-1} E_{\u\u}f_\u E\text{ (mod }\mathscr B_{r,t}^{\rhd(f,\lambda)}).$
%Then  $\bar{\Psi}_\u=-q^{-1} E_{\u\u}$. So,  $E_{\u\s}=E_{\u\u}+\sum_{j=b}^{r-1} (-q)^{-2(r-j)}E_{\u\u} =\lceil\lambda_k^{(1)}+1\rceil E_{\u\u}$.
It remains to  prove the claim.

 If $\t=\t_\lambda$ (hence $\u=\t_\lambda$), then  \eqref{emededddd} follows from  \eqref{s1}. Now, we assume that $\t\neq\t_\lambda$ and hence $\u\neq\t_\lambda$.
 %In this case, it is clear that $k\leq \ell(\lambda^{(1)})$.
% is known that $f_\u E=\sum_{\v\overset{r}\sim\t}E_{\u\v}f_\v$ by Lemma~\ref{co}(c).
%In order to see $E_{\u\s}$, we will consider the term
Let \begin{equation}\label{x}
\small{X: =\small{g_{a-1,r-1}^{-1} \sum_{j=b+1}^{a-1}(-q)^{j-a}g_{a-1,j} }, \text{ and }
\quad X_b: =\small{g_{a-1-b,r-1-b}^{-1} \sum_{j=1}^{a-1-b}(-q)^{j-a+b}g_{a-1-b,j}} }.\end{equation} Then,
$$
 y_\lambda g_{a,r-1}E g_{r-1}F_{\u,r} F_{\u,r+1}E X
 \overset{\eqref{fi}}= Ey_\lambda g_{a,r-1} \Psi_\u X
\overset{\eqref {ylambdanur1}}\equiv E f_{\u_{r-1}}g_{a,r-1} \Psi_\u X  ~~(\text{mod }\mathscr B_{r,1}^{\rhd(1,\lambda)}).
 $$
 Similarly, we have
 $$y_{\bar\lambda} g_{a-b,r-1-b}E_{r-b,1} g_{r-1-b}F_{\u,r-b} F_{\u,r+1-b}E_{r-b,1} X_b
\equiv E_{r-b,1} f_{\bar \u_{r-1}}g_{a-b,r-1-b} \Psi_{\bar u} X_b  ~~(\text{mod }\mathscr B_{r-b,1}^{\rhd(1,\bar{\lambda})}). $$
 By Lemma~\ref{sks}, we can  write
 \begin{equation}\label{emiss}
  f_{\u_{r-1}}g_{a,r-1}=\sum_{\s\in{\upd_{r-1,0}}(\lambda)_\u}a_\s f_\s
 \end{equation} for some scalars $a_\s$.
Thanks to Lemma~\ref{sks} and  Lemma~\ref{skss21}, we have
  \begin{equation}\label{reawiahdr}
   f_{\bar \u_{r-1}}g_{a-b,r-1-b}=\sum_{\s\in{\upd_{r-1,0}}(\lambda)_{\u}}a_\s f_{\bar\s}
  \end{equation}
with the same scalars $a_\s$'s as that in \eqref{emiss}.

For any $\s\in{\upd_{r-1,0}}(\lambda)_\u$, let $\check{\s}\in{\upd_{r,1}}(\lambda)_\u$ such that $\check\s_i=\s_i$ for $1\leq i\leq r-1$.
 Write  $f_\s\Psi_\u=  \Psi_{\u,\s}f_\s  $ and $f_{\check \s}\Psi_\u=  \Psi_{\u,\check\s}f_{\check \s} $ for some scalars $ \Psi_{\u,\s}, \Psi_{\u,\check\s}$. Then by (\ref{eigenv1}), $ \Psi_{\u,\s}= \Psi_{\u,\check\s}$. Hence,
 $$ \begin{aligned}
\Psi_{\u,\check \s} f_{\check \s}E \equiv& f_{\check\s} E g_{r-1}F_{\u,r}F_{\u,r+1} E
\equiv \sum_{\v\overset{r}\sim\check\s}E_{\check\s,\v} (r)f_{\v}g_{r-1}F_{\u,r}F_{\u,r+1}E\\
\equiv& \sum_{\v\overset{r}\sim\check\s} E_{\check\s,\v}(r)  (S_{\v,\v}(r-1) f_{\v}+S_{\v,\v s_{r-1}}(r-1) f_{\v s_{r-1}})F_{\u,r}F_{\u,r+1}E\\
\equiv&E_{\check \s, \check \s}(r)S_{ \check \s,\check \s}(r-1)  f_{\check\s}E \quad (\text{mod }\mathscr B_{r,t}^{\rhd(f,\lambda)}).
\end{aligned}$$
Then  $ \Psi_{\u,\check \s }=E_{\check \s, \check \s}(r)S_{ \check \s ,\check \s}(r-1)$.
 Since $\check\s \in{\upd_{r,1}}(\lambda)_\u $, we have $E_{\check \s, \check \s}(r)=E_{\u, \u}(r)$ by Proposition~\ref{f and E}.
 So $ \Psi_{\u, \s }=E_{\u, \u}(r) S_{ \check \s ,\check \s}(r-1)$ and  we can write
 \begin{equation}\label{emiss1} f_{\u_{r-1}}g_{a,r-1}\Psi_\u= E_{\u, \u}(r)\sum_{\s\in{\upd_{r-1,0}}(\lambda)_{\u}}S_{ \check \s ,\check \s}(r-1) a_\s f_{\s}.
 \end{equation}
Similarly, we have  $ \Psi_{\bar\u, \bar\s }=E_{\bar\u,\bar \u}(r-b) S_{ \check {\bar\s}, \check {\bar \s}}(r-1-b)$, where $ \check {\bar\s}\in{\upd_{r-b,1}}(\bar\lambda)_{\bar\u} $ such that $\check{\bar \s}_i=\bar \s_i$ for $1\leq i\leq r-b-1$ and  $  \check {\bar\s}_{i}=\bar\u_{i}$ for $i=r-b,r-b+1$. So,  $\check {\bar\s}= \bar{\check\s}$.  Moreover,  $S_{ \check {\bar\s}, \check {\bar \s}}(r-1-b)= S_{ \check \s, \check \s}(r-1) $ by Lemma~\ref{skss21}. Then
\begin{equation}\label{reawiahdr1}
   f_{\bar \u_{r-1}}g_{a-b,r-1-b}\Psi_{\bar\u}\overset{ \eqref{reawiahdr}}= E_{\bar\u, \bar\u}(r-b)\sum_{\s\in{\upd_{r-1,0}}(\lambda)_\u}S_{ \check \s \check \s}(r-1)a_\s f_{\bar\s}.
  \end{equation}
Thanks to Lemma~\ref{skss21}, for any $\s\in{\upd_{r-1,0}}(\lambda)_\u$, if $f_\s X=\sum_{\v \in{\upd_{r-1,0}}(\lambda)_\u}b_\v f_\v$ for some scalars $b_\v$,
then $f_{\bar \s}X_b=\sum_{\v\in{\upd_{r-1,0}}({\lambda})_{\u}}b_\v f_{\bar\v}$ for the same scalars $b_\v$'s, where $X$ and $X_b$ are defined in (\ref{x}).
Combining this with  \eqref{emiss1} --\eqref {reawiahdr1}, we can write
$$
\small{f_{\u_{r-1}}g_{a,r-1}\Psi_\u X=E_{\u, \u}(r)\sum_{\s\in{\upd_{r-1,0}}(\lambda)_\u}c_\s f_\s,\quad
  f_{\bar \u_{r-1}}g_{a-b,r-1-b}\Psi_{\bar\u}X_b =E_{\bar\u, \bar\u}(r-b)\sum_{\s\in{\upd_{r-1,0}}(\lambda)_{\u}}c_\s f_{\bar\s}}
$$
for some scalar $c_\s$. For any $\s\in{\upd_{r-1,0}}(\lambda)_\u$, $f_\s Y_\u\equiv \delta_{\s,\u_{r-1}}f_{\u_{r-1}} (\text{mod }\mathscr H_{r-1}^{\rhd(0,\lambda)})$ with $Y_\u=\prod_{1\leq k\leq r-1} F_{\u,k}$. Therefore,
\begin{equation}\label{eeetlam1}
\begin{aligned}
Ef_{\u_{r-1}}g_{a,r-1}\Psi_\u XY_{\u} \equiv  E E_{\u, \u}(r)& c_{\u_{r-1}} f_{\u_{r-1}}\equiv  E_{\u, \u}(r) c_{\u_{r-1}} f_{\t_\lambda} +\sum_{\t_\lambda\prec \s} d_\s f_\s (\text{mod }\mathscr B_{r,1}^{\rhd(1,\lambda)}),  \\
 E_{r-b,1} f_{\bar \u_{r-1}}g_{a-b,r-1-b}\Psi_{\bar\u}X_bY_{\bar \u}& \equiv E_{r-b,1}E_{\bar\u, \bar\u}(r-b)c_{\u_{r-1}} f_{\bar\u_{r-1}}\\
 &\equiv  E_{\bar\u, \bar\u}(r-b) c_{\u_{r-1}} f_{\t_{\bar\lambda}} +\sum_{\t_{\bar\lambda}\prec \v} d_\v f_\v (\text{mod }\mathscr B_{r-b,1}^{\rhd(1,\bar\lambda)}).
\end{aligned}\end{equation}
Combining  \eqref{eeetlam} and \eqref{eeetlam1}, we have
$ E_{\u, \t_\lambda}(r)= E_{\u, \u}(r)(1+c_{\u_{r-1}})$ and $ E_{\bar\u,\t_{\bar\lambda}}(r-b)= E_{\bar\u, \bar\u}(r-b)(1+c_{\u_{r-1}})$.
Hence, we get the claim \eqref{emededddd} and finish the proof.
\end{proof}

%$S_{\t s^*_{t,j},\t s^*_{t,j-1}}(r+j-1)$ and $S_{\t s_{t,1}^*s_{a_k+1,j},\t s_{t,1}^* s_{a_k+1, j-1}
%}(j)$ are given in (\ref{sttsk1})--(\ref{sttsk3})
\begin{Prop}\label{14}Suppose $(f,\lambda)\in\Lambda_{r,t}$ and $\t\in{\upd_{r,t}}(\lambda)$ such that $\t^\mu=\hat{\t}$ and  $\mu=\lambda\cup p$ with $p=(k,\lambda_k^{(1)}+1)\in\mathscr A(\lambda^{(1)})$.
Let $a_l=\sum_{j=1}^l\lambda^{(1)}_j$, $l\in\{k-1,k\}$ $b=a_{k-1}+f-1$ and $h=\lambda_k^{(1)}+1$.
\begin{enumerate}
\item If $k=1$, then $\t^\lambda=\t s^*_{t,f}$ and \begin{equation}\label{1}
\small{\gamma_{\lambda/\mu}=  q^{-\lambda_1^{(1)}}[\delta-\lambda_1^{(1)}]
\prod_{j=f+1}^{t}S_{\t s^*_{t,j},\t s^*_{t,j-1}}(r+j-1)};\end{equation}
\item If $k>1$, let $\v=\t s^*_{t,1}s_{a_k+1,r}$, $\u=(\v_0,\v_1,\ldots,\v_{r+1})$ and $\nu=\v_{r+1}$.
Then
\begin{equation}\label{secondcomputationoffu}
\gamma_{\lambda/\mu}=\small{ q^{-\lambda_k^{(1)}} \frac{[\delta-\lambda_k^{(1)}] E_{\u,\u}(r)}{E_{\bar\u, \bar\u}(r-b)}
A_\u B_\t},\end{equation}
where $B_\t=\prod_{j=2}^{t}S_{\t s^*_{t,j},\t s^*_{t,j-1}}(r+j-1)\prod_{j=a_k+2}^{r}S^{-1}_{\t s_{t,1}^*s_{a_k+1,j},\t s_{t,1}^* s_{a_k+1, j-1}
}(j-1)$, $A_\u=\prod_{j=h}^{r-b-1}S_{\t^{\bar\nu}s_{h,j+1}, \t^{\bar\nu} s_{h,j}}(j)$ and $\bar \u$ is given in Definition~\ref{defiojjdied}(c).
\end{enumerate}\end{Prop}

\begin{proof}

%$E_{\u\u}$ is given in Lemma~\ref{f and E}.
%%\langle f_{\t},f_{\t}\rangle=  \langle f_{\t^{\tilde\lambda}},f_{\t^{\tilde\lambda}} \rangle\lceil\lambda_k^{(1)}+1\rceil^2 E_{\u\u}\prod_{j=1}^{t-1}S_{\u s^*_{1,j+1},\u s^*_{1,j}}(r+j)\prod_{j=a_k}^{r-1}S^{-1}_{\t s_{a_k,j+1},\t s_{a_k, j}
%}(j) ,\end{equation}

Thanks to Lemma~\ref{9}(b), we have $\langle f_{\t},f_{\t}\rangle=  \langle f_{\t^\lambda},f_{\t^\lambda} \rangle\prod_{j=f+1}^{t}S_{\t s^*_{t,j},\t s^*_{t,j-1}}(r+j-1)$.
Hence  (a) follows from Lemma~\ref{tmumaximal}. For (b), by Lemma~\ref{9} again, we  have
$\langle f_{\t},f_{\t}\rangle=\langle f_{\v},f_{\v}\rangle B_\t$.
Using  Proposition~\ref{2} repeatedly, we have
$\langle f_{\v},f_{\v}\rangle=\frac{\langle f_{\t^\lambda}, f_{\t^\lambda}\rangle}{\langle f_{\t^\nu}, f_{\t^\nu}\rangle }\langle f_{\u}, f_{\u}\rangle$, where $\nu=\v_{r+1}$.
Hence
\begin{equation}\label{jdiedjeiddddd}
\langle f_{\t},f_{\t}\rangle= \frac{\langle f_{\t^\lambda}, f_{\t^\lambda}\rangle}{\langle f_{\t^\nu}, f_{\t^\nu}\rangle }\langle f_{\u}, f_{\u}\rangle B_\t.
\end{equation}
Note that $\u\in\upd_{r,1}(\nu)$ and $\u$ satisfies the conditions in \eqref{conditionofu} for $(\w_0,\w_1,\ldots,\w_{r+1})\in \upd_{r,1}(\nu)$, where $\w=\t s^*_{t,1} $.
Combining this with \eqref{jdiedjeiddddd}, Proposition~\ref{uu} and Lemma~\ref{tmumaximal} yields (b).
\end{proof}

%In generally,  $r-1>s$, compute $\langle f_{\u},f_{\u}\rangle$, for $\u=\t s_{r-1}\ldots s_s$, then use Propositions~\ref{13} and \ref{14}.

\begin{Theorem}\label{main}
Let $\mathscr B_{r,t}(\rho,q)$ be the quantized walled Brauer algebra over $R$. Let det $G_{f,\lambda}$ be the Gram determinant associated to the cell module
$C(f,\lambda)$ of $\mathscr B_{r,t}(\rho,q)$. Then
$$ \small{\text{det }G_{f,\lambda}=\prod_{\mu\rightarrow\lambda}\text{det }G_{l,\mu}\gamma_{\lambda/\mu}^{\text{dim }C(l,\mu)}}\in R.$$
Furthermore, $\gamma_{\lambda/\mu}$ (see Definition~\ref{gamahhsu}) can be  computed explicitly by Propositions~\ref{15}, \ref{14}.
\end{Theorem}

\begin{proof}We first consider $\mathscr B_{r,t}(\rho,q)$ over $\C(\rho,q)$.
So, we can use the previous results in this section together with Lemmas~\ref{f and E} and  ~\ref{sks} to compute the   Gram determinants.
Note that $\text{det }\tilde{G}_{f,\lambda}=\prod_{\t\in{\upd_{r,t}}(\lambda)}\langle f_\t,f_\t\rangle$. By Proposition~\ref{2},
$$\small{\text{det }G_{f,\lambda}=\text{det }\tilde{G}_{f,\lambda}=\prod_{\mu\rightarrow\lambda}\text{det }G_{l,\mu}\gamma_{\lambda/\mu}^{\text{dim }C(l,\mu)}}. $$
Since the Jucys-Murphy basis of $C(f,\lambda)$ is defined over $R$,  the Gram matrices associated to $C(f,\lambda)$ which are defined over $R$ are the same as that defined over $\C(\rho,q)$. We have $\text{det }G_{f,\lambda}\in R$ as required.
\end{proof}

Recall $ \text{dim }C(l,\mu)=\sharp I(l,\mu)$ (see \eqref{ifla}) for all $(l,\mu)$. Moreover, the Gram determinant given in Theorem~\ref{main} is in $R$ and $ \mathscr B_{r,t}(\rho,q)_\kappa={\mathscr B_{r,t}}(\rho,q)_R \otimes_{R}\kappa $ for any field $\kappa$. Therefore, one can get the formula for the Gram determinant over an arbitrary field $\kappa$
by specialization.
\begin{example} We give an example to compute  the Gram determinant associated to the cell module $C(1,\lambda)$ for $\mathscr B_{2,2}(\rho,q)$ with $\lambda=((1),(1))$.
 Then $\{\n_{\w}+\mathscr B_{2,2}^{\rhd(1,\lambda)}\mid \w\in\{\t , \u,\s, \v\}\}$ is the Jucys-Murphy basis of $C(1,\lambda)$, where $\t,\s,\u,\v$ are given  Example~\ref{exofupd} and $\n_\w$'s  are  given below Definition ~\ref{definition of m}. Note that $\t=\t^\lambda$ and $\s=\t_\lambda$.
 By (\ref{comofm}), we get the  Gram matrix:
$$\small{ G_{1,\lambda}=\left(
     \begin{array}{cccc}
       (-q^{-1}-q^{-3})[\delta-1] & (q^{-2}+q^{-4})\rho & -q^{-1}[\delta-1]& q^{-2}\rho  \\
       (q^{-2}+q^{-4})\rho & (1+q^{-2})(\delta-q^{-3}\rho) &q^{-2}\rho & \delta-q^{-3}\rho  \\
       -q^{-1}[\delta-1] & q^{-2}\rho& \delta& \rho \\
      q^{-2}\rho  & \delta-q^{-3}\rho & \rho &\delta+(q-q^{-1})\rho \\
           \end{array}
   \right)}
$$
So, $\text{det }G_{1,\lambda}=q^{-4}\delta^2[\delta-2][\delta+2]$ by direct computation. On the other hand, $ \langle  f_{\t^\lambda},f_{\t^\lambda}\rangle =(q^{-1}+q^{-3})[\delta-1]$.
By Propositions~\ref{15}, Lemma~\ref{sks} and Lemma~\ref{f and E}, $\langle f_{\t_\lambda}, f_{\t_\lambda}\rangle= E_{\t_\lambda,\t_\lambda} (2)=\frac{q^{-1}}{1+q^{-2}}[\delta+1] $, $\langle f_\u, f_\u\rangle = S_{\u,\u s^*_1}(3) \langle f_{\t^\lambda}, f_{\t^\lambda}\rangle  =(q^{-1}+q^{-3})\frac{\delta[\delta-2]}{[\delta-1]}$, $\langle f_\v, f_\v\rangle = S_{\v,\v s^*_1}(3) \langle f_{\t_\lambda}, f_{\t_\lambda}\rangle=\frac{q^{-1}}{1+q^{-2}} \frac{\delta[\delta+2]}{[\delta+1]} $ .
By Theorem~\ref{main}, $\text{det }G_{1,\lambda}=q^{-4}\delta^2[\delta-2][\delta+2]$.
\end{example}

\section{Blocks  }
  The aim of this section is to classify  blocks of $\mathscr{B}_{r,t}(\rho, q)$ over a field $\kappa$ under the assumption that  $e>\max\{r,t\}$ and $\rho^2=q^{2n}$, $n\in \mathbb Z$ (we keep this   assumption in this section  unless otherwise stated) and to give a criterion for a cell module of $\mathscr B_{r, t}(\rho, q)$ being equal to its simple head over an arbitrary field $\kappa$, where $e$ is the quantum characteristic of $q^2$. Recall that $\mathscr H_{r,t}$ is semisimple if  $e>\max\{r,t\}$.

 We will use  two cellular bases  of $\mathscr B_{r,t}(\rho,q)$ other than $\mathcal C$  in Theorem~ \ref{celb}. Recall the isomorphisms $\pi$ and $\Phi$   in Lemma~\ref{three isom}(c) and (e).
 Then  $\Phi(\mathcal C)$ is another cellular basis of  $\mathscr B_{r, t}(\rho, q)$.
 Let \begin{equation}x_\lambda=\sum _{w\in \mathfrak S_\lambda}q^{\ell(w)}g_w\in \mathscr H_r, ~\text{ for }\lambda\in \Lambda^+(r) \text{ and } x_{\mu}=x_{\mu^{(1)}}x_{\mu^{(2)}}\in \mathscr H_{r,t}, \text{ for } \mu\in\Lambda^+_{r,t}.
\end{equation}
Note that $\pi(x_\lambda)=y_\lambda$ and $\pi(y_\lambda)=x_\lambda$ for any $\lambda\in\Lambda^+_{r,t}$. So, $\tilde C:= \pi\circ\Phi(\mathcal C)$ is another cellular basis of $\mathscr B_{r, t}(\rho, q)$ (see also \cite{Enyang2}). In fact,  $\tilde {\mathcal C}$ is the set of elements which are obtained by replacing $y_\lambda$ with $x_\lambda$ in elements of $\Phi(\mathcal C)$ for all $(f,\lambda)\in\Lambda_{r,t}$.
Let $\tilde C(f,\lambda)$ be the corresponding  cell module of $\mathscr{B}_{r,t}(\rho, q)$ with respect to the cellular basis $\tilde{\mathcal C}$.
To simplify the notation, we also denote by $C(f,\lambda)$ the cell module with respect to the cellular basis $\Phi(\mathcal C)$,  $(f,\lambda)\in\Lambda_{r,t}$.
Moreover,
we omit $\Phi$ for the images of any factors in \eqref{cellbasis} under $\Phi$.  For example, $\Phi(E)=E_{r,t}$ (resp., $\Phi (E^{f})$)    is denoted by $E$ (resp., $E^{f}$) in this section.

%Recall that $\mathscr H_r$ is the Hecke algebra associated to the symmetric group $\mathfrak S_r$.
%We also need another cellular basis for $\mathscr H_r$ and  $\mathscr{B}_{r,t}(\rho, q)$.
%Let

Recall that for any partition $\lambda$ we have  the standard $\lambda$- tableau $\T_\lambda$ in section ~2 (e.g., Example \ref{exampleoft}(b)). Write $w_\lambda:=d(\T_\lambda)$.
For each $\lambda\in \Lambda^+(r)$, there is a classical  Specht  module $S^\lambda=x_\lambda g_{w_\lambda} y_{\lambda'} \mathscr  H_r$ (resp., dual Specht module $\tilde S^\lambda=y_\lambda g_{w_\lambda}x_{\lambda'}\mathscr H_r$), where $\lambda'$ is the dual partition of $\lambda$. Similarly, we have   $S^\lambda:=S^{\lambda^{(1)}}\otimes S^{\lambda^{(2)}}$ and $\tilde S^\lambda:=\tilde S^{\lambda^{(1)}}\otimes \tilde  S^{\lambda^{(2)}}$ for $\mathscr H_{r,t}$, $\lambda\in\Lambda^+(r)\times \Lambda^{+}(t)$.
The following two results are well-known.
\begin{Lemma}\label{sddd}\cite[Proposition~1.3]{AG}Suppose that $e>\max\{r,t\}$.
For each $\lambda\in \Lambda^+(r)\times \Lambda^{+}(t)$, there is an isomorphism of $\mathscr H_{r,t}$-modules $\phi_{r,t}:S^\lambda \rightarrow \tilde S^{\lambda'}$, where $\lambda'=(\lambda^{(1)'}, \lambda^{(2)'})$.
\end{Lemma}
\begin{Lemma}\label{cscs}
\cite[Theorem~2.9]{DR}For $(0,\lambda)\in\Lambda_{r,t}$, we have
 $C(0,\lambda')\cong S^{\lambda}$ and $\tilde C(0,\lambda')\cong \tilde S^{\lambda}$ as $\mathscr H_{r,t}$-modules.

\end{Lemma}
Let $\mathscr{B}^{f+1}_{r,t}(\rho, q)$ be the two-sided ideal of $\mathscr B_{r,t}(\rho,q)$ generated by $E^{f+1}$. Define
$$w_\lambda=w_{\lambda^{(1)}}w_{\lambda^{(2)}}\in \mathfrak S_r\times \mathfrak S^*_t \text{ and } g_{w_\lambda}=g_{w_{\lambda^{(1)}}}g^*_{w_{\lambda^{(2)}}}\in\mathscr B_{r,t}(\rho,q), \text{  for any   $(f,\lambda)\in\Lambda_{r,t}$.}$$
%where $E^f $ is the image of the previous $E^f$ in \eqref{cellbasis} under $\Phi$ (see Lemma~\ref{siopjhi})and Similarly, we have
\begin{Lemma}\label{isocell}
As $\mathscr{B}_{r,t}(\rho, q)$-modules, $C(f,\lambda)\cong \tilde C(f,\lambda')$.
\end{Lemma}
\begin{proof}By \cite[Proposition~3.8]{Rsong}, we have $C(f,\lambda)\cong W(f,\lambda)$ as right $\mathscr B_{r,t}(\rho,q)$-modules, where $$W(f,\lambda):= E^f x_{\lambda'}g_{w_{\lambda'}}y_\lambda \mathscr{B}_{r,t}(\rho, q) + \mathscr{B}^{f+1}_{r,t}(\rho, q).$$
Using the isomorphism $\pi$ in Lemma~\ref{three isom}(c), we have $ \tilde C(f,\lambda')\cong \tilde W(f,\lambda')$,
where $\tilde W(f,\lambda') $ is obtained from $W(f,\lambda)$ by replacing
 $x_{\lambda'} g_{w_{\lambda'}}y_{\lambda}$ with $y_\lambda g_{w_\lambda}x_{\lambda'}$. It remains  to prove  $W(f,\lambda)\cong \tilde W(f,\lambda')$.

By \cite[Proposition~2.9, Lemma~2.11]{Rsong}, for any $h\in\mathscr{B}_{r,t}(\rho, q)$ and $d\in \mathscr D_{r,t}^f$ we have
\begin{equation}\label{djfdd}
E^fg_d h+ \mathscr{B}^{f+1}_{r,t}(\rho, q)\in \kappa\text{-span} \{ aE^fg_z  + \mathscr{B}^{f+1}_{r,t}(\rho, q)\mid a\in \mathscr H_{r-f,t-f}, z\in\mathscr D_{r,t}^f\}.
\end{equation}
 So,
$E^f\mathscr{B}_{r,t}(\rho, q)+\mathscr{B}^{f+1}_{r,t}(\rho, q)=\kappa\text{-span}\{ \mathscr H_{r-f,t-f}E^{f}g_d+\mathscr{B}^{f+1}_{r,t}(\rho, q)\mid d\in\mathscr D^f_{r,t} \}$.
Comparing dimensions we see that
  \begin{equation}   \{bE^{f}g_d+\mathscr{B}^{f+1}_{r,t}(\rho, q)\mid b\in B_\lambda, d\in\mathscr D^f_{r,t}\} \text{ and } \{bE^{f}g_d+\mathscr{B}^{f+1}_{r,t}(\rho, q)\mid b\in \tilde B_{\lambda'}, d\in\mathscr D^f_{r,t}\}
 \end{equation}
 are bases  of $W(f,\lambda)$ and  $\tilde W(f,\lambda') $, respectively,  where $B_\lambda$ (resp., $\tilde B_{\lambda'}$) is a basis of $ x_{\lambda'} g_{w_{\lambda'}}y_{\lambda}\mathscr H_{r-f,t-f}$ (resp., $ y_{\lambda} g_{w_{\lambda}}x_{\lambda'}\mathscr H_{r-f,t-f}$).
Let $\phi: W(f,\lambda)\rightarrow \tilde W(f,\lambda')$ be the linear map such that
 $$ \phi(bE^{f}g_d+\mathscr{B}^{f+1}_{r,t}(\rho, q) )= \phi_{r-f,t-f}(b)E^fg_d +\mathscr{B}^{f+1}_{r,t}(\rho, q), \text{ for $b\in B_\lambda$, $d\in \mathscr D^f_{r,t}$}, $$
where $\phi_{r-f,t-f} $ is the isomorphism of $\mathscr H_{r-f,t-f}$-modules given in Lemma~\ref{sddd}. From \eqref{djfdd} and the fact that $\phi_{r-f,t-f}$ is an $\mathscr H_{r-f,t-f}$-homomorphism, we see that $\phi$ is a $\mathscr{B}_{r,t}(\rho, q)$-homomorphism.
Moreover, it is  an isomorphism since  $\phi$ is obviously surjective and $\text{dim }W(f,\lambda)=\text{dim }\tilde W(f,\lambda')$.
\end{proof}

 Suppose $\lambda\in \Lambda^+(r+t)$, $\mu\in \Lambda^+(r)$ and $\nu\in \Lambda^+(t)$.
Let $c_{\mu, \nu}^\lambda$ be  the multiplicity
of $S^\lambda$ in  $\Ind_{\mathscr H_{r,t}}^{\mathscr H_{r+t}}(S^\mu\otimes S^\nu)$. It is clear that $\mu, \nu\subseteq\lambda$ if $c_{\mu, \nu}^\lambda\neq 0$.
 Identifying  $C(l,\mu')$  with $M^\mu_q$ in \cite{TH} and using  \eqref{decomp123},  the following result for $\mathscr B_{r,t}(q^n,q)$ is given in \cite[Corollary~7.24]{TH} when $q$ is not a root of unity and  $n\geq r+t$.
In fact, it also holds  when  $e>\max\{r,t\}$ since    $\mathscr H_{r,t}$ is semisimple.
Moreover, it was mentioned  in \cite{TH} that this result remains true for $\mathscr{B}_{r,t}(\rho, q)$ (with two parameters $\rho$ and $q$).
We remark that Halverson \cite{TH} presented  $\mathscr{B}_{r,t}(q^n, q)$ on the set of generators $  \hat g_i$'s, $\hat g_j^*$'s and $\hat E$, where
$g_i=q^{-1/2}\hat g_i$,  $g^*_j=q^{-1/2}\hat g^*_j$ and $ E=q^{1/2(1-n)} \hat E$.

\begin{Lemma}\label{res} \cite[Corollary~7.24]{TH}  Suppose  $e>\max\{r,t\}$ and $(0,\lambda),(l,\mu)\in\Lambda_{r,t}$. Then
$$[\Res_{\mathscr H_{r,t}}C(l,\mu'):C(0,\lambda')]
=\sum_{\tau\in \Lambda^+(l)}c_{\mu^{(1)}, \tau}^{\lambda^{{(1)}}}c_{\mu^{(2)}, \tau}^{\lambda^{(2)}}, $$
where $\Res_{\mathscr H_{r,t}}C(l,\mu')$ is the restriction of  $C(l, \mu')$ to $\mathscr H_{r,t}$.
\end{Lemma}

%\begin{Cor}\label{nes}
%Suppose $(f,\lambda),(l,\mu)\in\Lambda_{r,t}$. If
%$\Hom(C(f,\lambda),C(l,\mu))\neq0$,
%then $f<l$, and  $\lambda^{(i)}\supset\mu^{(i)}$, for $i=1,2.$
%\end{Cor}
%\begin{proof}
%By \cite[Lemma~4.3(d)]{Rsong},  $\Hom(C(0,\lambda), C(l-f,\mu))\cong\Hom(C(f,\lambda), C(l,\mu))$ as $\kappa$-spaces.  Since we are assuming  $\Hom(C(f,\lambda), C(l,\mu))\neq0$,
%by  Lemma~\ref{res},  $c_{\mu^{(1)}, \tau}^{\lambda^{{(1)}}}c_{\mu^{(2)}, \tau}^{\lambda^{(2)}}\neq0$ for some $\tau\vdash l-f$, proving the result.
%\end{proof}
Let $\text{Res}^L$ (resp., $\text{Ind}^L$) be the restriction functor (resp., induction functor) from $\mathscr B_{r, t}(\rho,q)\text{-mod}$ to $ \mathscr B_{r-1, t}(\rho,q)\text{-mod}$ (resp., $\mathscr B_{r-1, t}(\rho,q)\text{-mod}$ to $\mathscr B_{r, t}(\rho,q)\text{-mod}$).
Similarly, we have two functors $\text{Res}^R$ and $\text{Ind}^R$  for the pair $(\mathscr B_{r,t}(\rho,q),  \mathscr B_{r,t-1}(\rho,q))$.
Given a partition $\lambda$, $\lambda$ is $e$-restricted if $\lambda_i-\lambda_{i+1}<e$ for all $i\geq 1$. A bi-partition $\lambda$ is called $e$-restricted if both parts are $e$-restricted.
It is proved in \cite[Theorem~5.3]{Rsong} that $C(f,\lambda)$ has a simple head denoted by $D^{f,\lambda}$ with $\lambda $ being $e$-restricted
(resp., $f<r$ and $\lambda $ being $e$-restricted) if either $\delta\neq0$ or $\delta=0$ and $r\neq t$ (resp., $\delta=0$ and $r=t$).

For any two integers $a$ and $b$, we say $a$ and $b$ are $n$-paired if $q^{2(a+b-n)}=1$, i.e., $a+b\equiv n (\text{mod } e)$.
 %The argument in proof of the following result is motivated by \cite{CVDM}.
\begin{Prop}\label{cell-weyl} Suppose $(f, \lambda), (l, \mu)\in\Lambda_{r, t}$ such that   $[C(l,\mu):D^{f,\lambda}]\ne 0 $. Then  $f\leq l$ and  $\mu\subseteq\lambda$. Moreover, there is  a pairing of the boxes in $[\lambda^{(1)}/\mu^{(1)}]$ with those in
$[\lambda^{(2)}/\mu^{(2)}]$ such that  the contents of the boxes in each pair are $n$-paired.
\end{Prop}

\begin{proof} If $[C(l,\mu):D^{f,\lambda}]\ne 0 $, then $(f, \lambda)\unlhd (l, \mu)$ and hence $f\le l$. Applying the functor $\mathcal F$ in \cite[Lemma~4.3]{Rsong}, we have    $[C(l,\mu):D^{f,\lambda}]=[C(l-f,\mu):D^{0,\lambda}]$.
Now we can assume $f=0$ and $\lambda\in\Lambda^+(r)\times \Lambda^+(t)$.

 If $f=l$, then   $[C(0,\mu):D^{0,\lambda}]\neq 0$. Since $\mathscr H_{r,t}$ is semisimple,
we have $\lambda=\mu$.  Applying Lemma~\ref{res} on $C(l-f,\mu)$ yields $\mu\subseteq \lambda$ when $f<l$. It remains to prove the second assertion when $l>f=0$.
%The second assertion  for the walled Brauer algebras was proved  in \cite[Proposition~7.1]{CVDM} via
%\cite[Corollary~3.6, 4.3, Thereom~3.3]{CVDM}.
  %Arguments in the proof of \cite[Proposition~7.1]{CVDM} can be used here smoothly  since we have the quantum analog of   \cite[Corollaries~3.6, 4.3, Thereom~3.3]{CVDM} for $\mathscr B_{r,t}(\rho, q)$.

Suppose that  $[C( l,\mu):D^{0,\lambda}]\ne 0$ and $l>0$. Then  $r\neq 0\neq t$.
Note that   $C(0,\lambda)$ is simple and $D^{0,\lambda}=C(0,\lambda)$. So, there is a submodule $M$ of $C(l,\mu)$ such that $ \Hom_{\mathscr B_{r,t}(\rho,q)}(C(0,\lambda),C(l,\mu)/M )\neq0$. Thanks to \cite[Proposition~4.7]{Rsong} and \cite[Theorem ~4.15]{Rsong},
  there is a   node $p\in\mathscr R(\lambda^{(1)})$ such that (the simple module) $C(0,\lambda)$ is contained in the head of $\text{Ind}^LC(0,\lambda\setminus p)$. So,
$$\Hom_{\mathscr B_{r,t}(\rho,q)}(\text{Ind}^LC(0,\lambda\setminus p),C(l,\mu)/M )\neq0. $$
By Frobenius reciprocity, $\Hom(C(0,\lambda\setminus p),\text{Res}^L C(l,\mu)/M )\neq0. $
Then the simple module $ C(0,\lambda\setminus p)$ is a composition factor of $\text{Res}^L C(l,\mu)$. By \cite[Theorem ~4.15]{Rsong},
\begin{equation}\label{twocasesss}
[C(l,\mu\setminus p_1): C(0,\lambda\setminus p)]\neq0, \text { or } [C(l-1,\mu\cup p_2): C(0,\lambda\setminus p)]\neq0 \end{equation}
for some $p_1\in\mathscr R(\mu^{(1)})$ or $p_2\in\mathscr A(\mu^{(2)})$.

 Thanks to \cite[Lemma~6.3]{Rsong},
 the central element $c_{r,t}$ in \eqref{cenele} acts on each cell module  $C(h,\nu)$ of $\mathscr B_{r,t}(\rho,q)$ as $c_{r+t}(\nu)$ in \eqref{crtlambda}.
In the first case of \eqref{twocasesss}, we have $\mu^{(1)}\setminus p_1\subseteq\lambda^{(1)}\setminus p$ (by induction on $r+t$) and
\begin{equation}\label{eqdelta}
c_{r+t}(\lambda)=c_{r+t}(\mu) \text{ and } c_{r+t-1}(\mu\setminus p_1)=c_{r+t-1}(\lambda\setminus p).
\end{equation}
%If $r=1=t$, then $\mu=(\emptyset,\emptyset), \lambda=((1),(1))$.
%By \eqref{eqdelta}, $\rho^2=1$, $n=0$ and the result holds.
By \eqref{eqdelta}, we have $q^{2(c(p)-c(p_1))}=1$, forcing $ c(p)=c(p_1)$. Then the result follows  from  the $n$-paired  paring between
 $[\lambda^{(1)}\setminus p/\mu^{(1)}\setminus p]$ and $[\lambda^{(2)}/\mu^{(2)}]$ (by induction on $r+t$).

 In the second case of \eqref{twocasesss}, we have $\mu^{(1)} \subseteq\lambda^{(1)}\setminus p$, $\mu^{(2)} \cup p_2 \subseteq \lambda^{(2)}$ (by induction on $r+t$)
and
\begin{equation}\label{hwue1}
c_{r+t}(\lambda)=c_{r+t}(\mu) \text{ and } c_{r+t-1}(\mu\cup p_2)=c_{r+t-1}(\lambda\setminus p).
\end{equation}
By  \eqref{hwue1}, $\rho^2=q^{2(c(p)+c(p_2))}$, forcing $c(p)+c(p_2)\equiv n (\text{mod }e)$, i.e., $c(p)$ and $c(p_2)$ are $n$-paired.
Therefore, the result follows from the $n$-paired pairing between   $[(\lambda^{(1)}\setminus p)/\mu^{(1)}]$ and $[\lambda^{(2)}/(\mu^{(2)} \cup p_2)]$.
\end{proof}

Given two (bi-)partitions $\lambda$ and $\mu$, we denote by $\lambda\cap\mu$ the partition whose corresponding Young diagram is the intersection of
 $[\lambda]$ and $[\mu]$.
 % Recall that $\delta=\frac{\rho-\rho^{-1}}{q-q^{-1}}$, $\rho^2=q^{2a}$, $a\in\Z$.

\begin{Defn}\label{balance}\cite{CVDM} For $(f,\lambda),(l,\mu)\in\Lambda_{r,t}$,
$\lambda$ and $\mu$ are $n$-balanced
if there is a pairing of boxes in $[\lambda^{(1)}/(\lambda^{(1)}\cap\mu^{(1)})]$ with those in $[\lambda^{(2)}/(\lambda^{(2)}\cap\mu^{(2)})]$
and of the boxes in $[\mu^{(1)}/(\lambda^{(1)}\cap\mu^{(1)})]$ with those in $[\mu^{(2)}/(\lambda^{(2)}\cap\mu^{(2)})]$ such that the    contents for  each pair are  $n$-paired.
\end{Defn}
If $F$ is a field such that the characteristic of $F$ is $e>\max\{r,t\}$, then the condition that $a$ and $b$ are $n$-paired is equivalent to that $a+b=n$ in $F$.
It was proved in   \cite[Corollary~7.7]{CVDM} that the $n$-balanced relation gives the block relation for the cell modules of the walled Brauer algebra over $F$. So, $n$-balanced relation is an equivalence relation on the set $\Lambda_{r,t}$.
Using Proposition~\ref{cell-weyl}  yields the following result.
\begin{Cor}\label{bl}
 If $C(f,\lambda)$ and $C(l,\mu)$ are in the same block of $\mathscr  B_{r,t}(\rho, q)$,
then $\lambda$ and $\mu$ are $n$-balanced.
\end{Cor}

%\begin{example} Suppose $\lambda=(2^3)$. Then
%\begin{equation}
% \t^{\lambda}=\young(12,34,56) \text{ , }
%\quad \quad \t_{\lambda}=\young(14,25,36)
%\end{equation}
%and  $w_\lambda=s_2s_4s_3$, $T_{1,3}=g_2g_1g^{-1}_2$, $T_{3,5}=g_4g_3g^{-1}_4$, $T_{2,4}=g_3g_2g^{-1}_3$, $T_{4,6}=g_5g_4g^{-1}_5$, $g_{w_\lambda}\n_{\lambda'}g^{-1}_{w_\lambda}=(1-q^{-1}T_{1,3}-q^{-1}T_{3,5}+q^{-2}T_{1,3}T_{3,5}+q^{-2} T_{3,5}T_{1,3}-q^{-3}T_{1,3}T_{3,5}T_{1,3})
%(1-q^{-1}T_{2,4}-q^{-1}T_{4,6}+q^{-2}T_{2,4}T_{4,6}+q^{-2} T_{4,6}T_{2,4}-q^{-3}T_{2,4}T_{4,6}T_{2,4})$.
%For $1\leq j\leq 2$, $3(j-1)+1\leq i<3j$, we have $T_{2(i-1)+1,2i+1}g_{w_\lambda}\n_{\lambda'}g^{-1}_{w_\lambda}=(-q)^{-1}g_{w_\lambda}\n_{\lambda'}g^{-1}_{w_\lambda}$.
%Note that, when $q\rightarrow1$, then $T_{2(i-1)+1,2i+1}=s_{2(i-1)+1,2i+1}$, and $ g_{w_\lambda}\n_{\lambda'}g^{-1}_{w_\lambda}=\sum_{w\in C(\t^\lambda)}(-1)^{-\ell(w)}w$, where $C(\t^\lambda)$ is the column stabilizer of $\t^\lambda$.
%\end{example}

%Moreover, $y_{\lambda'}g_{w_\lambda}^{-1} x_\lambda\neq0$.
%To simplify the notation, we also  denote $\Phi(E)=E_{r,t}$ by $E$, where $\Phi$ is given   in Lemma~\ref{siopjhi}.
If a partition $\lambda=(a,a,\ldots,a)$ with $b$ parts, we denote it by $(a^b)$. Then its Young diagram $[\lambda]$ is a rectangle and $\mathscr R(\lambda)=\{p=(b,a)\}$. Denote $\mu=\lambda\setminus p$ by $(a^{b-1},a-1)$.
For any  $(0,\lambda)\in\Lambda_{r,t}$, let  $e_\lambda:=x_\lambda g_{w_{\lambda}}y_{\lambda'}$. Recall that for simplicity of notation, we omit $\Phi$ for the images of any factors in \eqref{cellbasis} under $\Phi$. In particular, $E_{r,t}$ is also  denoted by $E$ in this section.
The following two results hold for any $\rho$.
\begin{Lemma}\label{lemeaforeee}
Suppose that $(0,\lambda), (1,\mu)\in \Lambda_{r,t}$ with $\lambda=((a^b),(c^d))$ and $ \mu=((a^{b-1},a-1), (c^{d-1}, c-1))$.
Then $E e_\lambda E\equiv  (\delta+\rho q^{a-b+c-d}[a-b+c-d ])Ee_\mu (\text{mod }\mathscr B^2_{r,t}(\rho,q))$, where $\mathscr B^2_{r,t}(\rho,q) $ is the two-sided  ideal of $\mathscr B_{r,t}(\rho,q)$ generated by $E_{r,t}E_{r-1,t-1}$.
\end{Lemma}
\begin{proof}The proof is elaborated  in Section~\ref{proofoflem6}.
\end{proof}
%\begin{Defn}Suppose $\lambda=((a^b),(c^d))\in\Lambda^+(r)\times \Lambda^+(t)$, $\mu=((a^{b-1}, a-1), (c^{d-1},c-1))$.
%Let $W_\lambda= C(1,\mu)e_{\lambda'}$ and
%\end{Defn}

\begin{Lemma}\label{subxset}
Suppose that $e>\max\{r,t\}$ and  $(0,\lambda), (1,\mu)\in \Lambda_{r,t}$ with $\lambda=((a^b),(c^d))$ and $ \mu=((a^{b-1},a-1), (c^{d-1}, c-1))$.
Let $W_\lambda= C(1,\mu)e_{\lambda'}$. Then
\begin{enumerate}
\item As $\mathscr H_{r,t}$-modules, $W_\lambda\cong C(0,\lambda)$;
\item  $ W_\lambda E\subseteq  Ey_\mu g^{-1}_{w_{\mu'}} e_{\lambda'}E\mathscr H_{r-1,t-1}(\text{mod }\mathscr B_{r,t}^{\rhd(1,\mu)}(\rho,q))$.
\end{enumerate}

\end{Lemma}
\begin{proof}Note that $\mathscr H_{r,t}$ is semisimple since $e>\max\{r,t\}$. Then as right $\mathscr H_{r,t}$-modules
$$\small{C(1,\mu)\overset {\text{Lem.~\ref{res}}}\cong \bigoplus_{\nu\in \Lambda^+_{r,t}, \mu\subseteq\nu} C(0,\nu) \overset{\text{Lem.~\ref{cscs}}}\cong \bigoplus_{\nu\in \Lambda^+_{r,t}, \mu\subseteq\nu} S^{\nu'}}.$$
It was proved in \cite{AG} that $g_{w_{\lambda'}^{-1}}e_{\lambda'}$ is an idempotent (up to a nonzero scalar) associated to $\T^{\lambda'}$ and $S^{\nu'}g_{w_{\lambda'}^{-1}}e_{\lambda'}=0$ (resp., $S^{\nu'}g_{w_{\lambda'}^{-1}}e_{\lambda'}\cong S^{\lambda'}$) if $\nu\neq \lambda$ (resp., $\nu=\lambda$).
So, $W_\lambda\cong S^{\lambda'}\cong C(0,\lambda)$.

Note that $E$ commutes with any element of $\mathscr H_{r-1,t-1}$.
Then $W_\lambda E $ and $C(1,\mu)E$ both are  $\mathscr H_{r-1,t-1}$-modules.
 Let $ \mathscr H_{r-1,t-1}^{\rhd \mu}:=\kappa\text{-span}\{\sigma(d(\S)) y_\nu d(\T)\mid \nu \in \Lambda^+_{r-1,t-1},\nu \rhd \mu, \S,\T\in\Std(\nu)\}$ and $U_\mu:=\kappa\text{-span}\{E y_\mu d(\T) + \mathscr B_{r,t}(\rho,q)^{\rhd (1,\mu)}\mid \T\in\Std(\mu)\}$.
 Since $ \kappa\text{-span}\{y_\mu d(\T) + \mathscr H_{r-1,t-1}^{\rhd \mu} \mid \T\in\Std(\mu) \}$ is the cell module $C(0,\mu)$ of $\mathscr H_{r-1,t-1}$, we see that
$U_\mu\cong C(0,\mu)$ as $\mathscr H_{r-1,t-1}$-modules.

 %we see that the right multiplication of $E$ gives an element in $\End_{\mathscr H_{r-1,t-1}}( C(1,\mu))$.
 By Definition~\ref{qwb}(i)--(l), there is an $h_w\in\mathscr H_{r-1,t-1}$ such that $Eg_wE\equiv h_wE(\text{mod} \mathscr B_{r,t}^2(\rho,q)) $ for any $w\in\mathscr D_{r,t}^1$ . So,
 \begin{equation}\label{uce}
 W_\lambda E\subset C(1,\mu)E\subseteq U_\mu \cong C(0,\mu) .
 \end{equation}
By (a), $W_\lambda\cong \bigoplus_{\nu\in \Lambda^+_{r-1,t-1}, \nu\subseteq\lambda}C(0,\nu)$ as $\mathscr H_{r-1,t-1}$-modules.
 Write $W_\lambda=W_1\oplus W_2$ such that $W_1\cong C(0,\mu)$ and $W_2\cong \bigoplus_{\nu\in \Lambda^+_{r-1,t-1},  \nu\subseteq\lambda,\nu\neq \mu}C(0,\nu)$.
 Hence $W_\lambda E=W_1E$ and $W_2E=0$ by \eqref{uce} (since $E$ commutes with $\mathscr H_{r-1,t-1}$).
 Write $Ey_\mu g^{-1}_{w_{\mu'}} e_{\lambda'}+\mathscr B_{r,t}^{\rhd(1,\mu)}(\rho,q) =w_1+w_2 $, $w_i\in W_i$, $i=1,2$. Then
  \begin{equation}\label{eiq0z}
 w_1E\equiv Ey_\mu g^{-1}_{w_{\mu'}} e_{\lambda'}E\overset{\text{ Lem.~\ref{lemeaforeee} }} \equiv(\delta+\rho q^{a-b+c-d}[a-b+c-d ])Ey_\mu g^{-1}_{w_{\mu'}}e_{\mu'} (\text{mod }\mathscr B_{r,t}^{\rhd(1,\mu)}(\rho,q)),
 \end{equation}
 which is not 0 (since  $0\neq g^{-1}_{w_{\mu'}}e_{\mu'}$ is an idempotent up to a nonzero scalar) for some  $\rho$ and $q$ such that the coefficient $\delta+\rho q^{a-b+c-d}[a-b+c-d ]\neq0$. It implies that $w_1\neq 0$. Since $W_1$ is irreducible, we have $W_1E=w_1E\mathscr H_{r-1,t-1}$
  and (b) follows by the first part of \eqref{eiq0z}.
\end{proof}
\begin{Lemma}\label{kedyf}
 Suppose that $e>\max\{r,t\}$ and $(0,\lambda), (1,\mu)\in \Lambda_{r,t}$ with $\lambda=((a^b),(c^d))$ and $ \mu=((a^{b-1},a-1), (c^{d-1}, c-1))$.
If $\rho^2= q^{2(a-b+c-d)}$, then $\Hom_{\mathscr B_{r,t}(\rho,q)}(C(0,\lambda),C(1,\mu))\neq0$.

\end{Lemma}

\begin{proof}
By Lemma~\ref{subxset}(a),  $W_\lambda\cong C(0,\lambda)$ as $\mathscr H_{r,t}$-modules. Now it remains to show that $ W_\lambda$ is a $\mathscr B_{r,t}(\rho,q)$-module, i.e.,
$W_\lambda J=0$, where $J$ is the two-sided ideal of $\mathscr B_{r,t}(\rho,q)$ generated by $E$.
By \cite[Proposition~2.9]{Rsong}, $J+\mathscr B_{r,t}^2(\rho,q)\subseteq\{h_1Eh_2 +\mathscr B_{r,t}^2(\rho,q)\mid h_1,h_2\in \mathscr H_{r,t}\}$. Therefore,
 $$\begin{aligned}
 W_\lambda J\subseteq W_\lambda \mathscr H_{r,t}E\mathscr H_{r,t}&\subseteq W_\lambda E\subseteq Ey_\mu g^{-1}_{w_{\mu'}} e_{\lambda'}E\mathscr H_{r,t} (\text{mod }\mathscr B_{r,t}^{\rhd(1,\mu)}(\rho,q))\quad  \text{by Lemma~\ref{subxset}(b) }\\
&\overset {\text{Lem.~\ref{lemeaforeee}}} \equiv (\delta+\rho q^{a-b+c-d}[a-b+c-d ])Ey_\mu g^{-1}_{w_{\mu'}}e_{\mu'}\mathscr H_{r,t}\equiv 0(\text{mod }\mathscr B_{r,t}^{\rhd(1,\mu)}(\rho,q)),
 \end{aligned}$$
where the last $``\equiv"$ follows from the fact that
$\delta+\rho q^{a-b+c-d}[a-b+c-d ]=0$ if $\rho^2= q^{2(a-b+c-d)}$.
 \end{proof}

%The following result which is motivated by \cite{CVDM}, is the key part in determining the blocks of $\mathscr B_{r, t}(\rho,q)$.

\begin{Theorem}\label{key}
Suppose $e>\max\{r,t\}$, and $(f,\lambda),(f+1,\mu)\in\Lambda_{r,t}$.
If  $\mu^{(i)}\subseteq\lambda^{(i)} $, and $[\lambda^{(i)}\setminus\mu^{(i)}]=\{p_i\}$ for $i=1,2,$ such that
$\rho^2=q^{2(c(p_1)+c(p_2))}$,
then $$\Hom_{\mathscr B_{r,t}(\rho,q)}(C(f,\lambda),C(f+1,\mu))\cong \kappa.$$
\end{Theorem}

\begin{proof}
By \cite[Lemma~4.3(e)]{Rsong}, $\Hom_{\mathscr B_{r,t}(\rho,q)} (C(f,\lambda),C(f+1,\mu)) \cong\Hom_{\mathscr B_{r,t}(\rho,q)} (C(0,\lambda),C(1,\mu))$. So, we can  assume that $f=0$.
By Lemma~\ref{res},
$\Hom_{\mathscr H_{r,t}}(C(0,\lambda),C(1,\mu))\cong k.$
It remains  to prove
\begin{equation}\label{homspace}\Hom_{\mathscr B_{r,t}(\rho,q)}(C(0,\lambda),C(1,\mu))\neq0.
\end{equation}
We prove \eqref{homspace} by induction on $r+t$.
If $r=t=1$, then $\lambda=(\varnothing,\varnothing), \mu=(1,1)$,
and $\rho^2=1,\delta=0$. Moreover,
$C(0,\mu)=\kappa\text{-span}\{1+\langle E\rangle\}$, $C(1,\lambda)=\kappa\text{-span}\{E\}$.
The $\mathscr{B}_{1,1}(\rho,q)$-homomorphism  which sends $1+\langle E\rangle$ to $E$ is an
isomorphism, proving (\ref{homspace}).

Suppose that there exists a box $p\in \mathscr R(\lambda^{(1)})$ such that $p\in[\mu^{(1)}]$ and $ p$ is at the last row of $\mu^{(1)}$.
Then $r>1$. By induction on $r+t$, we have
\begin{equation}\label{inducte}
\Hom_{\mathscr B_{r-1,t}(\rho,q)}(  C(0,(\lambda^{(1)}\setminus p,\lambda^{(2)})),C(1,(\mu^{(1)}\setminus p,\mu^{(2)})))\neq0.
\end{equation}
Since $p$ is at  the last row of $\mu^{(1)}$, $\Res^LC(1,\mu)$ has a cell filtration of cell modules for $\mathscr{B}_{r-1,t}(\rho,q)$ such that
$C(1,(\mu^{(1)}\setminus p,\mu^{(2)}))$ is a submodule of $\Res^LC(1,\mu)$
 (see \cite[Theorem~4.15]{Rsong}) and hence
 $$\Hom_{\mathscr B_{r-1,t}(\rho,q)}(C(0,(\lambda^{(1)}\setminus p,\lambda^{(2)})),\Res^LC(1,\mu))\neq0.$$
Then
$\Hom_{\mathscr B_{r,t}(\rho,q)}(\Ind^L C(0,(\lambda^{(1)}\setminus p,\lambda^{(2)})),C(1,\mu))\neq0$
by Frobenius reciprocity. By  \cite[Theorem~4.15]{Rsong} and \cite[Proposition~4.7]{Rsong}, the cell module of $\mathscr{B}_{r,t}(\rho,q)$ which appears
in the filtration of $\Ind^LC(0,(\lambda^{(1)}\setminus p,\lambda^{(2)}))$
is of form $C(0,(\lambda^{(1)}\setminus p\cup p'_1,\lambda^{(2)}))$ or $C(1,(\lambda^{(1)}\setminus p,\lambda^{(2)}\setminus p'_2))$,
where $p'_1\in\mathscr A(\lambda^{(1)}\setminus p)$ and $p'_2\in\mathscr R(\lambda^{(2)})$.
By Lemma~\ref{res}, we have
$\Hom_{\mathscr B_{r,t}(\rho,q)}(C(1,(\lambda^{(1)}\setminus p,\lambda^{(2)}\setminus p'_2)),C(1,\mu))=0$ (since $p\in[\mu^{(1)}]$)
and $\Hom_{\mathscr B_{r,t}(\rho,q)}(C(0,(\lambda^{(1)}\setminus p\cup p'_1,\lambda^{(2)})),C(1,\mu))=0$ if $p'_1\neq p$, forcing (\ref{homspace}).

Suppose that  there   exists a box $p\in \mathscr R(\lambda^{(1)})$ such that $p\in[\mu^{(1)}]$ and $ p$ is at the last column   of $\mu^{(1)}$.
Then there is a $p'\in \mathscr R(\lambda^{(1)'})$ such that $p'\in[\mu^{(1)'}]$ and $ p'$ is at the last  row of $\mu^{(1)'}$. Moreover, the dual partition of $\mu\setminus p$
is $\mu'\setminus p'$.
Using the isomorphism $\pi$ in Lemma~\ref{three isom}(c),
a similar result as \cite[Theorem~4.15]{Rsong} for the cell filtration for $\Res^L\tilde C(1,\mu')$ also holds.
 So,   $\tilde C(1,\nu )$ is a submodule of $\Res^L\tilde C(1,\mu')$, where $\nu=\mu'\setminus p'$ and $\nu'= \mu\setminus p$.
By Lemma~\ref{isocell},     $C(1,(\mu^{(1)}\setminus p,\mu^{(2)}))$ is also  a submodule of $\Res^LC(1,\mu)$. By the argument above, we have (\ref{homspace}) in this case.
Therefore, we have  (\ref{homspace}) if  there exists a box $p\in\mathscr R(\lambda^{(1)})$ such that $p\in[\mu^{(1)}]$.

If there exists a box $p\in\mathscr R(\lambda^{(2)})$ such that $p\in[\mu^{(2)}]$, then we also have (\ref{homspace}) by
using the functors $\Res^R$ and $\Ind^R$ instead of $\Res^L$ and $\Ind^L$ in above arguments.

So we only need to prove (\ref{homspace}) for the case that $p\notin[\mu^{(i)}]$, for any box $p\in\mathscr R(\lambda^{(i)})$,
$i=1,2$. In this case, there is only one box in $\mathscr R(\lambda^{(i)})$  and the Young diagram of $\lambda^{(i)}$
is  a rectangle,  $i=1,2$. We write $\lambda^{(1)}=(a^b)$ and $\lambda^{(2)}=(c^d)$, for some positive integers $a,b,c,d$ such that $a b=r$ and  $c d=t$.
Then $c(p_1)=a-b$, $c(p_2)=c-d$. So, $\rho^2=q^{2(a-b+c-d)}$.
In this case, (\ref{homspace}) follows from Lemma~\ref{kedyf}.
\end{proof}

  Given a (bi-)partition $\mu\subseteq\lambda$, we denote by $\mathscr R(\lambda/\mu)$ the set of boxes in $\mathscr R(\lambda)$
which are not in $[\mu]$.

\begin{Defn}\cite[Definition~7.4]{CVDM}
Suppose that $\mu=(\mu^{(1)},\mu^{(2)})\subseteq\lambda=(\lambda^{(1)},\lambda^{(2)})$ is a $n$-balanced pair. For each $p_i\in\mathscr R(\lambda/\mu)$,
let  $\mu^i$ be the maximal balanced sub-bipartition between $\lambda$ and
$\mu$. This is the maximal bipartition in $\lambda$ not containing $p_i$ such that
 $\lambda$ and $\mu^i$  are $n$-balanced.
 \end{Defn}

 The  explicit recursive construction of  $\mu^i$ is given in \cite[Definition~7.4, Example~7.5]{CVDM}.
 The inclusion is  a partial order on the set of all $n$-balanced sub-bipartition between $\lambda$
 and $\mu$.
 Let $\mu_\lambda=(\mu^{(1)}_\lambda,\mu^{(2)}_\lambda)$ be the  maximal balanced sub-bipartition
   between  $\lambda$
 and $\mu$. Note that by recursive construction of  $\mu^i$  given in \cite{CVDM}, we have that $\mu_\lambda\subsetneqq \lambda $ and $\mu_\lambda$ is not unique in general.
 % As the example illustrates in \cite[Example~7.5]{CVDM}, the choice will not be unique in general.

 % By the key part of this paper (Theorem~\ref{key}), Lemma~\ref{fgfunctor}, Theorem~\ref{brunch}, Corollary~\ref{ind}
% and the same arguments as that in \cite[Theorem~7.6]{CVDM}, we have the following result, with which we can determine the blocks of of $\mathscr B_{r, t}(\rho,q)$ by induction.

\begin{Cor}\label{bap}
Suppose that  $e> \max\{r,t\}$, $\rho^2=q^{2n}$, for some $n\in\mathbb Z$. If $(\mu^{(1)},\mu^{(2)})\subseteq(\lambda^{(1)},\lambda^{(2)})$ is a $n$-balanced pair,
then
$$\Hom_{\mathscr B_{r,t}(\rho,q)}(C(f,\lambda),C(l_1,\mu_\lambda))\neq0,$$
where $\mu_\lambda$ is  any maximal balanced sub-bipartition and $l_1=r-|\mu^{(1)}_\lambda|$.
\end{Cor}
\begin{proof} Note that this result is the quantum analog  of \cite[Theorem~7.6]{CVDM} for walled Brauer algebras.
The key point of the proof of this result is the case $l_1=1$ and $f=0$. This is  proved in Theorem~\ref{key}
Then everything of  the arguments in the proof of \cite[Theorem~7.6]{CVDM} applies here since we are assuming $e> \max\{r,t\}$, $\rho^2=q^{2n}$.
The only difference  is  that we need to replace their  $\delta$   by $n$.
\end{proof}

The following result follow from the same argument as that in  \cite[Corollary~7.7]{CVDM}) by using Corollary~\ref{bl} and Corollary~\ref{bap}.
\begin{Theorem}\label{qblocks}Suppose $e>\max\{r,t\}$, $\rho^2=q^{2n}$, for some $n\in\mathbb Z$. Then $C(f,\lambda)$ and $C(l,\nu)$ are in the same block of $\mathscr B_{r, t}(\rho,q)$
if and only if $\lambda$ and $\nu$ are $n$-balanced.
\end{Theorem}

  In the remaining part of this section, we give a criterion for det$\text{ G}_{f,\lambda}=0$ over an arbitrary field for any $(f,\lambda)\in \Lambda_{r,t}$.
This will give a necessary and sufficient condition for each  cell module being equal to its simple head.

By the formulae in Theorem~\ref{main}, Proposition~\ref{f and E} and \eqref{sttsk1}--\eqref{sttsk3}, any non-invertible factor of det$\text{ G}_{f,\lambda}$ in $R$ is one of the following forms:
%Suppose that $\delta\neq 0$. By \cite[Theorem~6.10]{Rsong}, $\mathscr B_{r, t}(\rho,q)$ is semisimple if $\rho^2\notin q^{2\mathbb Z}$ and $e>\max\{r,t\}$. Hence, det$\text{ G}_{f,\lambda}\neq 0$ for any $(f,\lambda)\in\Lambda_{r,t}$.
%So,  any non-invertible factor of det$\text{ G}_{f,\lambda}$ in $R$ is one of the following forms:
\begin{enumerate}
\item $[\delta+n]$, for some $n\in\mathbb Z$.
\item $[n]$, for some $n\in\mathbb Z$.
\end{enumerate}
 Case (a) is equivalent  to that $\rho^2-q^{2n}$  is a factor in det$\text{ G}_{f,\lambda}$. For this case, we use the result on the blocks of $\mathscr B_{r, t}(\rho,q)$ over $\mathbb C$ when $\rho, q\in\mathbb C$ and $q^2$ is not a root of unity.

 \begin{Lemma}\label{block}Let $\mathscr B_{r, t}(\rho,q)$ be the quantized  walled Brauer algebra over $\mathbb C$ with $q^2$  not a root of unity. Then
  det$\text{ G}_{f,\lambda}=0$ if and only if $\rho^2=q^{2n}$, where $n$ is an integer determined by
  \begin{enumerate}
  \item[(1).] there exists some $(l,\mu)\in\Lambda_{r,t}$, $\mu\supseteq \lambda$, and
  \item[(2).] there exists a pairing of the boxes in $\mu^{(1)}/\lambda^{(1)}$ with those in
  $\mu^{(2)}/\lambda^{(2)}$ such that the sum of the contents of the boxes in each pair equals $n$.
  \end{enumerate}
 \end{Lemma}
 \begin{proof}This is a direct corollary of Theorem~\ref{qblocks}.
 \end{proof}

 For   case (b), we determine when det$\text{ G}_{f,\lambda}=0$ over the field $\mathbb C$ such that $\rho^2 \notin q^{2\mathbb Z}$, $\rho,q\in\mathbb C$ and $o(q^2)=e$ (the quantum characteristic of $q^2$). The following result   shows that $[a]$ is a factor of  det$\text{ G}_{f,\lambda}$ if and only if it is a factor of det$\text{ G}_{0,\lambda}$.
 \begin{Prop}\label{gram1}
Let $\mathscr B_{r, t}(\rho,q)$ be the quantized walled Brauer algebra  over  $\mathbb C$ with   $\rho^2 \notin q^{2\mathbb Z}$ and $o(q^2)=e$.
 Then det$\text{ G}_{f,\lambda}\neq 0$ if and only if det$\text{ G}_{0,\lambda}\neq 0$.
 \end{Prop}
 \begin{proof}
 By \cite[Theorem~3.2]{Rsong2}, $C(f,\lambda)$ and $C(l,\mu)$ are in the same block if and only if $f=l$ and $C(0,\lambda)$ and $C(0,\mu)$ are in the same block.
 Moreover, $ [C(f,\lambda):D^{l,\mu}]=\delta_{k,l}[C(0,\lambda):D^{0,\mu}]$ for any $\mu$ being $e$-restricted.
 The result follows from the above observation and   fact that det$\text{ G}_{f,\lambda}=0$ if and only if there is an element $(l,\mu)\in\Lambda_{r,t}$ such that $(l,\mu)\lhd (f,\lambda)$ with $\mu$ being $e$-restricted and
 $[C(f,\lambda):D^{l,\mu}]\neq0$.
 \end{proof}
% Given a partition $\lambda$, $\lambda$ is $e$-restricted if $\lambda_i-\lambda_{i+1}<e$ for all $i\geq 1$.
 Let $\kappa$ be a field of characteristic $p$
 with char$(\kappa)=p>0$. If char $\kappa=0$, set $p=\infty$.
 For each integer $h$, let $v_p(h)$ be the largest power of $p$ dividing $h$ if $p$ is finite and $v_\infty(h)=0$ if $p=\infty$.
 Let $v_{e,p}(h)=v_p(\frac h e)$ (resp., $-1$) if $e<\infty$ and $e\mid h$ (resp., otherwise).
 \begin{Theorem}\label{maina} Let $\kappa$ be a field with  $\text{char}(\kappa)=p$ which contains $\rho^{\pm1},q^{\pm 1}, (q-q^{-1})^{-1}$. Assume $(f,\lambda)\in\Lambda_{r,t}$.
 Then det$\text{ G}_{f,\lambda}\neq 0$ if and only if the following conditions hold:
 \begin{enumerate}
 \item $\rho^2\neq q^{2n}$, where  $n$ is any integer determined in Lemma~\ref{block},
 \item  $\lambda$ is $e$-restricted, i.e, both $\lambda^{(1)}$ and $\lambda^{(2)}$ are $e$-restricted,
 \item $v_{e,p}(h^{\lambda^{(i)}}_{(a,c)})=v_{e,p}(h^{\lambda^{(i)}}_{(a,b)})$ for any $(a,c),(a,b)\in[\lambda^{(i)}]$, $i=1,2$.
 \end{enumerate}

 \end{Theorem}
 \begin{proof}By \cite[Theorem~5.42]{Ma}, (b)-(c) is equivalent to det$\text{ G}_{0,\lambda}\neq 0$ over $\kappa$. So, Theorem~\ref{maina} follows
 from Lemma~\ref{block} and Proposition~\ref{gram1}.
 \end{proof}

\section{Appendix}\label{appendix}

First recall a useful well-known fact that for $\lambda\in\Lambda^+(r)$, we have
 \begin{equation}\label{xyxyx}
 g_i x_\lambda=x_\lambda g_i=qx_\lambda, ~~\quad g_iy_{\lambda}=y_\lambda g_i=-q^{-1}y_{\lambda}, \text{ for } s_i\in\mathfrak S_\lambda.
\end{equation}
  \subsection{Proof of Lemma~\ref{down}} \label{proof of lemma 5.3}

Since  $\nu/\mu=(1,\mu_1^{(1)}+1)\in \mathscr A(\mu^{(1)})$, we have $t=f\geq1$. By Definition~\ref{definition of m},
 \begin{equation}\label{nsnnns}\n_{\t^\mu}=E_{r-f+1,f}g^{-1}_{a,r-f+1}\n_{\t^\nu},\end{equation}
 where $a=\mu^{(1)}_1+1$. For $0\leq l\leq r-f-a$, we write $\mathbf j^0=\emptyset$ and $\mathbf j^l=(j_1,\ldots,j_l)$ with  $a\leq j_l< j_{l-1}< \ldots< j_1\leq r-f$ for $l\geq 1$.
Then
\begin{equation}\label{ehueh}
g_{a,r-f+1}^{-1}=(g_{r-f}+w)\dots (g_a+w)
=g_{r-f+1,a}+\sum_{l=0}^{r-f-a} w^{r-f+1-a-\ell}B_{\mathbf j^l},
\end{equation}
where $w:=q^{-1}-q$, $B_{\mathbf j^0}=1$ and $B_{\mathbf j^l}=g_{j_1}g_{j_2}\dots g_{j_\ell}$ if $l\geq1$. Recall that $f_{\t^\nu}=\n_{\t^\nu}$.
Write $X:=\sum_{j=1}^{a} (-q)^{j-a}g_{a,j} g_{r-f+1,a}^{-1}$. Using \eqref{nsnnns}, Theorem~\ref{mcell}(c), Corollary~\ref{fmcell}(c) and \eqref{eigenv1}, we have
 \begin{equation} \label{n}
 \small{\begin{aligned}n_\mu b_{\t^\mu} \sigma (b_{\t^\mu} ) n_\mu &\equiv \langle f_{\t^\nu},f_{\t^\nu}\rangle E_{r-f+1,f} g_{a,r-f+1}^{-1} n_\nu g_{r-f+1,a}^{-1} E_{r-f+1,f}
\equiv\langle f_{\t^\nu},f_{\t^\nu}\rangle n_\mu g_{a,r-f+1}^{-1}X E_{r-f+1,f}\\
&\overset{\eqref{ehueh}}\equiv\langle f_{\t^\nu},f_{\t^\nu}\rangle n_\mu  (g_{r-f+1,a}+\sum_{l=0}^{r-f-a}w^{r-f+1-a-\ell}B_{\mathbf j^l})XE_{r-f+1,f} \quad (\text{mod }\mathscr B_{r,t}^{\rhd(f,\mu)}).
\end{aligned}}
\end{equation}
Using  Definition~\ref{definition of m} and  Theorem~\ref{ll}, we have
\begin{equation}\label{hswushw}
n_\mu=y_{\mu}E^f =\n_{\t_\mu}\equiv f_{\t_\mu}+\sum_{\t_\mu\prec\t}a_\t f_\t(\text{mod }\mathscr B_{r,t}^{\rhd(f,\mu)}),\end{equation}
for some scalars $a_\t$. Note that if  $\t_\mu \prec\t$, then $\t_k\rhd (\t_\mu)_k$ for some $k\geq r$. Otherwise $\t_\mu=\t$ if $\t_\mu\preceq\t$.
By Lemma~\ref{co}, for any $h\in\mathscr H_{r-1}$, we have
$f_\t h=\sum_{\u}b_\u f_\u$ where $b_\u$ are some scalars  and the sum is over all  $\u$  such that   $\u_k=\t_k$ for $k\geq r$. Hence, if $\t_\mu\prec \t$, then  $\t_\mu\prec \u$ and
\begin{equation}\label{wuhdk}
f_\t h=\sum_{\t_\mu\prec\u}b_\u f_\u, \text{for any } h\in\mathscr H_{r-1}.
\end{equation}
Next we will compute $ n_\mu g_{r-f+1,a}X E_{r-f+1,f}$ and $ n_\mu B_{\mathbf j^l} XE_{r-f+1,f}$ in \eqref{n}. Note that  $E_{r-f+1,f}g_{r-f}E_{r-f+1,f}=\rho E_{r-f+1,f}$ and $E^2_{r-f+1,f}=\delta E_{r-f+1,f}$. So,
\begin{equation}\label{nnn}\small{ \begin{aligned} n_\mu g_{r-f+1,a}X E_{r-f+1,f}
=&\delta n_\mu  +E^{f-1}y_{\mu}\sum_{j=1}^{a-1} (-q)^{j-a} g_{r-f,a-1}^{-1} E_{r-f+1,f} g_{r-f+1,j}  E_{r-f+1,f}\\
\overset{\eqref{xyxyx}}=&\delta n_\mu+\sum_{j=1}^{a-1} (-q)^{2(j-a)+1}\rho n_\mu\overset{\eqref {hswushw}}\equiv  q^{-\mu_1^{(1)}}[\delta-\mu_1^{(1)}]f_{\t_\mu}+\sum_{\t_\mu\prec \u}a'_\u f_\u (\text{mod }\mathscr B_{r,t}^{\rhd(f,\mu)}),  \end{aligned}}\end{equation}
for some scalars $a_\u'$.
Suppose that  $l=0$ or $l>0$ and $j_1<r-f$.  Then by \eqref {hswushw}--\eqref{wuhdk},
\begin{equation}\label{smuna1}
\begin{aligned}
 n_\mu B_{\mathbf j^l} XE_{r-f+1,f} & \equiv \rho^{-1}f_{\t_\mu} B_{\mathbf j^l}\sum_{j=1}^{a} (-q)^{j-a} g_{a,j}g_{r-f,a}^{-1}+\sum_{\t_\mu\prec\u}b'_\u f_\u (\text{mod }\mathscr B_{r,t}^{\rhd(f,\mu)}),
 \end{aligned} \end{equation}
for some scalars $b_\u'$.

Suppose that  $l\ge 1$ and $j_1=r-f$.  If there is an $h$, $1\leq h\leq l-1$ such that $j_h>j_{h+1}+1$,
 let  $k$ be the minimal number such that $j_k>j_{k+1}+1$. Otherwise, $j_l>a$ and let $k=l$.
 Then using  \eqref {hswushw}--\eqref{wuhdk} we have
\begin{equation}\label{smuna2}
\small{\begin{aligned} n_\mu B_{\mathbf j^l} XE_{r-f+1,f}&
\equiv\rho^{-1}& f_{\t_\mu} g_{j_{k+1}}\dots g_{j_\ell}  \sum_{j=1}^{a} (-q)^{j-a} g_{a,j}  g_{r-f,a}^{-1} g_{j_1-1}g_{j_2-1}\dots g_{j_k-1}+\sum_{\t_\mu\prec\u}b''_\u f_\u (\text{mod }\mathscr B_{r,t}^{\rhd(f,\mu)}),
\end{aligned}}
\end{equation}
for some scalars $b_\u''$.

We claim that the coefficient of $f_{\t_\mu}$ in \eqref{smuna1} and \eqref{smuna2} is zero. If so,  the final result follows from  (\ref{nnn}).
 In the following statements, we only prove it for  \eqref{smuna2}, one can check the claim for the case in \eqref{smuna1} similarly.

Recall $j_1>j_2>\dots>j_\ell\geq a$. By Lemma~\ref{sks}, $f_{\t_\mu} g_{j_{k+1}}\dots g_{j_{l-1}}=\sum_{\v=\t_\mu \omega} c_\v f_\v, $
 where $c_\v$'s are some scalars and  $\omega\in \mathfrak{S}_{\{a+1,\dots,r-f-1\}}$, the symmetric group on the letters $\{a+1,\dots,r-f-1\}$.
For any $\v=\t_\mu \omega$, we have $ \v_{a}\ominus \v_{a-1}$ and $\v_{a+1}\ominus\v_a$ are in the same row or in the same column. By Lemma~\ref{sks},  $f_\v g_a= -q^{-1}  f_\v$ or $f_\v g_a=q f_\v$. Hence,
 $f_{\t_\mu} g_{j_{k+1}}\dots g_{j_\ell}=\sum_{\v=\t_\mu \omega} c'_\v f_\v, $
for some scalars $c_\v'$,  where $\omega\in \mathfrak{S}_{\{a+1,\dots,r-f-1\}}$.
Moreover, $\v s_{a-1}\prec\v$ and  $\v _j\ominus \v_{j-1} $ and $\v _{j+1}\ominus \v_{j}$ are in the same row, for  $1\leq j\leq a-2$.
Since $\v_{a}\ominus \v_{a-1}=(2,1), \v_{a-1}\ominus \v_{a-2}=(1,a-1)$, we have $S_{\v,\v}(a-1)=q^{a-1}/[a-1]$ (see ~\eqref{sttsk1}). So, $S:=1+\sum_{j=1}^{a-1} (-q)^{2(j-a)+1} S_{\v,\v}(a-1)=0$.
Therefore,
\begin{equation} \label{eeeew}
\begin{aligned}
f_\v \sum_{j=1}^{a} (-q)^{j-a}g_{a,j}&=f_\v+(S_{\v,\v}(a-1)f_\v+ f_{\v s_{a-1}})\sum_{j=1}^{a-1}(-q)^{j-a}g_{a-1,j}\\
&=Sf_\v+ \sum_{j=1}^{a-1} (-q)^{j-a} f_{\v s_{a-1}}g_{a-1,j}=\sum_{j=1}^{a-1} (-q)^{j-a} f_{\v s_{a-1}}g_{a-1,j}.
 \end{aligned}
 \end{equation}
 %where the last equality follows from the fact that
 %Hence the coefficient of $f_\v$ in $f_\v \sum_{j=1}^{a} q^{a-j}g_{a,j}$ is zero.
By  \eqref{eeeew} and Lemma~\ref{sks}, any $f_\u$ which appears in $f_\v \sum_{j=1}^{a} (-q)^{j-a}g_{a,j}$ must satisfy that $\u_{a-1}\neq (\t_\mu)_{a-1}$.
 Furthermore,
since  $r-f=j_1>j_2>\dots>j_k\geq a+1$, by Lemma~\ref{sks} again, any $f_\w$ which appears in  $f_\u g_{r-f,a}^{-1} g_{j_1-1}g_{j_2-1}\dots g_{j_k-1} $ must satisfy that $\w_{a-1}=\u_{a-1}\neq (\t_\mu)_{a-1}$, proving the claim.

\subsection{Proof of Lemma~\ref{lemeaforeee}}\label{proofoflem6}
We need the following result in the proof of Lemma~\ref{lemeaforeee}.
\begin{Lemma}\label{trival}\cite[Lemma~4.1]{DJ1}
For any $w\in\mathfrak S_r$, $x_{\lambda}g_wy_{\lambda'}\neq 0$ if and only if $\mathfrak S_{\lambda}w\mathfrak S_{\lambda'}=\mathfrak S_{\lambda}w_\lambda\mathfrak S_{\lambda'}$.
In this case, $\ell(w)\geq \ell(w_\lambda)$, where $\ell(\cdot)$ is the length function on $\mathfrak S_r$.
\end{Lemma}
The following result is a direct corollary of Lemma~\ref{trival}.
\begin{Cor}\label{sjijd}
 For any $w\in\mathfrak S_r$ with $\ell(w)<\ell(w_\lambda)$, $x_{\lambda}g_wy_{\lambda'}=0$.
\end{Cor}

\begin{Lemma}\label{you}
Let $\lambda=(a^b)\in\Lambda^+(r)$ with $a,b \geq 2$. Then
\begin{enumerate}
\item $w_0 w_1\cdots w_{a-2} $ is a reduced expression of $w_\lambda$, where
$w_i=w_{i,1} w_{i,2}\cdots w_{i,b-1}$, and
 \begin{equation}\label{wijjj}w_{i,j}=s_{(b-j)(a-i),(a-i)b-j}
 \end{equation}
 for $0\leq i\leq a-2$, $1\leq j\leq b-1$.
 %\item $\hat w_0 \hat w_1\cdots \hat w_{a-2} $ is a reduced expression of $w_\nu$, where
%$\hat w_i=\hat w_{i,1}\cdots \hat w_{i,2}\cdots \hat w_{i,b-1}$, and
% \begin{equation}\label{wijjj}w_{i,j}=s_{(b-j)(a-i),(a-i)b-j}
% \end{equation}
% for $0\leq i\leq a-2$, $1\leq j\leq b-1$.
% \item $\hat w_0 \hat w_1\cdots \hat w_{a-2} $ is a reduced expression of $w_\lambda$, where
%$\hat w_i=\hat w_{i,1}\cdots \hat w_{i,2}\cdots \hat w_{i,b-1}$, and
% \begin{equation}\label{wijjj}\hat w_{i,j}=s_{bi+j(a-i),bi+j+1}
% \end{equation}
% for $0\leq i\leq a-2$, $1\leq j\leq b-1$.
\item Let $\mu=(a^{b-1}, a-1)$. Then for $2\leq j\leq b$, $2\leq k\leq a$, and $r-k\leq l\leq r-2$,
 \begin{equation}\label{widdj} x_\mu g_{r-(j-1)a,l}g_{l,r-k}B_j g_{\sigma_j}g_{\bar w_{\lambda}}y_{\mu'}=0, \end{equation}
 where $B_j=g_{\tilde w_{0,1}}g_{\tilde w_{0,2}}\cdots g_{\tilde w_{0,j-2}}$, $\tilde w_{0,i}=s_{(b-i)a-1,ab-i-1}$, $1\leq i\leq j-2$,
   $\sigma_j=w_{0,j}w_{0,j+1}\cdots w_{0,b-1}$, and $\bar w_{\lambda}=w_{1}w_2\cdots w_{a-2}$.
\end{enumerate}
\end{Lemma}
\begin{proof}Let $w:=w_0 w_1\cdots w_{a-2}$ and $N(w)=\{(j,k)\in\mathfrak S_r\mid 1\leq j<k\leq r \text{ and }jw>kw\}$. Then $\T^\lambda w=\T_\lambda$ (i.e, $w=w_\lambda$)
and  $|N(w)|=\frac{a(a-1)}{2}\frac {b(b-1)}{2}$.
It is well known that $\ell(w)=|N(w)|$.
Let $\hat\ell(w_{i,j})$ be the number of simple reflections  in  \eqref{wijjj}. Then $\hat\ell(w_{i,j})=(a-i+1)j$ and  the number of simple reflections in $w$
is $\sum_{0\leq i\leq a-2, 1\leq j\leq b-1}(a-i+1)j$, which is equal to $\frac{a(a-1)}{2}\frac {b(b-1)}{2}$. So, $w$ is a reduced expression of $w_\lambda$. This completes the proof of (a).

Let $\mathscr H_{\{s,s+1,\ldots,s+k\}}$ be the subalgebra of $\mathscr H_{s+k}$ generated by $g_s,g_{s+1}, \ldots,g_{s+k-1}$.
%Let $A$ denote the left side in \eqref{widdj}.
  We claim that
  \begin{equation}g_{r-(j-1)a,l}g_{l,r-k}B_jg_{\sigma_j}=\sum_wa_wf_wg_w h_w
   \end{equation}
   for some $f_w\in\mathscr H_{\{r-a+1,r-a+2,\ldots,r-2\}}$, $h_w\in\mathscr H_{\{r-b+1,r-b+2,\ldots,r-2\}}$ and $ a_w\in \kappa$ such that $ \ell(w)<\ell(w_0)$.
 By Definition~\ref{qwb}(b), $h_wg_{\bar w_{\lambda}}=g_{\bar w_{\lambda}}h_w $. By \eqref{xyxyx},  $x_\mu f_w=F_wx_{\mu}$ and $h_wy_{\mu'}=H_wy_{\mu'}$, for some scalars $F_w,H_w$.
 So,  \eqref{widdj} follows from the claim and Corollary~\ref{sjijd}. It remains to prove the claim.

 We prove the claim  by induction on $l$.  Suppose $l=r-k$. Note that $\ell(\tilde w_{0,i})=\ell(w_{0,i})$, $1\leq i\leq j-2$.
 If $k>j-1$, then $\ell(s_{r-a-1,r-k} )<\ell(w_{0,j-1})$ and
 $\ell(s_{r-a-1,r-k}\tilde w_{0,1}\cdots \tilde w_{0,j-2}   \sigma_j )<\ell(w_0)$. Hence
 $$g_{r-(j-1)a,r-k}B_jg_{\sigma_j}=\sum_{\ell(w)<\ell(w_0)}a_wg_w,$$   forcing the claim.
 If $2\leq k\leq j-1$, we define  $B_{i,k}:=g_{r-(j-1)a,r-k} g_{\tilde w_{0,i}}\cdots g_{\tilde w_{0,j-2}} $ for  $ k\geq i+1$ and  $1\leq i\leq j-2$.
Let $z_q=(q-q^{-1})$.
  By  Definition~\ref{qwb}(a)--(c),
 \begin{equation}\label{ubb}
 \begin{aligned}
 B_{i,k}& = g_{r-ia,r-k}g_{r-(j-1)a,r-k}g_{r-k-1,r-i-1}g_{\tilde w_{0,i+1}}\cdots g_{\tilde w_{0,j-2}}\\
 &=z_qU_{i,k}g_{r-j+1,r-i-1}+g_{r-ia,r-i-1}B_{i+1,k+1},
   \end{aligned} \end{equation}
where $U_{i,k}= g_{r-ia,r-k}g_{ w_{0,i+1}}\cdots g_{ w_{0,j-1}}$ and $B_{i+1,k+1}=1$ if $i=j-2$.
% \begin{equation}
% g_{r-j+1,r-i-1}g_{\sigma_j}=g_{\sigma_j}g_{r-j+1,r-i-1}~\text{ and }~g_{r-j+1,r-i-1}\in \mathscr H _{\{r-j+1, \ldots r-2\}}.
% \end{equation}
 Since $\ell(s_{r-ia,r-k})<\ell(w_{0,i})$,  we have
 $$ \text{$U_{i,k}=\sum_wa_w g_w$ with $\ell(w)<\ell(w_{0,i}\cdots w_{0,j-1})$ and $a_w\in\kappa$.}$$
 Note that $g_{r-j+1,r-i-1}\in \mathscr H _{\{r-j+1, \ldots, r-2\}}$ and $\ell( s_{r-ia,r-i-1})<\ell(w_{0,i})$.
So,
$$\text{$B_{j-2,k}=\sum_wa_w g_wh_w$ with $\ell(w)<\ell(w_{0,j-2} w_{0,j-1})$ and $h_w\in \mathscr H _{\{r-j+1, \ldots, r-2\}}$.}$$
By \eqref{ubb} and induction on $j-2-i$, we see that  $B_{i,k}=\sum_wa_w g_wh_w$
 with $\ell(w)<\ell(w_{0,i}\cdots w_{0,j-1})$ and $ h_w\in \mathscr H _{\{r-j+1, \ldots, r-2\}}$.
In particular, $\small{g_{r-(j-1)a,r-k}B_j=B_{1,k}=\sum _w a_w g_w h_w}$
 for some $h_w\in \mathscr H _{\{r-j+1, \ldots, r-2\}}$ with $\ell(w)<\ell(w_{0,1}\cdots w_{0,j-1})$.
 Then the claim follows for $l=r-k$ since  $h_w g_{\sigma_j}= g_{\sigma_j}h_w $ for $h_w\in \mathscr H _{\{r-j+1, \ldots, r-2\}}$.
Finally,  if $r-k<l\leq r-2$, then
$
g_{r-(j-1)a,l}g_{l,r-k}B_jg_{\sigma_j}=z_q g_{r-(j-1)a,r-k}B_jg_{\sigma_j}+ g_{l-1,r-k}g_{r-(j-1)a,l-1}g_{l-1,r-k}B_jg_{\sigma_j}
$.  Since $g_{l-1,r-k}\in \mathscr H_{\{r-a+1,r-a+2,\ldots,r-2\}}$, the claim follows from induction on $l$ and the result for $l=r-k$.
\end{proof}
\begin{example}\label{exofwlambda}
Suppose that $a=2,b=3$ and $\lambda=(2^3)$. Then $\T^\lambda=\tiny\young(12,34,56)$, $\T_\lambda=\tiny\young(14,25,36)$ and $w_\lambda=w_0=  w_{0,1} w_{0,2}$, where
$ w_{0,1}=s_{4}$, $ w_{0,2}=s_{2,4}$.
%For $s_1\in \mathfrak S_{\lambda'}$, $T_{1}=g_2g_1g_2^{-1}$ and $T_{1}g_{w_\lambda}y_{\lambda'}g^{-1}_{w_\lambda}= (-q)^{-1}g_{w_\lambda}y_{\lambda'}g^{-1}_{w_\lambda}$.
\end{example}

\begin{proof}[\bf Proof of Lemma~\ref{lemeaforeee}]First note that $w_{\lambda}=w_{\lambda^{(1)}}w_{\lambda^{(2)}}$, $w_\lambda=w_\mu$ and hence $e_\mu= x_\mu g_{w_\lambda}y_{\mu'}$.
To distinguish the factors of  $w_{\lambda^{(1)}}$ with that of $w_{\lambda^{(2)}}$, we denote by $w^*_i$ and $  w^*_{i,j}$
for the factors of $w_{\lambda^{(2)}}$ defined as that  in \eqref{wijjj}.
We know that $w_{\lambda^{(2)}}= w^*_0 w^*_1 \cdots  w^*_{c-1}$ and $w_{\lambda^{(1)}}= w_0 w_1 \cdots  w_{a-1}$, where $w_i$'s are given in Lemma~\ref{you} and
$$ w^*_i= w^*_{i,1}w^*_{i,2}\cdots w^*_{i,d-1}\text{ and } w^*_{i,j}=s^*_{(d-j)(c-i),(c-i)d-j},$$
 for $0\leq i\leq c-2$, $1\leq j\leq d-1$.
 Denote by $\bar w_{\lambda}=\bar w_{\lambda^{(1)}}\bar w_{\lambda^{(2)}}$, where $\bar   w_{\lambda^{(1)}}$ is given in Lemma~\ref{you}.
 So, $w_\lambda=w_0w^*_0 \bar w_{\lambda}$.
Write $x_\lambda=x_\mu (1+X_a)(1+X_c),y_{\lambda'}=(1+Y_b)(1+Y_d)y_{\mu'}, g_{w_\lambda}=g_{w_0}g^*_{w^*_0}g_{\bar w_{\lambda}},$
where $X_a=g_{r-1} \bar X_a$, $X_c=g^*_{t-1}\bar X_c $,  $Y_b= \bar Y_b g_{r-1}$, $Y_d=\bar Y_d g^*_{t-1}$
and
\begin{equation}\label{bary1y2}\small{
\begin{aligned}\bar X_a&= \sum_{j=2}^a q^{j-1}g_{r-1,r-j+1},~~~\bar X_c=\sum_{j=2}^c q^{j-1}g^*_{t-1,t-j+1},\\
\bar Y_b&= \sum_{j=2}^b (-q)^{1-j}g_{r-j+1,r-1},~~~\bar Y_d=\sum_{j=2}^d (-q)^{1-j}g^*_{t-j+1, t-1}\end{aligned}}.\end{equation}
% \begin{equation}\label{x1y1x2y2}
%\begin{aligned}
%X_1&=\sum_{j=2}^{a}q^{j-1}g_{r,r-j+1}, X_2= \sum_{j=2}^{c}q^{j-1}g^*_{t,t-j+1}, \\
%Y_1&=\sum_{j=2}^{b}(-q)^{ 1-j}g_{r-j+1,r}, Y_2=\sum_{j=2}^{d}(-q)^{1-j}g^*_{t-j+1,t},
%\end{aligned}\end{equation}
Recall that $E=\Phi(E)=E_{r,t}$(see Lemma~\ref{three isom}(e)). By Definition~\ref{qwb}(b) and (g),
\begin{equation}\label{cimnsy}
g_{\bar w_{\lambda}}g_w=g_w g_{\bar w_{\lambda}}, Eg_{w_\lambda}=g_{w_\lambda}E\end{equation}
 for any $w\in \mathfrak S_{\{r-b+1,\ldots,r\}}\times\mathfrak S_{\{t-d+1,\ldots,t\}}$.
 Then $$E e_\lambda E= x_\mu E(1+X_a)(1+X_c)g_{w_0}g^*_{w_0^*}(1+Y_b)(1+Y_d)Eg_{\bar w_{\lambda}}y_{\mu'}.$$
We compute each summand as follows.

\textbf{Case 1 $(1,1,1,1)$: }
 By Definition~\ref{qwb}, $E^2=\delta E$ and
 \begin{equation}\label{comujh}
 E g_{w_0}=g_{w_0}E, Eg^*_{w_0^*}= g^*_{w_0^*}E.
 \end{equation}
 We can get  $x_\mu E g_{w_0}g^*_{w_0^*}Eg_{\bar w_{\lambda}}y_{\mu'}=\delta Ee_\mu$.

 \textbf{Case 2 $(X_a,1,1,1), (1,X_c,1,1),(1,1,Y_b,1),(1,1,1, Y_d) $: }
 By Definition~\ref{qwb},
 \begin{equation}\label{rpho}
 Eg_{r-1}E=Eg_{t-1}E=\rho E.
 \end{equation}
Then we have
$$\small{\begin{aligned}
 x_\mu EX_a g_{w_0}g^*_{w_0^*}Eg_{\bar w_{\lambda}}y_{\mu'}&\overset{\eqref{comujh}}=x_\mu EX_aEg_{w_{\lambda}}y_{\mu'}
 \overset{\eqref{rpho}}=\rho E x_\mu\bar X_ag_{w_{\lambda}}y_{\mu'}\overset{\eqref{xyxyx}}=\rho \frac{q^{2a-2}-1}{q-q^{-1}}Ee_\mu.
 \end{aligned}}$$
Using the isomorphism $\varphi$ in Lemma~\ref{three isom}(b), we have $x_\mu EX_c g_{w_0}g^*_{w_0^*}Eg_{\bar w_{\lambda}}y_{\mu'}= \rho \frac{q^{2c-2}-1}{q-q^{-1}}Ee_\mu$.
Similarly, using \eqref{xyxyx}, \eqref{cimnsy}--\eqref{rpho} and the isomorphism $\varphi$ in Lemma~\ref{three isom}(b), for the summand $(1,1,Y_b,1)$ and $(1,1,1,Y_d)$, we have $x_\mu E g_{w_0}g^*_{w_0^*}Y_i Eg_{\bar w_{\lambda}}y_{\mu'}=\rho \frac{q^{2-2i}-1}{q-q^{-1}}Ee_\mu$ for $i=b,d$.

\textbf{Case 3 $(X_a,X_c,1,1), (1,1,Y_b,Y_d)$, $ (X_a,1,1,Y_d),  (1,X_c,Y_b, 1)$:   }
By Definition~\ref{qwb}(a) and (k),
\begin{equation}\label{elwsn2}
\small{
Eg_{r-1}g^*_{t-1}E=E_{r,t}E_{r-1,t-1}+\rho(q-q^{-1})E
\equiv \rho(q-q^{-1})E (\text{mod} \mathscr B_{r,t}^2(\rho,q)).
}\end{equation}

Then we have
$$\small{\begin{aligned}
x_\mu EX_aX_c g_{w_0}g^*_{w_0^*}Eg_{\bar w_{\lambda}}y_{\mu'}&\overset{\eqref{comujh}}=x_\mu EX_aX_cEg_{ w_{\lambda}}y_{\mu'}
\overset{\eqref{elwsn2}}\equiv \rho(q-q^{-1})E x_\mu  \bar X_a\bar X_cg_{ w_{\lambda}}y_{\mu'}\\
&\overset{\eqref{xyxyx}}  \equiv\rho \frac{(q^{2a-2}-1)(q^{2c-2}-1)}{q-q^{-1}}Ee_\mu(\text{mod} \mathscr B_{r,t}^2(\rho,q)).
\end{aligned}}
$$
Similarly, using \eqref{xyxyx}, \eqref{comujh}--\eqref{cimnsy}, \eqref{elwsn2} and   the isomorphism $\varphi$ in Lemma~\ref{three isom}(b), we have
$$\small{x_\mu E g_{w_0}g^*_{w_0^*}Y_bY_d Eg_{\bar w_{\lambda}}y_{\mu'}\equiv\rho \frac{(q^{2-2b}-1)(q^{2-2d}-1)}{q-q^{-1}}Ee_\mu(\text{mod} \mathscr B_{r,t}^2(\rho,q))}, $$
and $x_\mu EX_i  g_{w_0}g^*_{w_0^*}Y_jEg_{\bar w_{\lambda}}y_{\mu'}\equiv\rho \frac{(q^{2i-2}-1)(q^{2-2j}-1)}{q-q^{-1}}Ee_\mu (\text{mod} \mathscr B_{r,t}^2(\rho,q))$ for $(i,j)\in\{(a,d),(c,b)\}$.

%For the summand $ (X_1,1,1,Y_2)$,
% we have
% \begin{equation}
% \begin{aligned}
% x_\mu EX_1  g_{w_0}g^*_{w_0^*}Y_2Eg_{\bar w_{\lambda}}y_{\mu'}&\overset{}=x_\mu g^*_{w_0^*}EX_1Y_2E g_{w_0}g_{\bar w_{\lambda}}y_{\mu'}\\&
% \overset{}}\equiv\rho(q-q^{-1})E x_\mu \bar X_1g^*_{w_0^*}\bar Y_2g_{w_0}g_{\bar w_{\lambda}}y_{\mu'}\\
% &\overset{}\equiv\rho(q-q^{-1})E x_\mu\bar X_1g_{w_ \lambda}\bar Y_2y_{\mu'}\\
% &\overset{}\equiv\rho(q-q^{-1})  \sum_{j=2}^a q^{2j-3} \sum_{j=2}^d (-q)^{3-2j}E x_\mu g_{w_ \lambda}y_{\mu'}\\
% &\equiv\rho \frac{(q^{2a-2}-1)(q^{2-2d}-1)}{q-q^{-1}}Ee_\mu (\text{mod} \mathscr B_{r,t}^2(\rho,q)).
%\end{aligned} \end{equation}
%, for the summand $(1,X_2,Y_1,1)$,  we have
% $$ \small{x_\mu EX_2  g_{w_0}g^*_{w_0^*}Y_1Eg_{\bar w_{\lambda}}y_{\mu'}\equiv\rho \frac{(q^{2c-2}-1)(q^{2-2b}-1)}{q-q^{-1}}Ee_\mu (\text{mod} \mathscr B_{r,t}^2(\rho,q))}.$$
%
% we have
%

\textbf{Case 4: $(X_a,1,Y_b,1), (1,X_c,1,Y_d)$}

Recall
$B_j=g_{\tilde w_{0,1}}\cdots  g_{\tilde w_{0,j-2}}$ and $\sigma_j=w_{0,j}w_{0,j+1}\cdots w_{0,b-1}$ given in Lemma~\ref{you}(b). Similarly, we have the notations $B^*_j=g^*_{\tilde w^*_{0,1}}\cdots  g^*_{\tilde w^*_{0,j-2}}$ and $\sigma^*_j=w^*_{0,j}w^*_{0,j+1}\cdots w^*_{0,d-1}$ for $2\leq j\leq d$.
Let
\begin{equation}\label{xyxyhtilde}
\tilde X_1=\sum_{k=2}^{a}q^{k-1}g_{r-1,r-k},~~ \tilde X_2=\sum_{k=2}^{c}q^{k-1}g^*_{t-1,t-k}, ~~H_1=g_{\bar w_{\lambda^{(1)}}}y_{\mu^{(1)'}},~~ H_2=g^*_{\bar w_{\lambda^{(1)}}}y_{\mu^{(2)'}}.
\end{equation}
   By \eqref{wijjj},  $w_{0,j}=s_{r-ja,r-j}$, for $1\leq j\leq b-1$. By the braid relation, for $2\leq j\leq b$
\begin{equation}\label{elwsn2we}
\begin{aligned}
X_ag_{w_0}g_{r-j+1,r}&\overset{\eqref{wijjj}}=X_ag_{w_{0,1}}\cdots  g_{w_{0,j-1}}g_{r-j+1,r} g_{\sigma_j}
=g_{r-(j-1)a,r}\tilde X_aB_jg_{\sigma_j}.
\end{aligned}
\end{equation}
By \eqref{elwsn2we} and the same computation for $X_c  g^*_{w_0^*}Y_d$, we have
\begin{equation}\label{wxyshshh}
X_a  g_{w_0}Y_b= \sum_{j=2}^{b}(-q)^{1-j} g_{r-(j-1)a,r}\tilde X_1B_jg_{\sigma_j}, ~~X_c  g^*_{w_0^*}Y_d= \sum_{j=2}^{d}(-q)^{1-j} g^*_{t-(j-1)c,t}\tilde X_2B^*_jg^*_{\sigma^*_j}.
\end{equation}
Then using \eqref{comujh}--\eqref{rpho} and \eqref{wxyshshh} we have
\begin{equation}\label{whesqure22}
\begin{aligned}
x_\mu EX_a  g_{w_0}g^*_{w_0^*}Y_bEg_{\bar w_{\lambda}}y_{\mu'}=  \rho E\sum_{j=2}^b(-q)^{1-j}x_{\mu^{(1)}} g_{r-(j-1)a,r-1}\tilde X_1B_jg_{\sigma_j}H_1 e_{\mu^{(2)}},
\end{aligned}
\end{equation}
where $H=H_1H_2$.
By Definition~\ref{qwb}(a), we have $g_i^2=(q-q^{-1})g_i+1$. So,
\begin{equation}\label{whesqure}
 g_{r-(j-1)a,r-1}\tilde X_1\overset{Def.\ref{qwb} (c)} =(q-q^{-1})\sum_{k=2}^{a}q^{k-1}g_{r-1,r-k+1} g_{r-(j-1)a,r-1}+B,
  \end{equation}
 where  $B= g_{r-(j-1)a,r-2}\sum_{k=2}^{a}q^{k-1}g_{r-2,r-k}$. By \eqref{whesqure},
 \begin{equation}\label{qiwhs}
\small{\sum_{j=2}^b(-q)^{1-j} x_{\mu^{(1)}}g_{r-(j-1)a,r-1}\tilde X_1B_jg_{\sigma_j}H_1=(q-q^{-1})\sum_{j=2}^b(-q)^{1-j}A_j+ \sum_{j=2}^b(-q)^{1-j}x_{\mu^{(1)}} B B_jg_{\sigma_j}H_1},
 \end{equation}
where $A_j= x_{\mu^{(1)}}\sum_{k=2}^{a}q^{k-1}g_{r-1,r-k+1} g_{r-(j-1)a,r-1}B_jg_{\sigma_j}H_1$.
We have
\begin{equation}\label{coumofaj}
\begin{aligned}
A_j&\overset{\eqref{xyxyx}}=\sum_{k=2}^{a}q^{2k-3} x_{\mu^{(1)}} g_{r-(j-1)a,r-1}B_jg_{\sigma_j}H_1
\overset{Def.\ref{qwb} (c)} =\sum_{k=2}^{a}q^{2k-3}x_{\mu^{(1)}} g_{w_0}g_{r-j+1,r-1}H_1 \\
&\overset{\eqref{cimnsy}}=\sum_{k=2}^{a}q^{2k-3}x_{\mu^{(1)}} g_{w_{\lambda^{(1)}}}g_{r-j+1,r-1}y_{\mu^{(1)'}} \overset{\eqref{xyxyx}}=(-q)^{2-j}\sum_{k=2}^{a}q^{2k-3}e_{\mu^{(1)}}.
\end{aligned}
\end{equation}
By Lemma~\ref{you}(b),
\begin{equation}\label{bj=0}
x_{\mu^{(1)}} B B_jg_{\sigma_j}H_1=x_{\mu^{(1)}} g_{r-(j-1)a,r-2}\sum_{k=2}^{a}q^{k-1}g_{r-2,r-k}B_jg_{\sigma_j}g_{\bar w_{\lambda^{(1)}}}y_{\mu^{(1)'}}=0.
\end{equation}
Combining \eqref{qiwhs}--\eqref{bj=0} yields that
\begin{equation}\label{mu1iqa}
\sum_{j=2}^b(-q)^{1-j} x_{\mu^{(1)}}g_{r-(j-1)a,r-1}\tilde X_1B_jg_{\sigma_j}H_1= \frac{(q^{2-2b}-1)( q^{2a-2}-1)}{q-q^{-1}} e_{\mu^{(1)}}.
\end{equation}
By \eqref{whesqure22} and \eqref{mu1iqa},
$$
x_\mu EX_a  g_{w_0}g^*_{w_0^*}Y_bEg_{\bar w_{\lambda}}y_{\mu'}=\rho \frac{(q^{2a-2}-1)(q^{2-2b}-1)}{q-q^{-1}}Ee_\mu. $$
Using the isomorphism $\varphi$ in Lemma~\ref{three isom}(b), we have the same result as \eqref{mu1iqa} that
\begin{equation}\label{mu1iqa2}
\sum_{j=2}^d(-q)^{1-j} x_{\mu^{(2)}}g^*_{t-(j-1)c,t-1}\tilde X_2B^*_jg^*_{\sigma_j}H_2= \frac{(q^{2-2d}-1)( q^{2c-2}-1)}{q-q^{-1}} e_{\mu^{(2)}}.
\end{equation}
Similarly,  we have
$$x_\mu EX_c  g_{w_0}g^*_{w_0^*}Y_dEg_{\bar w_{\lambda}}y_{\mu'}=\rho \frac{(q^{2c-2}-1)(q^{2-2d}-1)}{q-q^{-1}}Ee_\mu. $$

\textbf{Case 5 $(X_a,X_c,Y_b,1), ( X_a,X_c,1,Y_d)$, $ (X_a,1,Y_b,Y_d), (1,X_c,Y_b,Y_d)$: }
For the summand $(X_a,X_c,Y_b,1)$, we have
$$
\begin{aligned}
x_\mu EX_aX_c g_{w_0}g^*_{w_0^*}Y_bEg_{\bar w_{\lambda}}y_{\mu'}&\overset {\substack{\eqref{cimnsy}\\ \eqref{wxyshshh}}}= \sum_{j=2}^b(-q)^{1-j}x_\mu E g_{r-(j-1)a,r} g^*_{t-1}E \tilde X_1 B_jg_{\sigma_j}  \bar X_c g^*_{w_0^*}H_1H_2\\
&\overset{\substack{\eqref{elwsn2}\\ \eqref{mu1iqa}}}\equiv\rho(q-q^{-1})E\frac{(q^{2-2b}-1)( q^{2a-2}-1)}{q-q^{-1}}x_\mu\bar X_c g^*_{w_0^*}H_1H_2\\
&\overset{\eqref{xyxyx}}\equiv\rho \frac{(q^{2-2b}-1)( q^{2a-2}-1)(q^{2c-2}-1)}{q-q^{-1}}Ee_\mu (\text{mod}\mathscr B_{r,t}^2(\rho,q)),
\end{aligned}
$$
 where $H_1, H_2$ and $\tilde X_1$ are given in \eqref{xyxyhtilde}.
 Using the isomorphism $\varphi$ in Lemma~\ref{three isom}(b), for the summand  $( X_a,X_c,1,Y_d)$ we have
 $$ x_\mu EX_aX_c g_{w_0}g^*_{w_0^*}Y_dEg_{\bar w_{\lambda}}y_{\mu'}\equiv\rho \frac{(q^{2-2d}-1)( q^{2a-2}-1)(q^{2c-2}-1)}{q-q^{-1}}Ee_\mu (\text{mod}\mathscr B_{r,t}^2(\rho,q)). $$
Similarly, for $i\in\{a,c\}$, $$x_\mu EX_i  g_{w_0}g^*_{w_0^*}Y_bY_dEg_{\bar w_{\lambda}}y_{\mu'}\equiv\rho \frac{(q^{2-2b}-1)( q^{2i-2}-1)(q^{2-2d}-1)}{q-q^{-1}}Ee_\mu (\text{mod}\mathscr B_{r,t}^2(\rho,q)).$$

\textbf{Case 6 $(X_a,X_c,Y_b,Y_d)$: }
By \eqref{wxyshshh} and \eqref{elwsn2},
\begin{equation}\label{qqwsa}
EX_a g_{w_0}Y_b X_cg^*_{w_0^*}Y_d E\equiv\rho(q-q^{-1})EQ Q^*(\text{mod}\mathscr B_{r,t}^2(\rho,q)).
\end{equation}
where $Q=\sum_{j=2}^b(-q)^{1-j}g_{r-(j-1)a,r-1} \tilde X_1 B_jg_{\sigma_j}$, $Q^*=\sum_{j=2}^d(-q)^{1-j}g^*_{t-(j-1)c,t-1} \tilde X_2 B^*_jg^*_{\sigma^*_j}$.
Recall that $H_i$ and $\tilde X_i$, $i=1,2$ are given in \eqref{xyxyhtilde}. We have
$$
\begin{aligned}
x_\mu EX_aX_c  g_{w_0}g^*_{w_0^*}Y_bY_dEg_{\bar w_{\lambda}}y_{\mu'}&\overset{\eqref{cimnsy}}=x_\mu E X_a g_{w_0}Y_b X_cg^*_{w_0^*}Y_d E H_1H_2
\overset{\eqref{qqwsa}}\equiv\rho(q-q^{-1})Ex_\mu QQ^*  H_1H_2\\
\overset{\eqref{mu1iqa},\eqref{mu1iqa2}}
\equiv&\rho \frac{(q^{2-2b}-1)( q^{2a-2}-1)(q^{2-2d}-1)( q^{2c-2}-1)}{q-q^{-1}}Ee_\mu (\text{mod}\mathscr B_{r,t}^2(\rho,q)).
\end{aligned}
$$

Finally, combining the above  computations for cases 1--6, we have
$$E e_\lambda E\equiv (\delta+\rho(q-q^{-1})^{-1}X) Ee_\mu (\text{mod}\mathscr B_{r,t}^2(\rho,q)),$$ where
$X=x+y+z+h+xy+xz+xh+yz+yh+zh+xyz+xyh+yzh+xzh+xyzh$
 and $x=q^{2a-2}-1, y=q^{2c-2}-1, z=q^{2-2b}-1, h=q^{2-2d}-1$.
 It is straightforward to check that $X=(x+1)(y+1)(z+1)(h+1)-1= q^{2(a-b+c-d)}-1$.
This completes the proof of Lemma~\ref{lemeaforeee}.
\end{proof}

\begin{example} \label{ex} Suppose $\lambda=((2^3),(1))$, $\mu=((2^2,1),\emptyset)$ and
\begin{equation}\label{tlar}
 \t^{\lambda}=\left(\tiny{\young(12,34,56), \quad \young(1)}\right),
\quad \text{ } \quad \t^{\mu}=\left (\tiny{\young(12,34,5),\quad\young(\times)}\right),
\quad \text{ } \quad \t_{\lambda'}=\left(\tiny{\young(135,246), \quad \young(1)}\right),
\quad \text{ }\quad \t^{\mu'}=\left (\tiny{\young(123,45), \quad \young(\times)}\right).
\end{equation}
Then $x_\lambda=x_\mu(1+qg_5)$, $y_{\lambda'}=(1-q^{-1}g_5+q^{-2}g_4g_5)y_{\mu'}$,  $X_a=q g_5$, $X_c=Y_d=0$, $Y_b=-q^{-1}g_5+q^{-2}g_4 g_5$,
 $w_\lambda=w_0=s_4s_2s_3$. Further,  $E x_\mu g_{w_0}y_{\mu'}E=\delta Ee_\mu $,
 $Ex_\mu X_a g_{w_0}y_{\mu'}E=\rho\frac{q^2-1}{q-q^{-1}} Ee_\mu$,
 $Ex_\mu g_{w_0}Y_b y_{\mu'}E= \rho\frac{q^{-4}-1}{q-q^{-1}} Ee_\mu$, and $Ex_\mu X_ag_{w_0}Y_b y_{\mu'}E \equiv\rho\frac{(q^{-4}-1)(q^2-1)}{q-q^{-1}} Ee_\mu(\text{mod}\mathscr B_{6,1}^2(\rho,q)) $.
% \begin{itemize}
% \item
% \item
% \item
% \item
% \end{itemize}
 So, $E e_\lambda E\equiv(\delta+\rho\frac{q^{-2}-1}{q-q^{-1}}) Ee_\mu (\text{mod}\mathscr B_{6,1}^2(\rho,q))$.
\end{example}

%Hence $$ \begin{aligned} & E_{r-f+1,f} n_\nu g_{j_{k+1}}\dots g_{j_\ell}  \sum_{j=1}^{a} (-q)^{-(a-j)} g_{a,j}   \rho^{-1}  g_{r-f,a}^{-1} g_{j_1-1}g_{j_2-1}\dots g_{j_k-1} \\
%\equiv  &\rho^{-1} \sum\limits_{ \v= \t_\mu \omega\atop \omega\in \mathfrak{S}_{\{a+1,\dots,r-f-1\}}}   a_\v f_\v \sum_{j=1}^{a} (-q)^{-(a-j)} g_{a,j}  g_{r-f,a}^{-1} g_{j_1-1}g_{j_2-1}\dots g_{j_k-1} \quad (\text{mod }\mathscr B_{r,t}^{\rhd(f,\mu)})\end{aligned}$$
%
%
%
%Therefore,  by (\ref{n}), we know that (\ref{n1}) is proved.

%\begin{rem}
%If $e>\max\{r,s\}$, $\rho^2=q^{2a}$, for some $a\in\mathbb Z$, then we can have an alcove geometry
%for the quantum walled Brauer algebra $\mathscr B_{r, t}(\rho,q)$ almost the same as the classical case \cite[Section ~8]{CVDM}.
%The only different is that we use $a$ instead of $\delta$ to define
%the fixed vector $\rho(\delta)$ in \cite[Section ~8]{CVDM}. We can also use the same method as \cite{CV} to determine the decomposition numbers of $\mathscr B_{r, t}(\rho,q)$.
%
%
%\end{rem}

\providecommand{\bysame}{\leavevmode ---\ } \providecommand{\og}{``}
\providecommand{\fg}{''} \providecommand{\smfandname}{and}
\providecommand{\smfedsname}{\'eds.}
\providecommand{\smfedname}{\'ed.}
\providecommand{\smfmastersthesisname}{M\'emoire}
\providecommand{\smfphdthesisname}{Th\`ese}

\newtheorem*{Acknowledgements*}{Acknowledgements}

\begin{Acknowledgements*} The authors are grateful to Hebing Rui for helpful conversations related to this research.

\end{Acknowledgements*}

\end{document}